\documentclass[twoside]{amsart}

\usepackage{amsmath, amsthm, amssymb, mathrsfs}
\usepackage[colorlinks=true, citecolor=blue]{hyperref}
\usepackage{epsfig, graphicx}
\usepackage{verbatim,booktabs}

\newtheorem*{stahl}{Theorem S}

\newtheorem*{kuzmina}{Theorem K}
\newtheorem*{gonchar}{Theorem GR}

\newtheorem{thm}{Theorem}
\newtheorem{prop}[thm]{Proposition}
\newtheorem{cor}[thm]{Corollary}
\newtheorem{lem}[thm]{Lemma}
\newtheorem{df}[thm]{Definition}

\numberwithin{equation}{section}

\newcommand{\T}		{\mathbb{T}}
\newcommand{\D}		{\mathbb{D}}
\newcommand{\Om}	{\mathbb{O}}

\newcommand{\C}		{\mathbb{C}}
\newcommand{\N}		{\mathbb{N}}

\newcommand{\pd}		{\mathcal{G}}
\newcommand{\pk}		{\mathcal{K}}
\newcommand{\K}		{\textnormal{K}}
\newcommand{\E}		{\mathcal{E}}
\newcommand{\alg}		{\mathcal{A}}
\newcommand{\rat}		{\mathcal{R}}

\newcommand{\poly}	{\mathcal{P}}
\newcommand{\mpoly}	{\mathcal{M}}
\newcommand{\pr}		{\textnormal{pr}}

\newcommand{\eba}	{\Psi}
\newcommand{\di}		{\mathscr{D}}
\newcommand{\ged}	{\omega_{(\K,\T)}}
\newcommand{\clos}	{\textnormal{clos}}
\newcommand{\cp}	{\textnormal{cap}}
\newcommand{\im}	{\textnormal{Im}}
\newcommand{\re}	{\textnormal{Re}}
\newcommand{\n}	{\textnormal{\bf n}}
\newcommand{\Lm}	{\Lambda}
\newcommand{\const}	{\textnormal{const.}}
\newcommand{\supp}	{\textnormal{supp}}

\newcommand{\dist}	{\textnormal{dist}}
\newcommand{\diam}	{\textnormal{diam}}
\newcommand{\cws}	{\stackrel{*}{\to}}
\newcommand{\cic}	{\stackrel{\scriptsize\cp}{\rightarrow}}

\begin{document}

\title[Rational Approximants to Algebraic Functions]{Weighted Extremal Domains and Best Rational Approximation}

\author[L. Baratchart]{Laurent Baratchart}
\address{INRIA, Project APICS, 2004 route des Lucioles --- BP 93, 06902 Sophia-Antipolis, France}
\email{laurent.baratchart@sophia.inria.fr}

\author[H. Stahl]{Herbert Stahl}
\address{Beuth Hochschule/FB II; Luxemburger Str. 10; D-13353 Berlin, Germany}
\email{stahl@tfh-berlin.de}

\author[M. Yattselev]{Maxim Yattselev}
\address{Corresponding author, Department of Mathematics, University of Oregon, Eugene, OR, 97403, USA}
\email{maximy@uoregon.edu}

\thanks{The research of first and the third authors was partially supported by the ANR project ``AHPI'' (ANR-07-BLAN-0247-01). The research of the second author has been supported by the Deutsche Forschungsgemeinschaft (AZ: STA 299/13-1).}

\begin{abstract}
Let $f$ be holomorphically continuable over the complex plane except for finitely many branch points contained in the unit disk.  We prove that best rational approximants to $f$ of degree $n$, in the $L^2$-sense on the unit circle, have poles that asymptotically distribute according to the equilibrium measure on the compact set outside of which $f$ is single-valued and which has minimal Green capacity in the disk among all such sets. This provides us with $n$-th root asymptotics of the approximation error. By conformal mapping, we deduce further estimates in approximation by rational or meromorphic functions  to $f$ in the $L^2$-sense on more general Jordan curves encompassing the
branch points. The key to these approximation-theoretic results is a characterization of extremal domains of holomorphy for $f$ in the sense of a weighted logarithmic potential, which is the technical core of the paper.
\end{abstract}

\subjclass[2000]{42C05, 41A20, 41A21}

\keywords{rational approximation, meromorphic approximation, extremal domains, weak asymptotics, non-Hermitian orthogonality.}

\maketitle

\section*{List of Symbols}

\begin{flushleft}
\begin{tabular}{ p{2cm} l }
{\bf Sets}:& \\
$\overline\C$ & extended complex plane\\
$T$ & Jordan curve with exterior domain $O$ and interior domain $G$\\
$\T$ & unit circle with exterior domain $\Om$ and interior domain $\D$\\
$E_f$ & set of the branch points of $f$\\
$K^*$ & reflected set $\{z:1/\bar z\in K\}$\\
$\K$ & set of minimal condenser capacity in $\pk_f(G)$\\
$\Gamma_\nu$ and  $D_\nu$ & minimal set for Problem $(f,\nu)$ and its complement in $\overline\C$ \\
$(\Gamma)_\epsilon$ & $\{z\in\D:\dist(z,\Gamma)<\epsilon\}$\\
$\gamma^u$ & image of a set $\gamma$ under $1/(\cdot - u)$\\
\end{tabular}

\begin{tabular}{p{2cm} l}
{\bf Collections}: &\\ 
$\pk_f(G)$ & admissible sets for $f\in\alg(G)$, $\pk_f=\pk_f(\D)$\\
$\pd_f$ & admissible sets in $\pk_f$ comprised of a finite number of continua\\
$\Lambda(F)$ & probability measures on $F$\\
\end{tabular}

\begin{tabular}{p{2cm} l}
{\bf Spaces}:&\\
$\poly_n$  &  algebraic polynomials of degree at most $n$\\
$\mpoly_n(G)$  & monic algebraic polynomials of degree $n$ with $n$ zeros in $G$, $\mpoly_n=\mpoly_n(\D)$\\
$\rat_n(G)$   & $\rat_n(G):=\poly_{n-1}\mpoly_n^{-1}(G)$, $\rat_n=\rat_n(\D)$\\
$\alg(G)$   & holomorphic functions $\overline\C$ except for branch-type singularities in $G$\\
$L^p(T)$ & classical $L^p$ spaces, $p<\infty$, with respect to arclength on $T$ and the norm $\|\cdot\|_{p,T}$\\ 
$\|\cdot\|_K$ & supremum norm on a set $K$\\
$E^2(G)$ & Smirnov class of holomorphic functions in $G$ with $L^2$ traces on $T$\\
$E_n^2(G)$ & $E_n^2(G):=E^2(G)\mpoly_n^{-1}(G)$\\
$H^2$ & classical Hardy space of holomorphic functions in $\D$ with $L^2$ traces on $\T$\\
$H_n^2$ & $H_n^2:=H^2\mpoly_n^{-1}$\\
\end{tabular}

\begin{tabular}{p{2cm} l}
{\bf Measures}:&\\
$\omega^*$ & reflected measure, $\omega^*(B)=\omega(B^*)$\\
$\widehat\omega$ or $\widetilde\omega$ & balayage of $\omega$, $\supp(\omega)\subset D$, onto $\partial D$\\
$\omega_F$ & equilibrium distribution on $F$\\
$\omega_{F,\psi}$ & weighted equilibrium distribution on $F$ in the field $\psi$\\
$\omega_{(F,E)}$ & Green equilibrium distribution on $F$ relative to $\overline\C\setminus E$\\
\end{tabular}

\begin{tabular}{p{2cm} l}
{\bf Capacities}:&\\
$\cp(K)$ & logarithmic capacity of $K$\\
$\cp_\nu(K)$ & $\nu$-capacity of $K$\\
$\cp(E,F)$ & capacity of the condenser $(E,F)$\\
\end{tabular}

\begin{tabular}{p{2cm} l}
{\bf Energies}:&\\
$I[\omega]$ & logarithmic energy of $\omega$\\
$I_\psi[\omega]$ & weighted logarithmic energy of $\omega$ in the field $\psi$\\
$I_D[\omega]$ & Green energy of $\omega$ relative to $D$\\
$I_\nu[K]$ & $\nu$-energy of a set $K$\\
$\di_D(u,v)$ & Dirichlet integral of functions $u,v$ in a domain $D$\\
\end{tabular}

\begin{tabular}{p{2cm} l}
{\bf Potentials}:&\\ 
$V^\omega$ & logarithmic potential of $\omega$\\
$V_*^\omega$ & spherical logarithmic potential of $\omega$\\
$U^\nu$ & spherically normalized logarithmic potential of $\nu^*$\\
$V_D^\omega$ & Green potential of $\omega$ relative to $D$\\
$g_D(\cdot,u)$ & Green function for $D$ with pole at $u$\\
\end{tabular}

\begin{tabular}{p{2cm} l}
{\bf Constants}: &\\
$c(\psi;F)$ & modified Robin constant,  $c(\psi;F)=I_\psi[\omega_{F,\psi}]-\int\psi d\omega_{F,\psi}$\\
$c(\nu;D)$ & is equal to $\int g_D(z,\infty)d\nu(z)$ if $D$ is unbounded and to $0$ otherwise\\
\end{tabular}

\end{flushleft}

\section{Introduction}
\label{sec:intro}

Approximation theory in the complex domain has undergone striking developments over the last years that gave new impetus to this classical subject. After the solution to the Gonchar conjecture \cite{Par86,Pr93} and the achievement of weak asymptotics in Pad\'e approximation \cite{St85,St86,GRakh87} came the disproof of the Baker-Gammel-Wills conjecture \cite{Lub03,Bus02}, and the Riemann-Hilbert approach to the limiting behavior of orthogonal polynomials \cite{DKMLVZ99a,KMcLVAV04} that opened the way to unprecedented strong  asymptotics in rational  interpolation \cite{ApKVA08,Ap02,BY10} (see \cite{Deift,KamvissisMcLaughlinMiller} for other applications of this powerful device). Meanwhile, the spectral approach to  meromorphic approximation \cite{AAK71}, already instrumental in \cite{Par86}, has produced sharp converse theorems in rational approximation and fueled engineering applications to control systems and signal processing \cite{GL84,Peller,Nikolskii,Partington2}.

In most investigations involved with non-Hermitian orthogonal polynomials and rational interpolation, a central role has been  played by certain geometric extremal problems from logarithmic potential theory, close in spirit to the Lavrentiev type \cite{Kuz80}, that were introduced in \cite{St85b}. On the one hand, their solution produces systems of arcs over which non-Hermitian orthogonal polynomials can be analyzed; on the other hand such polynomials are precisely denominators of rational interpolants to functions that may be expressed as Cauchy integrals over this system of arcs, the interpolation points being chosen in close relation with the latter.

One issue facing now the theory is to extend to \emph{best} rational or meromorphic approximants of prescribed degree to a given function the knowledge that was gained about rational interpolants. Optimality may of course be given various meanings. However, in view of the connections with interpolation theory pointed out in \cite{Lev69,BSW96,BS02}, and granted their relevance to spectral theory, the modeling of signals and systems, as well as inverse problems \cite{Antoulas,FournierLeblond03,HannanDeistler,Nikolskii2,BMSW06,Isakov,Regalia}, it is natural to consider foremost best approximants in Hardy classes. 

The main interest there attaches to the behavior of the poles whose determination is the non-convex and most difficult part of the problem. The first obstacle to value interpolation theory in this context is that it is unclear whether best approximants of a given degree should interpolate the function at enough points, and even if they do these interpolation points are no longer parameters to be chosen adequately in order to produce convergence but rather unknown quantities implicitly determined by the optimality property. The present paper deals with $H^2$-best rational approximation in the complement of the unit disk, for which maximum interpolation is known to take place; it thus remains in this case to locate the interpolation points. This we do asymptotically, when the degree of the approximant goes large, for functions whose singularities consist of finitely many poles and branch points in the disk. More precisely, we prove that the normalized (probability) counting measures of the poles of the approximants converge, in the weak star sense, to the equilibrium distribution of the continuum of minimum Green capacity, in the disk, outside of which the approximated function is single-valued. By conformal mapping, the result carries over to best meromorphic approximants with a prescribed number of poles, in the $L^2$-sense on a Jordan curve encompassing the poles and branch points. We also estimate the approximation error in the $n$-th root sense, that turns out to be the same as in uniform approximation for the functions under consideration. Note that $H^2$-best rational approximants on the disk are of fundamental importance in stochastic identification \cite{HannanDeistler} and that functions with branch points arise naturally in inverse sources and potential problems \cite{BAbHL05,BLM06}, so the result may be regarded as a prototypical case of the above-mentioned program.

The paper is organized as follows. In Sections~\ref{sec:ral2} and~\ref{sec:ra}, we fix the terminology and recall some known facts about $H^2$-best rational approximants and sets of minimal condenser capacity, before stating our main results (Theorems~\ref{thm:L2T} and~\ref{thm:convcap}) along with some corollaries. We set up in Section~\ref{sec:min} a weighted version of the extremal potential problem introduced in \cite{St85b} ({\it cf.} Definition~\ref{df:minset}) and stress its main features. Namely, a solution exists uniquely and can be characterized, among continua outside of which the approximated function is single-valued, as a system of arcs possessing the so-called $S$-property in the field generated by the weight ({\it cf.} Definition~\ref{df:sym} and Theorem~\ref{thm:minset}). Section~\ref{sec:pade} is a brief introduction to multipoint Pad\'e interpolants, of which $H^2$-best rational approximants are a particular case. Section~\ref{sec:proofs} contains the proofs of all the results: first we establish Theorem~\ref{thm:minset}, which is the technical core of the paper, using compactness properties of the Hausdorff metric together with the {\it a priori} geometric estimate of Lemma~\ref{lem:1pr} to prove existence; the $S$-property is obtained by showing the local equivalence of our weighted extremal problem with one of minimal condenser capacity (Lemma~\ref{lem:1sym}); uniqueness then follows from a variational argument using Dirichlet integrals (Lemma~\ref{lem:1e}). After Theorem~\ref{thm:minset} is established, the proof of Theorem~\ref{thm:convcap} is not too difficult. We choose as weight (minus) the potential of a limit point of the normalized counting measures of the interpolation points of the approximants and, since we now know that a compact set of minimal weighted capacity exists and that it possesses the $S$-property, we can  adapt results from \cite{GRakh87} to the effect that the normalized counting measures of the poles of the approximants converge to the weighted equilibrium distribution on this system of arcs. To see that this is nothing but the Green equilibrium distribution, we appeal to the fact that poles and interpolation points are reflected from each other across the unit circle in $H^2$-best  rational approximation. The results carry over to more general domains as in Theorem~\ref{thm:L2T} by a conformal mapping (Theorem~\ref{cor:L2T}). The appendix in Section~\ref{sec:pt} gathers some technical results from logarithmic potential theory that are needed throughout the paper.

\section{Rational Approximation in $L^2$}
\label{sec:ral2}

In this work we are concerned with rational approximation of functions analytic at infinity having multi-valued meromorphic continuation to the entire complex plane deprived of a finite number of points. The approximation will be understood in the $L^2$-norm on a rectifiable Jordan curve encompassing all the singularities of the approximated function. Namely, let $T$ be such a curve. Let further $G$ and $O$ be the interior and exterior domains of $T$, respectively, i.e., the bounded and unbounded components of the complement of $T$ in the extended complex plane~$\overline\C$. We denote by $L^2(T)$ the space of square-summable functions on $T$ endowed with the usual norm
\[
\|f\|_{2,T}^2 := \int_T |f|^2ds,
\]
where $ds$ is the arclength differential. Set $\poly_n$ to be the space of algebraic polynomials of degree at most $n$ and $\mpoly_n(G)$ to be its subset consisting of monic polynomials with $n$ zeros in $G$. Define 
\begin{equation}
\label{eq:rat}
\rat_n(G) := \left\{\frac{p(z)}{q(z)}=\frac{p_{n-1}z^{n-1}+p_{n-2}z^{n-2}+\cdots+p_0}{z^n+q_{n-1}z^{n-1}+\cdots+q_0}:~ p\in\poly_{n-1},~q\in\mpoly_n(G)\right\}.
\end{equation}
That is, $\rat_n(G)$ is the set of rational functions with at most $n$ poles that are holomorphic in some neighborhood of $\overline O$ and vanish at infinity. Let $f$ be a function holomorphic and vanishing at infinity (vanishing at infinity is a normalization required for convenience only). We say that $f$ belongs to the class $\alg(G)$ if
{\it
\begin{itemize}
\item [(i)] $f$ admits holomorphic and single-valued continuation from infinity to an open neighborhood of $\overline O$;
\item[(ii)] $f$ admits meromorphic, possibly multi-valued, continuation along any arc in $\overline G\setminus E_f$ starting from $T$, where $E_f$ is a finite set of points in $G$;
\item[(iii)] $E_f$ is non-empty, the meromorphic continuation of $f$ from infinity has a branch point at each element of $E_f$.
\end{itemize}
}

The primary example of functions in $\alg(G)$ is that of algebraic functions. Every algebraic function $f$ naturally defines a Riemann surface. Fixing a branch of $f$ at infinity is equivalent to selecting a sheet of this covering surface. If all the branch points and poles of $f$ on this sheet lie above $G$, the function $f$ belongs to $\alg(G)$. Other functions in $\alg(G)$ are those of the form $g\circ\log(l_1/l_2)+r$, where $g$ is entire and $l_1,l_2\in\mpoly_m(G)$ while $r\in\rat_k(G)$  for some $m,k\in\N$. However, $\alg(G)$ is defined in such a way that it contains no function in $\rat_n(G)$, $n\in\N$, in order to avoid degenerate cases.

With the above notation, the goal of this section is to describe the asymptotic behavior of
\begin{equation}
\label{eq:L2besterror}
\rho_{n,2}(f,T) := \inf\left\{\|f-r\|_{2,T}:~r\in\rat_n(G)\right\}, \quad 
f\in\alg(G).
\end{equation}
This problem is, in fact, a variation of a classical question in Chebyshev (uniform) rational approximation of holomorphic functions where it is required to describe the asymptotic behavior of
\[
\rho_{n,\infty}(f,T) := \inf\left\{\|f-r\|_T:~r\in\rat_n(G)\right\},
\]
where $\|\cdot\|_T$ is the supremum  norm on $T$. The theory behind Chebyshev approximation is rather well established while its $L^2$-counterpart, which naturally arises in system identification and control theory \cite{B_CMFT99} and serves as a method to approach inverse source problems \cite{BAbHL05,BLM06,BMSW06}, is not so much developed. In particular, it follows from the techniques of rational interpolation devised by Walsh \cite{Walsh} that 
\begin{equation}
\label{eq:limsup}
\limsup_{n\to\infty}\rho_{n,\infty}^{1/n}(f,T) \leq \exp\left\{-\frac{1}{\cp(K,T)}\right\}
\end{equation}
for any function $f$ holomorphic outside of $K\subset G$, where $\cp(K,T)$ is the condenser capacity (Section~\ref{sss:cc}) of a set $K$ contained in a domain $G$ relative to this domain\footnote{In Section~\ref{sec:pt} the authors provide a concise but self-contained account of logarithmic potential theory. The reader may want to consult this section to get accustomed with the employed notation for capacities, energies, potentials, and equilibrium measures.}. On the other hand, it was conjectured by Gonchar and proved by Parf\"enov \cite[Sec. 5]{Par86} on simply connected domains, also later by Prokhorov \cite{Pr93} in full generality, that
\begin{equation}
\label{eq:liminf}
\liminf_{n\to\infty}\rho_{n,\infty}^{1/2n}(f,T) \leq \exp\left\{-\frac{1}{\cp(K,T)}\right\}.
\end{equation}
Notice that only the $n$-th root is taken in \eqref{eq:limsup} while \eqref{eq:liminf} provides asymptotics for the $2n$-th root. Observe also that there are many compacts $K$ which make a given $f\in\alg(G)$ single-valued in their complement. Hence, \eqref{eq:limsup} and \eqref{eq:liminf} can be sharpened by taking the infimum over $K$ on the right-hand side of both inequalities.  To explore this fact we need the following definition.

\begin{df}
\label{df:admiss}
We say that a compact $K\subset G$ is admissible for $f\in\alg(G)$ if $\overline\C\setminus K$ is connected and $f$ has meromorphic and single-valued extension there. The collection of all admissible sets for $f$ we denote by $\pk_f(G)$.
\end{df}

As equations \eqref{eq:limsup} and \eqref{eq:liminf} suggest and Theorem \ref{thm:L2T} below shows, the relevant admissible set in rational approximation to $f\in\alg(G)$ is the set of \emph{minimal condenser capacity} \cite{St85,St85b,St86,St89} relative to $G$:

\begin{df}
\label{df:dmcc}
Let $f\in\alg(G)$. A compact $\K\in\pk_f(G)$ is said to be a set of minimal condenser capacity for $f$ if
\begin{itemize}
\item [(i)] $\cp(\K,T)\leq\cp(K,T)$ for any $K\in\pk_f(G)$; \smallskip
\item [(ii)] $\K\subset K$ for any $K\in\pk_f(G)$ such that $\cp(K,T)=\cp(\K,T)$.
\end{itemize}
\end{df}
It follows from the properties of condenser capacity that $\cp(\K,T)=\cp(T,\K)=\cp(\overline O,\K)$ since $\K$ has connected complement that contains $T$ by Definition~\ref{df:admiss}. In other words, the set $\K$ can be seen as the complement of the ``largest'' (in terms of capacity) domain containing $\overline O$ on which $f$ is single-valued and meromorphic. In fact, this is exactly the point of view taken up in \cite{St85,St85b,St86,St89}. It is known that such a set always exists, is unique, and has, in fact, a rather special structure. To describe it, we need the following definition.

\begin{df}
\label{df:elmin}
We say that a set $K\in\pk_f(G)$ is a smooth cut for $f$ if $K=E_0\cup E_1\cup\bigcup \gamma_j$, where $\bigcup \gamma_j$ is a finite union of open analytic arcs, $E_0\subseteq E_f$ and each point in $E_0$ is the endpoint of exactly one $\gamma_j$, while $E_1$ is a finite set of points each element of which is the endpoint of at least three arcs $\gamma_j$. Moreover, we assume that across each arc $\gamma_j$ the jump of $f$ is not identically zero.
\end{df}

Let us informally explain the motivation behind Definition~\ref{df:elmin}. In order to make $f$ single-valued, it is intuitively clear that one needs to choose a proper system of cuts joining certain points in $E_f$ so that one cannot encircle these points nor access the remaining ones without crossing the cut. It is then plausible that the geometrically ``smallest'' system of cuts comprises of Jordan arcs. In the latter situation, the set $E_1$ consists of the points of intersection of these arcs. Thus, each element of $E_1$ serves as an endpoint for at least three arcs since two arcs meeting at a point are considered to be one. In Definition~\ref{df:elmin} we also impose that the arcs be analytic. It turns out that the set of minimal condenser capacity (Theorem~\hyperref[thm:S]{S}) as well as minimal sets from Section~\ref{sec:min} (Theorem~\ref{thm:minset}) have exactly this structure. It is possible for $E_0$ to be a proper subset of $E_f$. This can happen when some of the branch points of $f$ lie above $G$ but on different sheets of the Riemann surface associated with $f$ that cannot be accessed without crossing the considered system of cuts. 

The following is known about the set $\K$ (Definition~\ref{df:dmcc}) \cite[Thm. 1 and 2]{St85} and \cite[Thm.~1]{St85b}.

\begin{stahl}
\label{thm:S}
Let $f\in\alg(G)$. Then $\K$, the set of minimal condenser capacity for $f$, exists and is unique. Moreover, it is a smooth cut for $f$ and
\begin{equation}
\label{eq:mincondcap}
\frac{\partial}{\partial\n^+}V_{\overline\C\setminus\K}^{\omega_{(T,\K)}} = \frac{\partial}{\partial\n^-} V_{\overline\C\setminus\K}^{\omega_{(T,\K)}} \quad \mbox{on} \quad \bigcup \gamma_j,
\end{equation}
where $\partial/\partial\n^\pm$ are the partial derivatives\footnote{\label{hacross}Since the arcs $\gamma_j$ are analytic and the potential $V_{\overline\C\setminus\K}^{\omega_{(T,\K)}}$ is identically zero on them, $V_{\overline\C\setminus\K}^{\omega_{(T,\K)}}$ can be harmonically continued across each $\gamma_j$ by reflection. Hence, the partial derivatives in \eqref{eq:mincondcap} exist and are continuous.} with respect to the one-sided normals on each $\gamma_j$, $V_{\overline\C\setminus\K}^{\omega_{(T,\K)}}$ is the Green potential of $\omega_{(T,\K)}$ relative to $\overline\C\setminus\K$, and $\omega_{(T,\K)}$ is the Green equilibrium distribution on $T$ relative to $\overline\C\setminus\K$ (Section~\ref{sss:cc}).
\end{stahl}

Note that \eqref{eq:mincondcap} is independent of the orientation chosen on
$\gamma_j$ to define $\partial/\partial\n^\pm$. 
Property \eqref{eq:mincondcap} turns out to be more beneficial than Definition~\ref{df:dmcc} in the sense that all the forthcoming proofs use only \eqref{eq:mincondcap}. However, one does not achieve greater generality by relinquishing the connection to the condenser capacity and considering \eqref{eq:mincondcap} by itself as this property uniquely characterizes $\K$. Indeed, the following theorem is proved in Section~\ref{ss:53}.

\begin{thm}
\label{thm:dmcc}
The set of minimal condenser capacity for $f\in\alg(G)$ is uniquely characterized as a smooth cut for $f$ that satisfies \eqref{eq:mincondcap}.
\end{thm}

With all the necessary definitions at hand, the following result takes place.
\begin{thm}
\label{thm:L2T}
Let $T$ be a rectifiable Jordan curve with interior domain $G$ and exterior domain $O$. If $f\in\alg(G)$, then
\begin{equation}
\label{eq:exactrate}
\lim_{n\to\infty} \rho_{n,2}^{1/2n}(f,T) = \lim_{n\to\infty} \rho_{n,\infty}^{1/2n}(f,T) = \exp\left\{-\frac{1}{\cp(\K,T)}\right\},
\end{equation}
where $\K$ is set of minimal condenser capacity for $f$.
\end{thm}
The second equality in \eqref{eq:exactrate} follows from \cite[Thm $1^\prime$]{GRakh87}, where a larger class of functions than $\alg(G)$  is considered (see Theorem~\hyperref[thm:GR]{GR} in Section \ref{ss:52}). To prove the first equality, we appeal to another type of approximation, namely, \emph{meromorphic} approximation in $L^2$-norm on $T$, for which asymptotics of the error and the poles are obtained below. This type of approximation turns out to be useful in certain inverse source problems \cite{BLM06,LPadRigZgh08,uClLMPap}. Observe that $|T|^{1/p-1/2}\|h\|_{2,T}\leq \|h\|_{p,T}\leq|T|^{1/p}\|h\|_T$ for any $p\in(2,\infty)$ and any bounded function $h$ on $T$ by  H\"older inequality, where $\|\cdot\|_{p,T}$ is the usual $p$-norm on $T$ with respect to $ds$ and $|T|$ is the arclength of $T$. Thus, Theorem~\ref{thm:L2T} implies that \eqref{eq:exactrate} holds for $L^p(T)$-best rational approximants as well when $p\in(2,\infty)$. In fact, as Vilmos Totik pointed out to the authors \cite{Vilmos}, with a different method of proof Theorem~\ref{thm:L2T} can be extended to include the full range $p\in[1,\infty]$.

Just mentioned best meromorphic approximants are defined as follows. Denote by $E^2(G)$ the Smirnov class\footnote{A function $h$ belongs to $E^2(G)$ if $h$ is holomorphic in $G$ and there exists a sequence of rectifiable Jordan curves, say $\{T_n\}$, whose interior domains exhaust $G$, such that $\|h\|_{2,T_n}\leq\const$ independently of $n$.} for $G$ \cite[Sec. 10.1]{Duren}. It is known that functions in $E^2(G)$ have non-tangential boundary values a.e. on $T$ and thus formed traces of functions in $E^2(G)$ belong to $L^2(T)$.  Now, put $E_n^2(G):=E^2(G)\mpoly^{-1}_n(G)$ to be the set of meromorphic functions in $G$ with at most $n$ poles there and square-summable traces on $T$. It is known \cite[Sec. 5]{BMSW06} that for each $n\in\N$ there exists $g_n\in E_n^2(G)$ such that
\[
\|f-g_n\|_{2,T} = \inf\left\{\|f-g\|_{2,T}:~g\in E_n^2(G)\right\}.
\]
That is, $g_n$ is a best meromorphic approximant for $f$ in the $L^2$-norm on $T$.
\begin{thm}
\label{cor:L2T}
Let $T$ be a rectifiable Jordan curve with interior domain $G$ and exterior domain $O$. If $f\in\alg(G)$, then
\begin{equation}
\label{ErrorBestMer}
|f-g_n|^{1/2n} \cic \exp\left\{V_G^{\omega_{(\K,T)}}-\frac{1}{\cp(\K,T)}\right\} \quad \mbox{in} \quad G\setminus\K,
\end{equation}
where the functions $g_n\in E^2_n(G)$ are best meromorphic approximants to $f$ in the $L^2$-norm on $T$, $\K$ is the set of minimal condenser capacity for $f$ in $G$, $\omega_{(\K,T)}$ is the Green equilibrium distribution on $\K$ relative to $G$, and $\cic$ denotes convergence in capacity (see Section~\ref{sss:lc}). Moreover, the counting measures of the poles of $g_n$ converge weak$^*$ to $\omega_{(\K,T)}$.
\end{thm}

\section{$\bar H_0^2$-Rational Approximation}
\label{sec:ra}

To prove Theorems~\ref{thm:L2T} and~\ref{cor:L2T}, we derive a stronger result in the model case where $G$ is the unit disk, $\D$. The strengthening comes from the facts that in this case $L^2$-best meromorphic approximants specialize to $L^2$-best rational approximants the latter also turn out to be interpolants.  In fact, we consider not only best rational approximants but also critical points in rational approximation.

Let $\T$ be the unit circle and set for brevity $L^2:=L^2(\T)$. Denote by $H^2\subset L^2$ the Hardy space of functions whose Fourier coefficients with strictly negative indices are zero. The space $H^2$ can be described as the set of traces of holomorphic functions in the unit disk whose square-means on concentric circles centered at zero are uniformly bounded above\footnote{Each such function has non-tangential boundary values almost everywhere on $\T$ and can be recovered from these boundary values by means of the Cauchy or Poisson integral.} \cite{Duren}. Further, denote by $\bar H_0^2$ the orthogonal complement of $H^2$ in $L^2$, $L^2=H^2\oplus\bar H_0^2$, with respect to the standard scalar product
\[
\langle f,g \rangle := \int_\T f(\tau)\overline{g(\tau)}|d\tau|, \quad f,g\in L^2.
\]
From the viewpoint of analytic function theory, $\bar H_0^2$ can be regarded as a space of traces of functions holomorphic in $\Om:=\overline\C\setminus\overline\D$ and vanishing at infinity whose square-means on the concentric circles centered at zero (this time with radii greater then 1) are uniformly bounded above. In what follows, we denote by $\|\cdot\|_2$ the norm on $L^2$ induced by the scalar product $\langle\cdot,\cdot\rangle$. In fact, $\|\cdot\|_2$ is a norm on $H^2$ and $\bar H_0^2$ as well.

We set $\mpoly_n:=\mpoly_n(\D)$ and $\rat_n:=\rat_n(\D)$. Observe that $\rat_n$ is the set of rational functions of degree at most $n$ belonging to $\bar H^2_0$. With the above notation, consider the following $\bar H_0^2$-rational approximation problem: 

\smallskip
\noindent
{\it Given $f\in\bar H_0^2$ and $n\in\N$, minimize $\|f-r_n\|_2$ over all $r\in\rat_n$.}
\smallskip

\noindent
It is well-known (see \cite[Prop. 3.1]{B06} for the proof and an extensive bibliography on the subject) that this minimum is always attained while any minimizing rational function, also called {\it a best rational approximant} to $f$, lies in $\rat_n\setminus\rat_{n-1}$ unless $f\in\rat_{n-1}$. 

Best rational approximants are part of the larger class of \emph{critical points} in $\bar H_0^2$-rational approximation. From the computational viewpoint, critical points are as important as best approximants  since a numerical search is more likely to yield a locally best rather than a best approximant. For fixed $f\in\bar H^2_0$, critical points can be defined as follows. Set
\begin{equation}
\label{eq:ef}
\begin{array}{rll}
\eba_{f,n}:\poly_{n-1}\times\mpoly_n &\to& [0,\infty) \\
(p,q) &\mapsto& \|f-p/q\|_2^2.
\end{array}
\end{equation}
In other words, $\eba_{f,n}$ is the squared error of approximation of $f$ by $r=p/q$ in $\rat_n$. We topologically identify $\poly_{n-1}\times\mpoly_n$ with an open subset of $\C^{2n}$ with coordinates $p_j$ and $q_k$, $j,k\in\{0,\ldots,n-1\}$ (see \eqref{eq:rat}). Then a pair of polynomials $(p_c,q_c)\in\poly_{n-1}\times\mpoly_n$, identified with a vector in $\C^{2n}$, is said to be a {\it critical pair of order} $n$, if all the partial derivatives of $\eba_{f,n}$ do vanish at $(p_c,q_c)$. Respectively, a rational function $r_c\in\rat_n$ is a {\it critical point of order} $n$ if it can be written as the ratio $r_c=p_c/q_c$ of a critical pair $(p_c,q_c)$ in $\poly_{n-1}\times\mpoly_n$. A particular example of a critical point is a {\it locally best approximant}. That is, a rational function $r_l=p_l/q_l$ associated with a pair $(p_l,q_l)\in\poly_{n-1}\times\mpoly_n$ such that $\eba_{f,n}(p_l,q_l)\leq\eba_{f,n}(p,q)$ for all pairs $(p,q)$ in some neighborhood of $(p_l,q_l)$ in $\poly_{n-1}\times\mpoly_n$. We call a critical point of order $n$ \emph{irreducible} if it belongs to $\rat_n\setminus\rat_{n-1}$. As we have already mentioned, best approximants, as well as local minima, are always irreducible critical points unless $f\in\rat_{n-1}$. In general there may be other critical points, reducible or irreducible, which are  saddles or maxima.  In fact,  to give  amenable  conditions for uniqueness of a critical point it is  a fairly open problem of great practical importance, see  \cite{B_CMFT99,BSW96,BY_RTOPAT10} and the bibliography therein.

One of the most important properties of critical points is the fact that they are ``maximal'' rational interpolants. More precisely, let $f\in\bar H^2_0$ and $r_n$ be an irreducible critical point of order $n$, then \emph{$r_n$ interpolates $f$ at the reflection ($z\mapsto1/\bar z$) of each pole of $r_n$ with order twice the multiplicity that pole} \cite{Lev69}, \cite[Prop. 2]{BY_RTOPAT10}, which is the maximal number of interpolation conditions ({\it i.e., $2n$}) that can be imposed in general on a rational function of type $(n-1,n)$ (i.e., the ratio of a polynomial of degree $n-1$ by a polynomial of degree~$n$).

With all the definitions at hand, we are ready to state our main results concerning the behavior of critical points in $\bar H_0^2$-rational approximation for functions in $\alg(\D)$, which will be proven in Section~\ref{ss:53}.
\begin{thm}
\label{thm:convcap}
Let $f\in\alg(\D)$ and $\{r_n\}_{n\in\N}$ be a sequence of irreducible critical points in $\bar H_0^2$-rational approximation for $f$. Further, let $\K$ be the set of minimal condenser capacity for $f$. Then the normalized counting measures\footnote{The normalized counting measure of poles/zeros of a given function is a probability measure having equal point masses at each pole/zero of the function counting multiplicity.} of the poles of $r_n$ converge weak$^*$ to the Green equilibrium distribution on $\K$ relative to $\D$, $\ged$. Moreover, it holds that
\begin{equation}
\label{eq:Convergence1}
|(f-r_n)|^{1/2n} \cic \exp\left\{-V_{\overline\C\setminus\K}^{\omega^*_{(\K,\T)}}\right\} \quad \mbox{in} \quad \overline\C\setminus(\K\cup \K^*),
\end{equation}
where $\K^*$ and $\omega_{(\K,\T)}^*$ are the reflections\footnote{For every set $K$ we define the reflected set $K^*$ as $K^*:=\{z:~1/\bar z\in K\}$. If $\omega$ is a Borel measure in $\overline\C$, then $\omega^*$ is a measure such that $\omega^*(B)=\omega(B^*)$ for every Borel set $B$.} of $\K$ and $\ged$ across $\T$, respectively, and $\cic$ denotes convergence in capacity. In addition, it holds that
\begin{equation}
\label{eq:Convergence2}
\limsup_{n\to\infty}|(f-r_n)(z)|^{1/2n} \leq \exp\left\{-V_{\overline\C\setminus\K}^{\omega^*_{(\K,\T)}}(z)\right\}
\end{equation}
uniformly for $z\in\overline\Om$.
\end{thm}

Using the fact that the Hardy space $H^2$ is orthogonal to $\bar H_0^2$, one can show that $L^2$-best meromorphic approximants discussed in Theorem~\ref{cor:L2T} specialize to $L^2$-best rational approximants when $G=\D$ (see the proof of Theorem~\ref{cor:L2T}). Moreover, it is shown in Lemma~\ref{lem:pt} in Section~\ref{sec:pt} that $\displaystyle -V_{\overline\C\setminus\K}^{\omega_{(\K,\T)}^*} \equiv V_\D^{\omega_{(\K,\T)}}-1/\cp(\K,\T)$ in $\overline\D$. So, formula \eqref{ErrorBestMer} is, in fact,  a generalization of \eqref{eq:Convergence1}, but only in $G\setminus\K$. Lemma~\ref{lem:pt} also implies that $V_{\overline\C\setminus\K}^{\omega_{(\K,\T)}^*}\equiv1/\cp(\K,\T)$ on $\T$. In particular, the following corollary to Theorem \ref{thm:convcap} can be stated.

\begin{cor}
\label{cor:normconv}
Let $f$, $\{r_n\}$, and $\K$ be as in Theorem \ref{thm:convcap}. Then
\begin{equation}
\label{eq:Convergence3}
\lim_{n\to\infty} \|f-r_n\|_2^{1/2n} = \lim_{n\to\infty}\|f-r_n\|_\T^{1/2n} = \exp\left\{-\frac{1}{\cp(\K,\T)}\right\},
\end{equation}
where $\|\cdot\|_\T$ stands for the supremum norm on $\T$.
\end{cor}
Observe that Corollary~\ref{cor:normconv} strengthens Theorem~\ref{thm:L2T} in the case when $T=\T$. Indeed, \eqref{eq:Convergence3} combined with \eqref{eq:exactrate} implies that the critical points in ${\bar H}_0^2$-rational approximation also provide the best rate of uniform approximation in the $n$-th root sense for $f$ on $\overline\Om$.

\section{Domains of Minimal Weighted Capacity}
\label{sec:min}

Our approach to Theorem~\ref{thm:convcap} lies in exploiting the interpolation properties of the critical points in $\bar H^2_0$-rational approximation. To this end we first study the behavior of rational interpolants with predetermined interpolation points (Theorem~\ref{thm:mpa} in Section~\ref{sec:pade}). However, before we are able to touch upon the subject of rational interpolation proper, we need to identify the corresponding minimal sets. These sets are the main object of investigation in this section.

Let $\nu$ be a probability Borel measure supported in $\overline\D$. We set
\begin{equation}
\label{eq:unu}
U^\nu(z) := -\int\log|1-z\bar u|d\nu(u).
\end{equation}
The function $U^\nu$ is simply the spherically normalized logarithmic potential of $\nu^*$, the reflection of $\nu$ across $\T$ (see \eqref{eq:sphpot}). Hence, it is a harmonic function outside of $\supp(\nu^*)$, in particular, in $\D$. Considering $-U^\nu$ as an {\it external field} acting on non-polar compact subsets of $\D$, we define the weighted capacity in the usual manner (Section~\ref{sss:wc}). Namely, for such a set $K\subset\D$, we define the $\nu$-capacity of $K$ by
\begin{equation}
\label{eq:wcap}
\cp_\nu(K) := \exp\left\{-I_\nu[K]\right\}, \quad I_\nu[K] := \min_\omega\left(I[\omega]-2\int U^\nu d\omega\right),
\end{equation}
where the minimum is taken over all probability Borel measures $\omega$
supported on $K$ (see Section~\ref{sss:lc} for the definition of energy $I[\cdot]$). Clearly, $U^{\delta_0}\equiv0$ and therefore $\cp_{\delta_0}(\cdot)$ is simply the classical logarithmic capacity (Section~\ref{sss:lc}), where $\delta_0$ is the Dirac delta at the origin.

The purpose of this section is to extend results in \cite{St85, St85b} obtained for $\nu=\delta_0$. For that, we introduce a notion of a \emph{minimal set} in a weighted context. This generalization is the key enabling us to adapt the results of \cite{GRakh87} to the present situation, and its study is really the technical core of the paper. For simplicity, we put $\pk_f:=\pk_f(\D)$.

\begin{df}
\label{df:minset}
Let $\nu$ be a probability Borel measure supported in $\overline\D$. A compact $\Gamma_\nu\in\pk_f$, $f\in\alg(\D)$, is said to be a minimal set for Problem $(f,\nu)$ if
\begin{itemize}
\item [(i)] $\cp_\nu(\Gamma_\nu)\leq\cp_\nu(K)$ for any  $K\in\pk_f$; \smallskip
\item [(ii)] $\Gamma_\nu\subset\Gamma$ for any $\Gamma\in\pk_f$ such that $\cp(\Gamma)=\cp(\Gamma_\nu)$.
\end{itemize}
\end{df}

The set $\Gamma_\nu$ will turn out to have geometric properties similar to those of minimal condenser capacity sets (Definition \ref{df:dmcc}). This motivates the following definition.

\begin{df}
\label{df:sym}
A compact $\Gamma\in\pk_f$ is said to be symmetric with respect to a Borel measure $\omega$, $\supp(\omega)\cap\Gamma=\varnothing$, if  $\Gamma$ is a smooth cut for $f$ (Definition \ref{df:elmin}) and
\begin{equation}
\label{eq:GreenPotSym}
\frac{\partial}{\partial\n^+}V_{\overline\C\setminus\Gamma}^\omega = \frac{\partial}{\partial\n^-}V_{\overline\C\setminus\Gamma}^\omega \quad \mbox{on} \quad \bigcup \gamma_j,
\end{equation}
where $\partial/\partial\n^\pm$ are the partial derivatives with respect to the one-sided normals on each side of $\gamma_j$ and $\displaystyle V_{\overline\C\setminus\Gamma}^\omega$ is the Green potential of $\omega$ relative to $\overline\C\setminus\Gamma$.
\end{df}

Definition \ref{df:sym} is given in the spirit of \cite{St85b} and thus appears to be different from the \emph{S-property} defined in \cite{GRakh87}. Namely, a compact $\Gamma\subset\D$ having the structure of a smooth cut is said to possess the S-property in the field $\psi$, assumed to be harmonic in some neighborhood of $\Gamma$, if
\begin{equation}
\label{eq:sproperty}
\frac{\partial (V^{\omega_{\Gamma,\psi}}+\psi)}{\partial\n^+} = \frac{\partial (V^{\omega_{\Gamma,\psi}}+\psi)}{\partial\n^-}, \quad \mbox{q. e. on} \quad \bigcup \gamma_j,
\end{equation}
where $\omega_{\Gamma,\psi}$ is the weighted equilibrium distribution on $\Gamma$ in the field $\psi$ and the normal derivatives exist at every tame point of
$\supp(\omega_{\Gamma,\psi})$ (see Section \ref{ss:52}). It follows from \eqref{eq:whnustar} and \eqref{eq:toRemind1} that $\Gamma$ has the S-property in the field $-U^\nu$ if and only if it is symmetric with respect to $\nu^*$,
taking into account that $V^{\omega_{\Gamma,-U^\nu}}-U^\nu$ is constant on the arcs $\gamma_j$ which are regular (see Section~\ref{dubalai}) hence the normal derivatives exist at every point.
This reconciles Definition \ref{df:sym} with the one given in \cite{GRakh87} in the setting of our work.

The symmetry property \eqref{eq:GreenPotSym} entails that $V_{\overline\C\setminus\Gamma}^\omega$ has a very special structure.

\begin{prop}
\label{prop:minset}
Let $\Gamma = E_0\cup E_1\cup\bigcup\gamma_j$ and $V_{\overline\C\setminus\Gamma}^\omega$ be as in Definitions~\ref{df:elmin} and~\ref{df:sym}. Then the arcs $\gamma_j$ possess definite tangents at their endpoints. The tangents 
to the arcs ending at $e\in E_1$ (there are at least three by definition of a
smooth cut) are equiangular. 
Further, set
\begin{equation}
\label{eq:funh}
 H_{\omega,\Gamma} := \partial_z V_{\overline\C\setminus\Gamma}^\omega, \quad \partial_z := ( \partial_x-i\partial_y)/2.
\end{equation}
Then $H_{\omega,\Gamma}$ is holomorphic in $\overline\C\setminus(\Gamma\cup\supp(\omega))$ and has continuous boundary values from each side of every $\gamma_j$ that satisfy $H_{\omega,\Gamma}^+ = -H_{\omega,\Gamma}^-$ on each $\gamma_j$. Moreover, $H_{\omega,\Gamma}^2$ is a meromorphic function  in 
$\overline{\C}\setminus\mbox{supp}(\omega)$ that has a simple pole at each element of $E_0$ and a zero at each element $e$ of $E_1$ whose order is equal to the number of arcs $\gamma_j$ having $e$ as endpoint minus 2. 
\end{prop}

The following theorem is the main result of this section and is a weighted generalization of \cite[Thm. 1 and 2]{St85} and \cite[Thm.~1]{St85b} for functions in $\alg(\D)$.

\begin{thm}
\label{thm:minset}
Let $f\in\alg(\D)$ and $\nu$ be a probability Borel measure supported in $\overline\D$. Then a minimal set for Problem $(f,\nu)$, say $\Gamma_\nu$, exists, is unique and contained in $\overline\D_r$, $r:=\max_{z\in E_f}|z|$. Moreover, $\Gamma\in\pk_f$ is minimal if and only if it is symmetric with respect to $\nu^*$.
\end{thm}

The proof of Theorem~\ref{thm:minset} is carried out in Section~\ref{ss:51} and the proof of Proposition~\ref{prop:minset} is presented in Section~\ref{ss:prop}.

\section{Multipoint Pad\'e Approximation}
\label{sec:pade}

In this section, we state a result that yields complete information on the $n$-th root behavior of rational interpolants to functions in $\alg(\D)$. It is essentially a consequence both of Theorem~\ref{thm:minset} and Theorem~4 in \cite{GRakh87} on the behavior of multipoint Pad\'e approximants to functions analytic off a symmetric contour, whose proof plays here an essential role.

Classically, diagonal multipoint Pad\'e approximants to $f$ are rational functions of type $(n,n)$ that interpolate $f$ at a prescribed system of $2n+1$ points. However, when the approximated function is holomorphic at infinity, as is the case $f\in\alg(\D)$, it is customary to place at least one interpolation point there. More precisely, let $\E=\{E_n\}$ be a triangular scheme of points in $\overline\C\setminus E_f$ and let $v_n$ be the monic polynomial with zeros at the finite points of $E_n$. In other words, $\E:=\{E_n\}_{n\in\N}$ is such that each $E_n$ consists of $2n$ not necessarily distinct nor finite points contained in $\overline\C\setminus E_f$. 

\begin{df}
\label{df:pade}
Given $f\in\alg(\D)$ and a triangular scheme $\E$, the $n$-th diagonal Pad\'e approximant to $f$ associated with $\E$ is the unique rational function $\Pi_n=p_n/q_n$ satisfying:
\begin{itemize}
\item $\deg p_n\leq n$, $\deg q_n\leq n$, and $q_n\not\equiv0$; \smallskip
\item $\left(q_n(z)f(z)-p_n(z)\right)/v_n(z)$ has analytic (multi-valued) extension to $\overline\C\setminus E_f$; \smallskip
\item $\left(q_n(z)f(z)-p_n(z)\right)/v_n(z)=O\left(1/z^{n+1}\right)$ as $z\to\infty$.
\end{itemize}
\end{df}

Multipoint Pad\'e approximants always exist since the conditions for $p_n$ and $q_n$ amount to solving a system of $2n+1$ homogeneous linear equations with $2n+2$ unknown coefficients, no solution of which can be such that $q_n\equiv0$ (we may thus assume that $q_n$ is monic); note that the required interpolation at infinity is entailed by the last condition and therefore $\Pi_n$ is, in fact, of type $(n-1,n)$.

We define {\it the support of} $\E$ as $\supp(\E):=\cap_{n\in\N}\overline{\cup_{k\geq n}E_k}$. Clearly, $\supp(\E)$ contains the support of any weak$^*$ limit point of the normalized counting measures of points in $E_n$ (see Section~\ref{sss:wsccic}). We say that a Borel measure $\omega$ is the \emph{asymptotic distribution for $\E$} if the normalized counting measures of points in $E_n$ converge to $\omega$ in the weak$^*$ sense.

\begin{thm}
\label{thm:mpa}
Let $f\in\alg(\D)$ and $\nu$ be a probability Borel measure supported in $\overline\D$. Further, let $\E$ be a triangular scheme of points, $\supp(\E)\subset\overline\Om$, with asymptotic distribution $\nu^*$. Then 
\begin{equation}
\label{eq:mpa}
\displaystyle |f-\Pi_n|^{1/2n}\cic\exp\left\{-V_{D_\nu}^{\nu^*}\right\} \quad \mbox{in} \quad D_\nu\setminus \supp(\nu^*), \quad D_\nu =\overline\C\setminus\Gamma_\nu,
\end{equation}
where $\Pi_n$ are the diagonal Pad\'e approximants to $f$ associated with $\E$ and $\Gamma_\nu$ is the minimal set for Problem $(f,\nu)$. It also holds that the normalized counting measures of poles of  $\Pi_n$ converge weak$^*$ to $\widehat\nu^*$, the balayage (Section~\ref{ss:balayage}) of $\nu^*$ onto $\Gamma_\nu$ relative to $D_\nu$. In particular, the poles of $\Pi_n$ tend to $\Gamma_\nu$ in full proportion.
\end{thm}

\section{Proofs}
\label{sec:proofs}

\subsection{Proof of Theorem \ref{thm:minset}}
\label{ss:51}

In this section we prove Theorem \ref{thm:minset} in several steps that are organized as separate lemmas.

Denote by $\pd_f$ the subset of $\pk_f$ comprised of those admissible sets that are unions of a finite number of disjoint continua each of which contains at least two point of $E_f$. In particular, each member of $\pd_f$ is a regular set \cite[Thm. 4.2.1]{Ransford} and $\cp(\Gamma_1\setminus(\Gamma_1\cap\Gamma_2))>0$ when $\Gamma_1\neq\Gamma_2$, $\Gamma_1,\Gamma_2\in\pd_f$ (if $\Gamma_1\neq\Gamma_2$, there exists a continuum $\gamma\subset\Gamma_1\setminus(\Gamma_1\cap\Gamma_2)$; as any continuum has positive capacity \cite[Thm. 5.3.2]{Ransford}, the claim follows). Considering $\pd_f$ instead of $\pk_f$ makes the forthcoming analysis simpler but does not alter the original problem as the following lemma shows.

\begin{lem}
\label{lem:1haus}
It holds that $\displaystyle \inf_{\Gamma\in\pd_f}\cp_\nu(\Gamma) = \inf_{K\in\pk_f}\cp_\nu(K)$.
\end{lem}
\begin{proof}
Pick $K\in\pk_f$ and let  $\mathcal{O}$ be the collection of all domains containing $\overline{\C}\setminus K$ to which $f$ extends meromorphically. The set  $\mathcal{O}$ is nonempty as it contains $\overline{\C}\setminus K$, it is partially ordered by inclusion, and any totally ordered subset $\{O_\alpha\}$ has an upper bound, {\it e.g.} $\cup_\alpha O_\alpha$. Therefore, by Zorn's lemma \cite[App. 2, Cor.2.5]{Lang}, $\mathcal{O}$ has a maximal element,
say  $O$. 

Put $F=\overline{\C}\setminus O$. With a slight abuse of notation, we still denote by $f$ the meromorphic continuation of the latter to $\overline{\C}\setminus F$. Note that a point in $E_f$ is either ``inactive'' (i.e., is not a 
branch point for that branch of $f$ that we consider over $\overline{\C}\setminus F$) or belongs to $F$.

If $F$ is not connected, there are two bounded disjoint open sets 
$V_1$, $V_2$ such that $(V_1\cup V_2)\cap F=F$ and,
for $j=1,2$, 
$\partial V_j\cap F=\varnothing$, $V_j\cap F\neq\varnothing$. 
If $V_j$ contains only one connected component of $F$,
we do not refine it further. Otherwise,
there are two disjoint open sets $V_{j,1}, V_{j,2}
\subset V_{j}$ such that $(V_{j,1}\cup V_{j,2})\cap F=V_{j}\cap F$ and,
for $k=1,2$, 
$\partial V_{j,k}\cap F=\varnothing$, $V_{j,k}\cap F\neq\varnothing$.
Iterating this process, we obtain successive generations  
of bounded finite disjoint open covers of $F$, each element of which 
contains at least one connected component of $F$ and has boundary that does 
not meet $F$. The process stops if
$F$ has finitely many components, and then 
the resulting open sets separate them.
Otherwise the process can continue indefinitely and, if 
$C_1,\ldots,C_N$ are the finitely many connected components of $F$ that
meet $E_f$, at least one open set of the $N+1$-st generation  contains no 
$C_j$.  In any case, if $F$ has more than $N$ connected components,
there is a bounded open set $V$, containing at least one connected component
of $F$ and no point of $E_f\cap F$, such that  $\partial V\cap F=\varnothing$.

Let $A$ be the unbounded connected component of $\overline{\C}\setminus V$ and $A_1,\ldots,A_L$ those bounded components of $\overline{\C}\setminus V$, if any, that contain some $C_j$ (if $L=0$ this is the empty collection). 
Since $O=\overline{\C}\setminus F$ is connected, each $\partial A_\ell$ can be connected to $\partial A$ by a closed arc $\gamma_{\ell}\subset O$. Then $W:=V\setminus \cup_{\ell}\gamma_{\ell}$ is open with $\partial W\cap F=\varnothing$, it contains at least one connected component of $F$, and no bounded component of its complement meets $E_f\cap F$. Let $X$ be the unbounded connected component of $\overline{\C}\setminus W$ and put $U:=\overline{\C}\setminus X$. The set $U$ is open, simply connected, and $\partial U\subset \partial W$ is compact and does not meet $F$. Moreover, since it is equal to the union of $W$ and all the bounded components of $\overline{\C}\setminus W$, $U$ does not meet $E_f\cap F$.

Now, $f$ is defined and meromorphic in a neighborhood of $\partial U\subset O$, and meromorphically  continuable along any path in $U$ since the latter contains no point of $E_f\cap F$. Since $U$ is simply connected, $f$ extends meromorphically to $O\cup U$ by the monodromy theorem. However the latter set is a domain which strictly contains
$O$ since $U$ contains $W$ and thus at least one connected component of $F$. This contradicts the maximality of $O$ and shows that $F$ consists precisely of $N$ connected components, namely $C_1,\ldots,C_N$. Moreover, if $\Gamma_j$ is a Jordan curve encompassing $C_j$ and no other $C_\ell$, then by what precedes $f$ must be single-valued  along $\Gamma_j$ which is impossible if $C_j\cap E_f$ is a single point by property (iii) in the definition of the 
class $\alg(G)$. Therefore $F\in \pd_f$ and since $F\subset K$ it holds that  $\cp_\nu (F)\leq\cp_\nu(K)$. This achieves the proof.
\end{proof}

For any $\Gamma\in\pd_f$ and $\epsilon>0$, set $(\Gamma)_\epsilon:=\{z\in\D:\dist(z,\Gamma)<\epsilon\}$.  We endow $\pd_f$ with the Hausdorff metric, i.e.,
\[
d_H(\Gamma_1,\Gamma_2) := \inf\{\epsilon:\Gamma_1\subset(\Gamma_2)_\epsilon,\Gamma_2\subset(\Gamma_1)_\epsilon\}.
\]
By standard properties of the Hausdorff distance \cite[Sec. 3.16]{Dieudonne}, $\clos_{d_H}(\pd_f)$, the closure of $\pd_f$ in the $d_H$-metric, is a compact metric space. Observe that taking $d_H$-limit cannot increase the number of connected components since any two components of the limit set have disjoint $\epsilon$-neighborhoods. That is, the $d_H$-limit of a sequence of compact sets having less than $N$ connected components has in turn less than $N$ connected components. Moreover, each component of the $d_H$-limit of a sequence of compact sets $E_n$ is the $d_H$-limit of a sequence of unions of components from $E_n$.  Thus, each element of $\clos_{d_H}(\pd_f)$ still consists of a finite number of continua each containing at least two points from $E_f$ but possibly with multiply connected complement. However, the polynomial convex hull of such a set, that is, the union of the set with the bounded components of its complement, again belongs to $\pd_f$ unless the set touches $\T$.

\begin{lem}
\label{lem:1cont}
Let $\pd\subset\pd_f$ be such that each element of $\clos_{d_H}(\pd)$ is contained in $\D$. Then the functional $I_\nu[\cdot]$ is finite and continuous on $\clos_{d_H}(\pd)$.
\end{lem}
\begin{proof}
Let $\Gamma_0\in\clos_{d_H}(\pd)$ be fixed. Set $\epsilon_0:=\dist(\Gamma_0,\T)/4>0$ and define
\begin{equation}
\label{eq:nbhd}
\mathcal{N}_{\epsilon_0}(\Gamma_0) := \left\{\Gamma\in\clos_{d_H}(\pd):~d_H(\Gamma_0,\Gamma)<\epsilon_0\right\}.
\end{equation}
Then it holds that $\dist((\Gamma)_\epsilon,\T) \geq 2\epsilon_0$ for any $\Gamma\in\mathcal{N}_{\epsilon_0}(\Gamma_0)$ and $\epsilon\leq\epsilon_0$. Thus, the closure of each such $(\Gamma)_\epsilon$ is at least $\epsilon_0$ away from $\T_{1-\epsilon_0}$.

Let $\Gamma\in\mathcal{N}_{\epsilon_0}(\Gamma_0)$ and set $\epsilon:=d_H(\Gamma_0,\Gamma)$. Denote by $D_0$ and $D$ the unbounded components of the complements of $\Gamma_0$ and $\Gamma$, respectively. It follows from \eqref{eq:weightgreen} that $I_\nu[\Gamma_0]$ is finite and that
\[
I_\nu[\Gamma] - I_\nu[\Gamma_0] = \iint \left(g_D(z,u)-g_{D_0}(z,u)\right)d\widetilde\nu^*(u)d\widetilde\nu^*(z),
\]
where $\widetilde\nu^*$ is the balayage of $\nu^*$ onto $\T_{1-\epsilon_0}$. Since $\Gamma\subset\overline{(\Gamma_0)_\epsilon}$ and $\Gamma_0\subset\overline{(\Gamma)_\epsilon}$, $g_D(\cdot,u)-g_{D_0}(\cdot,u)$ is a harmonic function in $G:=\overline\C\setminus(\overline{(\Gamma)_\epsilon}\cap\overline{(\Gamma_0)_\epsilon})$ for each $u\in G$ by the first claim in Section~\ref{ss:gp} (recall that we agreed to continue $g_{D_0}(\cdot,u)$ and $g_D(\cdot,u)$ by zero outside of the closures of $D_0$ and $D$, respectively). Thus, since Green functions are non-negative, we get from the maximum principle for harmonic functions and the fact that $\widetilde\nu^*$ is a unit measure that
\begin{eqnarray}
\left|I_\nu[\Gamma] - I_\nu[\Gamma_0]\right| &\leq& \max_{u\in\T_{1-\epsilon_0}} \max_{z\in\partial G}|g_D(z,u)-g_{D_0}(z,u)| \nonumber \\
\label{eq:gfbound}
{} &<& \max_{u\in\T_{1-\epsilon_0}}\left(\max_{z\in\partial(\Gamma)_\epsilon}g_D(z,u) + \max_{z\in\partial(\Gamma_0)_\epsilon}g_{D_0}(z,u)\right).
\end{eqnarray}

Let $\gamma$ be any connected component of $\Gamma$ and $G_\gamma$ be the unbounded component of its complement. Observe that $(\Gamma)_\epsilon=\cup_\gamma(\gamma)_\epsilon$, where the union is taken over the (finitely many) components of $\Gamma$. Since $D\subset G_\gamma$, we get that
\begin{equation}
\label{eq:twogreen}
g_D(z,u) \leq g_{G_\gamma}(z,u)
\end{equation}
for any $u\in D$ and $z\in G_\gamma\setminus{u}$ by the maximum principle. 

Set $\delta:=\sqrt{2\epsilon/\cp(\gamma)}$ and $L$ to be the $\log(1+\delta)$-level line of $g_{G_\gamma}(\cdot,\infty)$. As $G_\gamma$ is simply connected, $L$ is a smooth Jordan curve.\footnote{By conformal invariance of Green functions it is enough to check it for $G_\gamma=\Om$ in which case it is obvious.} Since $\gamma$ is a continuum, it is well-known that $\cp(\gamma)\geq\diam(\gamma)/4$ \cite[Thm. 5.3.2]{Ransford}. Recall also that  $\gamma$ contains at least two points from $E_f$. Thus, $\diam(\gamma)$ is bounded from below by the minimal distance between the algebraic singularities of $f$. Hence, we can assume without loss of generality that $\delta\leq1$.  We claim that $\dist(\gamma,L)\geq\epsilon$ and postpone the proof of this claim until the end of this lemma. The claim immediately implies that $(\gamma)_\epsilon$ is contained in the bounded component of the complement of $L$ and that
\begin{equation}
\label{eq:maxgreen}
\max_{z\in\partial(\gamma)_\epsilon} g_{G_\gamma}(z,\infty)\leq\log(1+\delta) \leq \delta.
\end{equation}

It follows from the conformal invariance of the Green function \cite[Thm. 4.4.2]{Ransford} and can be readily verified using the characteristic properties that $g_{G_\gamma}(z,u)=g_{G_\gamma^u}(1/(z-u),\infty)$, where $G_\gamma^u$ is the image of $G_\gamma$ under the map $1/(\cdot-u)$. It is also simple to compute that
\begin{equation}
\label{eq:neweps}
\dist(\gamma^u,\partial(\gamma)_\epsilon^u) \leq \frac{\epsilon}{\dist(u,\gamma)\dist(u,\partial(\gamma)_\epsilon)} \leq\frac{\epsilon}{\epsilon_0^2}, \quad u\in\T_{1-\epsilon_0},
\end{equation}
by the remark after \eqref{eq:nbhd}, where $\gamma^u$ and $(\gamma)_\epsilon^u$ have obvious meaning. So, combining \eqref{eq:neweps} with \eqref{eq:maxgreen} applied to $\gamma^u$, we deduce that
\begin{equation}
\label{eq:maxgreenu}
\max_{z\in\partial(\gamma)_\epsilon}g_{G_\gamma}(z,u)  = \max_{z\in\partial(\gamma)_\epsilon^u}g_{G_\gamma^u}(z,\infty) \leq \max_{z\in\partial(\gamma^u)_{\epsilon/\epsilon_0^2}}g_{G_\gamma^u}(z,\infty) \leq \delta_u, \quad u\in\T_{1-\epsilon_0}
,\end{equation}
where we put $\delta_u:=\sqrt{2\epsilon/\epsilon_0^2\cp(\gamma^u)}$. 

As we already mentioned, $\cp(\gamma)\geq\diam(\gamma)/4$. Hence, it holds that
\begin{equation}
\label{eq:belowcap}
\min_{u\in\T_{1-\epsilon_0}}\cp(\gamma^u) \geq \frac14\min_{u\in\T_{1-\epsilon_0}}\max_{z,w\in\gamma}\left|\frac{1}{z-u}-\frac{1}{w-u}\right| \geq \frac{\diam(\gamma)}{16}.
\end{equation}
Gathering together \eqref{eq:twogreen}, \eqref{eq:maxgreenu}, and \eqref{eq:belowcap}, we derive that
\[
\max_{u\in\T_{1-\epsilon_0}}\max_{z\in\partial(\Gamma)_\epsilon}g_D(z,u) \leq \max_\gamma \frac{4}{\epsilon_0}\sqrt{\frac{2\epsilon}{\diam(\gamma)}},
\]
where $\gamma$ ranges over all components of $\Gamma$. Recall that each component of $\Gamma$ contains at least two points from $E_f$. Thus, $1/\diam(\gamma)$ is bounded above by a constant that depends only on $f$.

Arguing in a similar fashion for $\Gamma_0$, we obtain from \eqref{eq:gfbound} that
\[
|I_\nu[\Gamma] - I_\nu[\Gamma_0]| \leq \frac{\const}{\epsilon_0} \sqrt{d_H(\Gamma,\Gamma_0)} \quad \mbox{for any} \quad \Gamma\in\mathcal{N}_{\epsilon_0}(\Gamma_0),
\]
where $\const$ is a constant depending only on $f$. This finishes the proof of the lemma granted we prove the claim made before \eqref{eq:maxgreen}.

It was claimed that for a continuum $\gamma$ and the $\log(1+\delta)$-level line $L$ of $g_{G_\gamma}(\cdot,\infty)$, $\delta\leq1$, it holds that
\begin{equation}
\label{eq:rakhper}
\dist(\gamma,L) \geq \frac{\delta^2\cp(\gamma)}{2},
\end{equation}
where $G_\gamma$ is the unbounded component of the complement of $\gamma$. Inequality \eqref{eq:rakhper} was proved in \cite[Lem. 1]{uPerevRakh}, however, this work was never published and the authors felt compelled to reproduce this lemma here.

Let $\Phi$ be a conformal map of $\Om$ onto $G_\gamma$, $\Phi(\infty)=\infty$. It is well-known that $|\Phi(z)z^{-1}|\to\cp(\gamma)$ as $z\to\infty$ and that $g_{G_\gamma}(\cdot,\infty)=\log|\Phi^{-1}|$, where $\Phi^{-1}$ is the inverse of $\Phi$ (that is, a conformal map of $G_\gamma$ onto $\Om$, $\Phi^{-1}(\infty)=\infty$). Then it follows from \cite[Thm. IV.2.1]{Goluzin} that
\begin{equation}
\label{eq:goluzin}
|\Phi^\prime(z)| \geq \cp(\gamma)\left(1-\frac{1}{|z|^2}\right), \quad z\in\Om.
\end{equation}
Let $z_1\in\gamma$ and $z_2\in L$ be such that $\dist(\gamma,L)=|z_1-z_2|$. Denote by $[z_1,z_2]$ the segment joining $z_1$ and $z_2$.  Observe that $\Phi^{-1}$ maps the annular domain bounded by $\gamma$ and $L$ onto the annulus $\{z:1<|z|<1+\delta\}$. Denote by $S$ the intersection of $\Phi^{-1}((z_1,z_2))$ with this annulus. Clearly, the angular projection of $S$ onto the real line is equal to $(1,1+\delta)$. Then
\begin{eqnarray}
\dist(\gamma,L) &=& \int_{(z_1,z_2)}|dz| = \int_{\Phi^{-1}((z_1,z_2))}|\Phi^\prime(z)||dz| \geq \cp(\gamma)\int_{\Phi^{-1}((z_1,z_2))}\left(1-\frac{1}{|z|^2}\right)|dz| \nonumber \\
{} &\geq& \cp(\gamma)\int_{S}\left(1-\frac{1}{|z|^2}\right)|dz| \geq \cp(\gamma)\int_{(1,1+\delta)}\left(1-\frac{1}{|z|^2}\right)|dz| = \frac{\delta^2\cp(\gamma)}{1+\delta}, \nonumber
\end{eqnarray}
where we used \eqref{eq:goluzin}. This proves \eqref{eq:rakhper} since it is assumed that $\delta\leq1$.
\end{proof}

Set $\pr_\rho(\cdot)$ to be the radial projection onto $\overline\D_\rho$, i.e., $\pr_\rho(z)=z$ if $|z|\leq\rho$ and $\pr_\rho(z)=\rho z/|z|$ if $\rho<|z|<\infty$. Put further $\pr_\rho(K):=\{\pr_\rho(z):z\in K\}$. In the following lemma we show that $\pr_\rho$ can only increase the value of $I_\nu[\cdot]$. 

\begin{lem}
\label{lem:1pr}
Let $\Gamma\in\pd_f$ and $\rho\in[r,1)$, $r=\max_{z\in E_f}|z|$. Then $\pr_\rho(\Gamma)\in\pd_f$ and $\cp_\nu(\pr_\rho(\Gamma))\leq\cp_\nu(\Gamma)$.
\end{lem}
\begin{proof} As $E_f\subset\overline\D_r$, $f$ naturally extends along any ray $t\xi$, $\xi\in\T$, $t\in(r,\infty)$. Thus, the germ $f$ has a representative which is single-valued and meromorphic outside of $\pr_\rho(\Gamma)$. It is also true that $\pr_\rho$ is a continuous map on $\C$ and therefore cannot disconnect the components of $\Gamma$ although it may merge some of them. Thus, $\pr_\rho(\Gamma)\in\pd_f$.

Set $w=\exp\{U^\nu\}$ and 
\[
\delta_m^w(\Gamma) := \sup_{z_1,\ldots,z_m\in\Gamma}\left[\prod_{1\leq j<i\leq m}|z_i-z_j|w(z_i)w(z_j)\right]^{2/m(m-1)}.
\]
It is known \cite[Thm. III.1.3]{SaffTotik} that $\delta_m^w(\Gamma)\to\cp_\nu(\Gamma)$ as $m\to\infty$. Thus, it is enough to obtain that $\delta_m^w(\pr_\rho(\Gamma))\leq\delta_m^w(\Gamma)$ holds for any $m$. In turn, it is sufficient to show that
\begin{equation}
\label{eq:leq}
|\pr_\rho(z_1)-\pr_\rho(z_2)|w(\pr_\rho(z_1))w(\pr_\rho(z_2)) \leq |z_1-z_2|w(z_1)w(z_2)
\end{equation}
for any $z_1,z_2\in\D$. 

Assume for the moment that $\nu=\delta_u$ for some $u\in\overline\D$, i.e., $w(z)=1/|1-z\bar u|$. It can be readily seen that it is enough to consider only two cases: $|z_1|\leq\rho$, $|z_2|=x>\rho$ and $|z_1|=|z_2|=x>\rho$. In the former situation, \eqref{eq:leq} will follow upon showing that
\[
l_1(x) := \frac{x^2+|z_1|^2-2x|z_1|\cos\phi}{1+x^2|u|^2-2x|u|\cos\psi}
\]
is an increasing function on $(|z_1|,1/|u|)$ for any choice of $\phi$ and $\psi$. Since
\begin{eqnarray}
l_1^\prime(x) &=& 2\frac{x(1-|u|^2|z_1|^2)-|z_1|\cos\phi(1-x^2|u|^2)-|u|\cos\psi(x^2-|z_1|^2)}{(1+x^2|u|^2-2x|u|\cos\psi)^2} \nonumber \\
{} &>& 2\frac{(1-|u||z_1|)(1-x|u|)(x-|z_1|)}{(1+x|u|)^4} > 0, \nonumber
\end{eqnarray}
$l_1$ is indeed strictly increasing on $(|z_1|,1/|u|)$. In the latter case, \eqref{eq:leq} is equivalent to showing that
\[
l_2(x) := (1/x+x|u|^2-2|u|\cos\phi)(1/x+x|u|^2-2|u|\cos\psi)
\]
is a decreasing function on $(\rho,1/|u|)$ for any choice of $\phi$ and $\psi$. This is true since
\[
l_2^\prime(x) = 2(|u|^2-1/x^2)(1/x+x|u|^2-|u|(\cos\phi+\cos\psi)) < 0.
\]
Thus, we verified \eqref{eq:leq} for $\nu=\delta_u$. 

In the general case it holds that
\[
|z_1-z_2|w(z_1)w(z_2) = \exp\left\{\int\log\frac{|z_1-z_2|}{|1-z_1\bar u||1-z_2\bar u|}d\nu(u)\right\}.
\]
As the kernel on the right-hand side of the equality above gets smaller when $z_j$ is replaced by $\pr_\rho(z_j)$, $j=1,2$, by what precedes, the validity of \eqref{eq:leq} follows.
\end{proof}

Combining Lemmas \ref{lem:1haus}--\ref{lem:1pr}, we obtain the existence of minimal sets.

\begin{lem}
\label{lem:1b}
A minimal set $\Gamma_\nu$ exists and is contained in $\overline\D_r$, $r=\max\{|z|:~z\in E_f\}$.
\end{lem}
\begin{proof}
By Lemma \ref{lem:1haus}, it is enough to consider only the sets in $\pd_f$. Let $\{\Gamma_n\}\subset\pd_f$ be a maximizing sequence for $I_\nu[\cdot]$ (minimizing sequence for the $\nu$-capacity), that is, $I_\nu[\Gamma_n]$ tends to  $\sup_{\Gamma\in\pd_f}I_\nu[\Gamma]$ as $n\to\infty$. Then it follows from Lemma \ref{lem:1pr} that $\{\pr_r(\Gamma_n)\}$ is another maximizing sequence for $I_\nu[\cdot]$ in $\pd_f$, and $\pr_r(\Gamma_n)\in\pd_r:=\{\Gamma\in\pd_f:~\Gamma\subseteq\overline\D_r\}$. As $\clos_{d_H}(\pd_r)$ is a compact metric space, there exists at least one limit point of $\{\pr_r(\Gamma_n)\}$ in $\clos_{d_H}(\pd_r)$, say $\Gamma_0$, and $\Gamma_0\subset\overline\D_r$. Since $I_\nu[\cdot]$ is continuous on $\clos_{d_H}(\pd_r)$ by Lemma \ref{lem:1cont}, $I_\nu[\Gamma_0]=\sup_{\Gamma\in\pd_f}I_\nu[\Gamma]$. Finally, as the polynomial convex hull of $\Gamma_0$, say $\Gamma_0^\prime$, belongs to $\pd_f$ and since $I_\nu[\Gamma_0]=I_\nu[\Gamma_0^\prime]$ (see Section~\ref{sss:wcf}), we may put $\Gamma_\nu=\Gamma_0^\prime$.
\end{proof}

To continue with our analysis we need the following theorem \cite[Thm. 3.1]{Kuz80}. It describes the continuum of minimal condenser capacity connecting finitely many given points as a union of closures of the \emph{non-closed negative critical trajectories} of a quadratic differential. Recall that a negative trajectory of the quadratic differential $q(z)dz^2$ is a maximally continued arc along which $q(z)dz^2<0$; the trajectory is called critical if it ends at a zero or a pole of $q(z)$\cite{Kuz80,Pommerenke2}.

\begin{kuzmina}
\label{thm:K}
Let $A=\{a_1,\ldots,a_m\}\subset\D$ be a set of $m\geq2$ distinct points. Then there uniquely exists a continuum $K_0$, $A\subset K_0\subset\D$, such that 
\[
\cp(K_0,\T) \leq \cp(K,\T)
\]
for any other continuum with $A\subset K\subset\D$. Moreover, there exist $m-2$ points $b_1,\ldots,b_{m-2}\in\D$ such that $K_0$ is the union of the closures of the non-closed negative critical trajectories of the quadratic differential
\[
q(z)dz^2, \quad q(z) := \frac{(z-b_1)\cdot\ldots\cdot(z-b_{m-2})(1-\bar b_1z)\cdot\ldots\cdot(1-\bar b_{m-2}z)}{(z-a_1)\cdot\ldots\cdot(z-a_m)(1-\bar a_1 z)\cdot\ldots\cdot(1-\bar a_mz)},
\]
contained in $\D$. There exists only finitely many  such trajectories. Furthermore, the equilibrium potential $V_\D^{\omega_{(K_0,\T)}}$ satisfies $\displaystyle\left(2\partial_zV_\D^{\omega_{(K_0,\T)}}(z)\right)^2 = q(z)$, $z\in\overline\D$.
\end{kuzmina}

The last equation in Theorem~\hyperref[thm:K]{K} should be understood as follows. The left-hand side of this equality is defined in $\D\setminus K_0$ and represents a holomorphic function there, which coincides with $q$ on its domain of definition. As $K_0$ has no interior because critical trajectories are analytic arcs  with limiting tangents at their endpoints \cite{Pommerenke2}, the equality on the whole set $\overline\D$ is obtained by continuity. Note also that $\D\setminus K_0$ is connected by unicity claimed in Theorem \ref{thm:K}, for the polynomial convex hull of $K_0$ has the same Green capacity  as $K_0$ ({\it cf.} section \ref{sss:cc}). Moreover, it follows from the local theory of quadratic differentials that each $b_j$ is the endpoint of at least three arcs of $K_0$ (because $b_j$ is a zero of $q(z)$) and that each $a_j$ is the endpoint of exactly one arc of $K_0$ (because $a_j$ is a simple pole of $q(z)$).

Having Theorem~\hyperref[thm:K]{K} at hand, we are ready to describe the structure of a minimal set $ \Gamma_\nu$. 

\begin{lem}
\label{lem:1sym}
A minimal set $\Gamma_\nu$ is symmetric \textnormal{(}Definition \ref{df:sym}\textnormal{)} with respect to $\nu^*$.
\end{lem}
\begin{proof}
Let $\widetilde\nu^*$ be the balayage of $\nu^*$ onto $\T_\rho$ with $\rho<1$ but large enough to contain $\Gamma_\nu$ in the interior of $\D_\rho$. Let $\gamma$ be any of the continua constituting $\Gamma_\nu$. Clearly $V:=V_{D_\nu}^{\widetilde\nu^*}$, where $D_\nu=\overline\C\setminus\Gamma_\nu$, is harmonic in $D_\nu\setminus\T_\rho$ and extends continuously to the zero function on $\Gamma_\nu$ since $\Gamma_\nu$ is a regular set.  Moreover, by Sard's theorem on regular values \cite[Sec. 1.7]{GuilleminPollack} there exists $\delta>0$ arbitrarily small such that $\Omega$, the component of $\{z:V(z)<\delta\}$ containing $\gamma$, is itself contained in $\D_\rho$ and its boundary is an analytic Jordan curve, say $L$. Let $\phi$ be a conformal map of $\Omega$ onto $\D$. Set $\widetilde\gamma:=\phi^{-1}(\widetilde K)$, where $\widetilde K$ is the continuum of minimal condenser capacity\footnote{In other words, if we put $\phi(E_f\cap\gamma)=\{p_1,\ldots,p_m\}$ and $g(z):=1/\sqrt[m]{\prod(z-p_j)}$, then $\widetilde K$ is the set of minimal condenser capacity for $g$ as defined in Definition \ref{df:dmcc}.} for $\phi(E_f\cap\gamma)$. Our immediate goal is to show that $\gamma=\widetilde\gamma$. 

Assume to the contrary that $\gamma\neq\widetilde\gamma$, i.e., $\phi(\gamma)=:K\neq\widetilde K$, and therefore
\begin{equation}
\label{eq:asump}
\cp(\widetilde K,\T)<\cp(K,\T).
\end{equation}
 Set
\begin{equation}
\label{eq:dftV}
\widetilde V := \left\{
\begin{array}{ll}
\delta\cp(\widetilde K,\T)\left[V_{\overline\C\setminus\widetilde K}^{\omega_{(\T,\widetilde K)}}\circ\phi\right], & z\in \overline\Omega, \smallskip \\
V, & z\notin\overline\Omega,
\end{array}
\right.
\end{equation}
where $\omega_{(\T,\widetilde K)}$ is the Green equilibrium distribution on $\T$ relative to $\overline\C\setminus\widetilde K$. The functions $V$ and $\widetilde V$ are continuous in $\overline\Omega$ and equal to $\delta$ on $L$. Furthermore, they are harmonic in $\Omega\setminus\gamma$ and $\Omega\setminus\widetilde\gamma$ and equal to zero on $\gamma$ and $\widetilde\gamma$, respectively. Then it follows from Lemma~\ref{lem:nder1} and the conformal invariance of the condenser capacity \eqref{eq:coninv} that
\begin{equation}
\label{eq:normalV}
\frac{1}{2\pi}\int_L\frac{\partial V}{\partial\n}ds = -\delta\cp(K,\T) \quad \mbox{and} \quad \frac{1}{2\pi}\int_L\frac{\partial\widetilde V}{\partial\n}ds = -\delta\cp(\widetilde K,\T),
\end{equation}
where $\partial/\partial\n$ stands for the partial derivative with respect to the inner normal on $L$. (In Lemma~\ref{lem:nder1}, $L$ should be contained within the domain of harmonicity of $V$ and $\widetilde V$. As $V$ and $\widetilde V$ are constant on $L$, they can be harmonically continued across by reflection. Thus, Lemma~\ref{lem:nder1} does apply.) Moreover, $\widetilde V-V_{\widetilde D}^{\widetilde\nu^*}$ is a continuous function on $\overline\C$ that is harmonic in $\widetilde D\setminus L$ by the first claim in Section~\ref{ss:gp}, where $\widetilde D:=(D_\nu\cup\gamma)\setminus\widetilde\gamma$, and is identically zero on $\Gamma:=\overline\C\setminus\widetilde D$. Thus, we can apply Lemma~\ref{lem:nder0} with $\widetilde V-V_{\widetilde D}^{\widetilde\nu^*}$ and $\widetilde D$ (smoothness properties of $V-V_{\widetilde D}^{\widetilde\nu^*}$ follow from the fact that $\widetilde V$ can be harmonically continued across $L$), which states that
\begin{equation}
\label{eq:reprtV}
\widetilde V = V_{\widetilde D}^{\widetilde\nu^*-\sigma}, \quad d\sigma := \frac{1}{2\pi}\frac{\partial(\widetilde V-V)}{\partial\n}ds,
\end{equation}
where $\sigma$ is a finite signed measure supported on $L$ (observe that the outer and inner normal derivatives of $V_{\widetilde D}^{\widetilde\nu^*}$ on $L$ are opposite to each other as $V_{\widetilde D}^{\widetilde\nu^*}$ is harmonic across $L$ and therefore they do not contribute to the density of $\sigma$; due to the same reasoning the outer normal derivative of $\widetilde V$ is equal to minus the inner normal derivative of $V$ by \eqref{eq:dftV}). Hence, one can easily deduce from \eqref{eq:normalV} and \eqref{eq:asump} that
\begin{equation}
\label{eq:positivetotal}
\sigma(L) = \delta\left(\cp(K,\T)-\cp(\widetilde K,\T)\right)>0.
\end{equation}

Since the components of $\Gamma_\nu$ and $\Gamma$ contain exactly the same branch points of $f$ and $\Gamma$ has connected complement (for $D_\nu$ is connected and so is $\overline{\C}\setminus\widetilde{\gamma}$ because $\D\setminus\widetilde{K}$ is connected), it follows that $\Gamma\in\pd_f$ by the monodromy theorem. Moreover, we obtain from \eqref{eq:weightgreen}, \eqref{eq:dftV}, and \eqref{eq:reprtV} that
\[
I_\nu[\Gamma] - I_\nu[\Gamma_\nu] = I_{\widetilde D}[\widetilde\nu^*] - I_{D_\nu}[\widetilde\nu^*] = \int\left(V_{\widetilde D}^{\widetilde\nu^*} - V\right)d\widetilde\nu^* = \int V_{\widetilde D}^\sigma d\widetilde\nu^*
\]
since $\supp(\widetilde\nu^*)\cap\overline\Omega=\varnothing$. Further, applying the Fubini-Tonelli theorem and using \eqref{eq:reprtV} once more, we get that
\[
I_\nu[\Gamma] - I_\nu[\Gamma_\nu] = \int V_{\widetilde D}^{\widetilde\nu^*}d\sigma = \int\widetilde Vd\sigma + I_{\widetilde D}[\sigma] = \delta\sigma(L) + I_{\widetilde D}[\sigma] > 0
\]
by \eqref{eq:positivetotal} and since the Green energy of a signed compactly supported measure of finite Green energy is positive by \cite[Thm. II.5.6]{SaffTotik}. However, the last inequality clearly contradicts the fact that $I_\nu[\Gamma_\nu]$ is maximal among all sets in $\pd_f$ and therefore $\gamma=\widetilde\gamma$. Hence, $K = \widetilde K=\phi(\gamma)$ and $\widetilde V=V$.

Observe now that by Theorem~\hyperref[thm:K]{K} stated just before this lemma and the remarks thereafter, the set  $K$ consists of a finite number of open analytic arcs and their endpoints. These fall into two classes $a_1,\ldots,a_m$ and $b_1,\ldots,b_{m-2}$, members of the first class being endpoints of exactly one arc and members of the second class being endpoints of at least three arcs. Thus, the same is true for $\gamma$. Moreover, the jump of $f$ across any open arc $C\subset\gamma$ cannot vanish, otherwise excising out this arc would leave us with an admissible compact set $\Gamma^\prime\subset\Gamma_\nu$ of strictly smaller $\nu$-capacity since $\omega_{\Gamma_\nu,-U^\nu}(C)>0$ by (\ref{eq:whnustar}) and the properties of balayage at regular points (see Section~\ref{sss:wcf}). 
Hence $\Gamma_\nu$ is a smooth cut (Definition \ref{df:elmin}). Finally, we have that
\[
\frac{\partial V}{\partial\n_\gamma^\pm} = \delta\cp(\phi(\gamma),\T)\left(\frac{\partial}{\partial\n_K^\pm} V_{\overline\C\setminus K}^{\omega_{(\T,K)}}\right)|\phi^\prime|
\]
by \eqref{eq:dftV} and the conformality of $\phi$, where $\partial/\partial\n_\gamma^\pm$ and $\partial/\partial\n_K^\pm$ are the partial derivatives with respect to the one-sided normals at the smooth points of $\gamma$ and $K$, respectively. Thus, it holds that
\[
\frac{\partial V}{\partial\n_\gamma^-} = \frac{\partial V}{\partial\n_\gamma^+}
\]
on the open arcs constituting $\gamma$ since the corresponding property holds for $V_{\overline\C\setminus K}^{\omega_{(\T,K)}}$ by \eqref{eq:mincondcap}. As $\gamma$ was arbitrary continuum from $\Gamma_\nu$, we see that all the requirements of Definition~\ref{df:sym} are fulfilled. 
\end{proof}

To finish the proof of Theorem~\ref{thm:minset}, it only remains to show uniqueness of $\Gamma_\nu$, which is achieved through the following lemma:

\begin{lem}
\label{lem:1e}
$\Gamma_\nu$ is uniquely characterized as a compact set symmetric with respect to $\nu^*$.
\end{lem}
\begin{proof}
Let $\Gamma_s\in\pd_f$ be symmetric with respect to $\nu^*$ and $\Gamma_\nu$ be any set of minimal capacity for Problem $(f,\nu)$. Such a set exists by Lemma \ref{lem:1b} and it is symmetric by Lemma \ref{lem:1sym}.
Suppose to the contrary that $\Gamma_s\neq\Gamma_\nu$, that is,
\begin{equation}
\label{assumption}
\Gamma_s\cap (\C\setminus \Gamma_\nu)\neq\varnothing
\end{equation}
($\Gamma_s$ cannot be a strict subset of $\Gamma_\nu$ for it would have strictly smaller $\nu$-capacity as pointed out in the proof of Lemma \ref{lem:1sym}). We want to show that \eqref{assumption} leads to
\begin{equation}
\label{contradiction}
I_\nu[\Gamma_s] - I_\nu[\Gamma_\nu] > 0.
\end{equation}
Clearly, \eqref{contradiction} is impossible by the very definition of  $\Gamma_\nu$ and therefore the lemma will be proven.

By the very definition of symmetry (Definition \ref{df:sym}), $\Gamma_\nu$ and $\Gamma_s$ are smooth cuts for $f$. In particular, $\overline{C}\setminus\Gamma_\nu$, $\overline{C}\setminus\Gamma_s$ are connected and we have a decomposition of the form
\[
\Gamma_s = E_0^s\cup E_1^s \cup \bigcup \gamma_j^s \quad \mbox{and} \quad \Gamma_\nu = E_0^\nu\cup E_1^\nu \cup \bigcup \gamma_j^\nu,
\]
where $E_0^s, E_0^\nu\subseteq E_f$, $\gamma_j^\nu,\gamma_j^s$ are open analytic arcs, and each element of $E_0^s, E_0^\nu$ is an endpoint of exactly one arc from $\bigcup \gamma_j^s$, $\bigcup \gamma_j^\nu$ while $E_1^s,E_1^\nu$ are finite sets of points each elements of which serving as an endpoint for at least three arcs from $\bigcup \gamma_j^s$, $\bigcup \gamma_j^\nu$, respectively. Moreover, the continuations of $f$ from infinity that are meromorphic outside of $\Gamma_s$ and $\Gamma_\nu$, say $f_s$ and $f_\nu$, are such that the jumps $f_s^+-f_s^-$ and $f_\nu^+-f_\nu^-$ do not vanish on any subset with a limit point of $\bigcup\gamma_j^s$ and $\bigcup\gamma_j^\nu$, respectively. Note that $\Gamma_s\cap\Gamma_\nu\neq\varnothing$ otherwise $\overline{\C}\setminus(\Gamma_\nu\cup\Gamma_s)$ would be connected, so $f$ could be continued analytically over $(\overline{C}\setminus\Gamma_\nu)\cup (\overline{C}\setminus\Gamma_s) =\overline{\C}$ and it would be identically zero by our normalization.

Write $\Gamma_s =\Gamma_s^1\cup\Gamma_s^2$ and $\Gamma_\nu =\Gamma_\nu^1\cup\Gamma_\nu^2$, where $\Gamma_s^k$  (resp. $\Gamma_\nu^k$) are compact disjoint sets such that each connected component of $\Gamma_s^1$ (resp. $\Gamma_\nu^1$) has nonempty intersection with $\Gamma_\nu$ (resp. $\Gamma_s$) while $\Gamma_s^2\cap\Gamma_\nu=\Gamma_\nu^2\cap\Gamma_s=\varnothing$. 

Now, put, for brevity, $D_\nu:=\overline\C\setminus\Gamma_\nu$ and $D_s:=\overline\C\setminus\Gamma_s$. Denote further by $\Omega$ the unbounded component of $D_\nu\cap D_s$. Then
\begin{equation}
\label{commonpoint}
\overline{\Omega}\cap E_0^s \cap \Gamma_s^1 = \overline{\Omega}\cap E_0^\nu \cap \Gamma_\nu^1.
\end{equation}
Indeed, assume that there exists $e\in(\overline{\Omega}\cap E_0^s \cap \Gamma_s^1)\setminus(E_0^\nu \cap \Gamma_\nu^1)$ and let $\gamma^s_e$ be the arc in the union $\bigcup\gamma_j^s$ that has $e$ as one of the endpoints. By our assumption there is an open  disk $W$ centered at $e$ such that $W\cap\Gamma_s=\{e\}\cup (W\cap \gamma^s_e)$ and $W\cap\Gamma_\nu=\varnothing$. Thus $W\setminus(\{e\}\cup\gamma^s_e)\subset D_\nu\cap D_s$. Anticipating the proof of Proposition~\ref{prop:minset} in Section~\ref{ss:prop} (which is independent of the present proof), $\gamma^s_e$ has well-defined tangent at $e$ so we can shrink $W$ to ensure that $\partial W\cap\gamma^s_e$ is a single point. Then $W\setminus(\{e\}\cup\gamma^s_e)$ is connected hence contained in a 
single connected component of $D_\nu\cap D_s$ which is necessarily $\Omega$ since $e\in\overline{\Omega}$. As $f_s$ and $f_\nu$ coincide on $\Omega$ and $f_\nu$ is meromorphic in $W$, $f_s$ has identically zero jump on $\gamma^s_e\cap W$ which is impossible by the definition of a smooth cut. Consequently the left hand side of \eqref{commonpoint} is included in the right hand side and the opposite inclusion can be shown similarly.

\begin{figure}[h!]
\centering
\includegraphics[scale=.6]{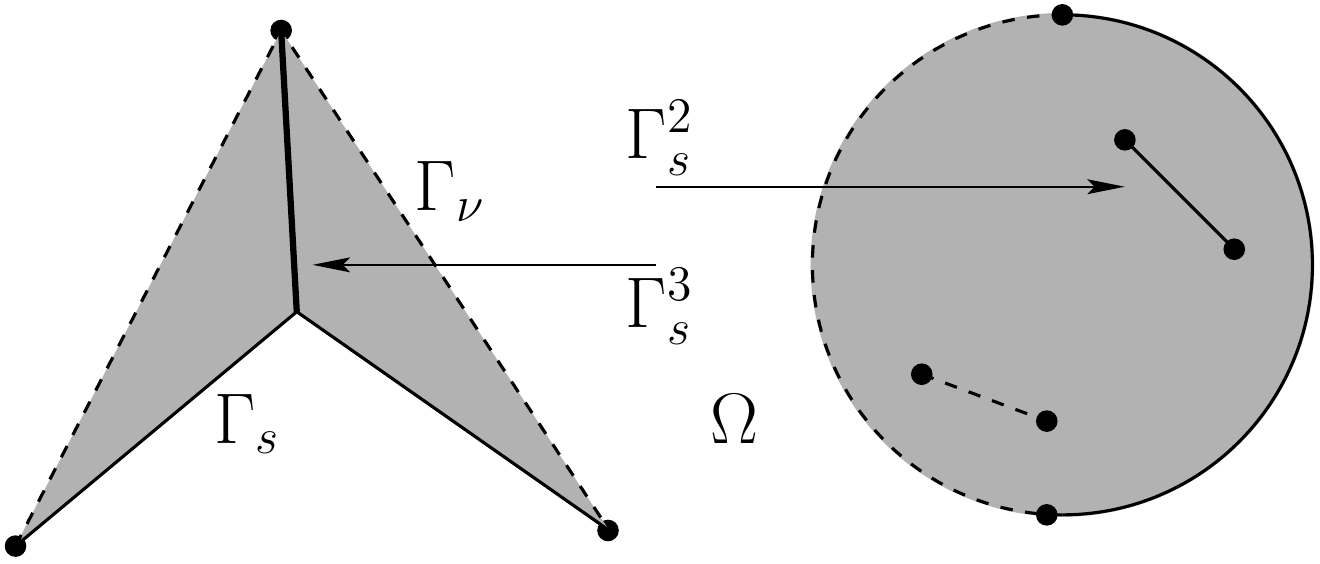}
\caption{\small A particular example of $\Gamma_s$ (solid lines) and $\Gamma_\nu$ (dashed lines). Black dots represent branch points (black dots within big gray disk are branch point $f$ that lie on other sheets of the Riemann surface than the one we fixed). The white area on the figure represents domain $\Omega$.}
\label{fig:TwoCuts}
\end{figure}

Next, observe that
\begin{equation}
\label{Gammas2}
\Gamma_s^2\cap\overline\Omega = \varnothing{\color{red}.}
\end{equation}
Indeed, since $\partial\Omega\subset\Gamma_s\cup\Gamma_\nu$ and $\Gamma_s^2$, $\Gamma^1_s\cup\Gamma_\nu$ are disjoint compact sets, a connected component of $\partial\Omega$ that meets $\Gamma_s^2$ is contained in it. If $z\in\Gamma_s^2\cap\partial\Omega$ lies on $\gamma_j^s$, then by analyticity of the latter each sufficiently small disk $D_z$ centered at $z$ is cut out by $\gamma_j^s\cap D_z$ into two connected components included in $D_\nu\cap D_s$, and of necessity one of them is contained in $\Omega$. Hence $\gamma_j^s\cap D_z$ is contained in $\partial\Omega$, and in turn so does the entire arc $\gamma_j^s$ by connectedness. Hence every component of $\Gamma_s^2\cap\partial\Omega$ consists of a union of arcs $\gamma_j^s$ connecting at their endpoints. Because $\Gamma_s^2$ has no loop, one of them has an endpoint $z_1\in E_0^s\cup E_1^s$ belonging  to no other arc. If $z_1\in E_0^s$, reasoning as we did to prove \eqref{commonpoint} leads to the absurd conclusion that $f_s$ has zero jump across the initial arc. If $z_1\in E_1^s$, anticipating the proof of Proposition~\ref{prop:minset} once again, each sufficiently small disk $D_{z_1}$ centered at $z_1$ is cut out by $\Gamma_s^2\cap D_{z_1}$ into curvilinear sectors included in $D_\nu\cap D_s$, and of necessity one of them is contained in $\Omega$ whence at least two adjacent arcs $\gamma_j^s$ emanating from $z_1$ are included in  $\partial\Omega$. This contradicts the fact that $z_1$ belongs to exactly one arc of the hypothesized component of $\Gamma_s^2\cap\partial\Omega$, and proves \eqref{Gammas2}.

Finally, set
\[
\Gamma_s^3 := \left[\Gamma_s^1\setminus\left(\partial\Omega\setminus E_1^s\right)\right]\cap D_\nu \quad \mbox{and} \quad \Gamma_s^4 := \left[\Gamma_s^1\cap\bigcup\gamma_j^s\right]\cap\partial\Omega\cap D_\nu.
\]
Clearly
\begin{equation}
\label{Gammas3}
\left(\Gamma_s^3\setminus E_1^s\right)\cap\overline\Omega = \varnothing.
\end{equation}
Moreover, observing that any two arcs $\gamma_j^s$, $\gamma_k^\nu$ either coincide or meet in a (possibly empty) discrete set and arguing as we did to prove \eqref{Gammas2}, we see that $\left[\Gamma_s^1\cap\bigcup\gamma_j^s\right]\cap\partial\Omega$ consists of subarcs of arcs $\gamma_j^s$ whose endpoints either belong
to some intersection $\gamma_j^s\cap\gamma_k^\nu$ (in which case they contain this endpoint) or else lie in $E_0^s\cup E_1^s$ (in which case they do not contain this endpoint). Thus $\Gamma_s^4$ is comprised of  open analytic arcs $\widetilde{\gamma}_\ell^s$ contained in $\partial\Omega\cap\bigcup\gamma_j^s$ and disjoint from $\Gamma_\nu$. Hence for any $z\in\Gamma_s^4$, say $z\in \widetilde{\gamma}_\ell^s$, and any disk $D_z$ centered at $z$ of small enough radius it holds that $D_z\cap\partial\Omega=D_z\cap\widetilde{\gamma}_\ell^s$ and that $D_z\setminus\widetilde{\gamma}_\ell^s$ has exactly two connected components:
\begin{equation}
\label{Gammas4}
D_z\cap\Omega \neq \varnothing \quad \mbox{and} \quad D_z\cap\left(\overline\C\setminus\overline\Omega\right)\neq\varnothing
\end{equation}
for if $z\in \widetilde{\gamma}_\ell^s$ was such that $D_z\setminus\widetilde{\gamma}_\ell^s\subset\Omega$, the jump of $f_s$ across $\widetilde{\gamma}_\ell^s$ would be zero as the jump of $f_\nu$ is zero there and $f_s=f_\nu$ in $\Omega$ (see Figure~\ref{fig:TwoCuts}).

As usual, denote by $\widetilde\nu^*$ the balayage of $\nu^*$ onto $\T_\rho$ with $\rho \in (r, 1)$ but large enough so that $\Gamma_s$ and $\Gamma_\nu$ are contained in the interior of $\D_\rho$ (see Lemma~\ref{lem:1b} for the definition of $r$). Then, according to \eqref{eq:weightgreen} and \eqref{eq:di1}, it holds that
\begin{equation}
\label{eq:apprstep2}
I_\nu[\Gamma_s]-I_\nu[\Gamma_\nu] = I_{D_s}[\widetilde\nu^*] - I_{D_\nu}[\widetilde\nu^*] = \di_{D_s}(V_s) - \di_{D_\nu}(V_\nu),
\end{equation}
where $V_s:=V_{D_s}^{\widetilde\nu^*}$ and $V_\nu:=V_{D_\nu}^{\widetilde\nu^*}$. Indeed, as $\widetilde\nu^*$ has finite energy (see Section~\ref{BP}),  the Dirichlet integrals of $V_s$ and $V_\nu$ in the considered domains (see Section~\ref{DI}) are well-defined by Proposition~\ref{prop:minset}, which is proven later but independently of the results in this section.

Set $D:=D_\nu\setminus(\Gamma_s^2\cup\Gamma_s^3)$. Since $\left[\Gamma_s^1\setminus\left(\partial\Omega\setminus E_1^s\right)\right]$ consists  of piecewise smooth arcs in $\Gamma_s^1$ whose endpoints either belong to this arc (if they lie in $E_1^s$), or to $E_0^s\cap\Gamma_s^1$ (hence also to $\Gamma_\nu$ by \eqref{commonpoint}), or else to some intersection $\gamma_j^s\cap\gamma_k^\nu$ (in which case they belong to $\Gamma_\nu$ again), we see that $D$ is an open set. As $V_\nu$ is harmonic across $\Gamma_s^2\cup\Gamma_s^3$ and $V_s$ is harmonic across $\Gamma_\nu\setminus\Gamma_s$, we get from \eqref{ajoutbord} that
\begin{equation}
\label{Dirichlet1}
\di_{D_\nu}(V_\nu) = \di_D(V_\nu) \quad \mbox{and} \quad \di_{D_s}(V_s) = \di_{D\setminus\Gamma_s^4}(V_s)
\end{equation}
since $D_s\setminus\Gamma_\nu=D_\nu\setminus\Gamma_s=D\setminus\Gamma_s^4$, by inspection on using \eqref{commonpoint}.

Now, recall that $\Gamma_s$ has no interior and $V_s\equiv0$ on $\Gamma_s$, that is, $V_s$ is defined in the whole complex plane. So, we can define a function on $\C$ by putting
\begin{equation}
\label{defVt}
\widetilde V := \left\{
\begin{array}{rl}
V_s, & \mbox{in} \quad \Omega, \smallskip \\
-V_s, & \mbox{otherwise}.
\end{array}
\right.
\end{equation}
We claim that $\widetilde V$ is superharmonic in $D$ and harmonic in $D\setminus\T_\rho$. Indeed, it is clearly harmonic in $D\setminus(\Gamma_s^4\cup\T_\rho)=(D_s\cap D_\nu)\setminus\T_\rho$and superharmonic in a neighborhood of $\T_\rho\subset\Omega$ where its weak Laplacian is $-2\pi\widetilde\nu^*$ which is a negative measure. Moreover, $\Gamma_s^4$ is a collection of open analytic arcs such that $\partial V_s/\partial\n^+=\partial V_s/\partial\n^-$ by the symmetry of $\Gamma_s$, where $\n^\pm$ are the two-sided normal on each subarc of $\Gamma_s^4$. The equality of the normals means that $V_s$ can be continued harmonically across each subarc of $\Gamma_s^4$ by $-V_s$. Hence, \eqref{Gammas4} and the definition of $\widetilde V$ yield that it is harmonic across $\Gamma_s^4$ thereby proving the claim. Thus, using \eqref{decompositivity} (applied with $D^\prime=\Omega$) and \eqref{ajoutbord}, we obtain
\begin{equation}
\label{Dirichlet2}
\di_{D\setminus\Gamma_s^4}(V_s) = \di_{D\setminus\Gamma_s^4}(\widetilde V) = \di_D(\widetilde V)
\end{equation}
hence combining \eqref{eq:apprstep2}, \eqref{Dirichlet1}, and \eqref{Dirichlet2}, we see that
\begin{equation}
\label{contr1}
I_\nu[\Gamma_s]-I_\nu[\Gamma_\nu] = \di_D(\widetilde V) - \di_D(V_\nu).
\end{equation}

By the first claim in Section~\ref{ss:gp}, it holds that $h:=\widetilde V-V_\nu$ is harmonic in $D$. Observe that $h$ is not a constant function, for it tends to zero at each point of $\Gamma_s\cap\Gamma_\nu\subset\partial D$ whereas it tends to a strictly negative value at each point of $\Gamma_s\cap D_\nu\subset\overline{D}$ which is
nonempty by \eqref{assumption}. Then
\begin{equation}
\label{contr2}
\di_D(\widetilde V) = \di_D(V_\nu) + \di_D(h) + 2\di_D(V_\nu,h).
\end{equation}
Now, $V_\nu\equiv0$ on $\Gamma_\nu$ and it is harmonic across $\Gamma_s^2\cup\Gamma_s^3$, hence 
\[
\frac{\partial h}{\partial \n^+}+\frac{\partial h}{\partial \n^-} = \frac{\partial {\widetilde V}}{\partial \n^+}+\frac{\partial {\widetilde V}}{\partial \n^-} \quad \mbox{on} \quad \Gamma_s^2\cup\Gamma_s^3.
\]
Consequently, we get from \eqref{eq:greenformula}, since $\widetilde V=-V_s$ in the neighborhood of $\Gamma_s^2\cup\Gamma_s^3$ by \eqref{Gammas2} and \eqref{Gammas3}, that
\begin{equation}
\label{contr3}
\di_D(V_\nu,h) = -\int_{\Gamma_s^2\cup\Gamma_s^3}V_\nu\left(\frac{\partial \widetilde V}{\partial \n^+}+\frac{\partial \widetilde V}{\partial \n^-}\right)\frac{ds}{2\pi} = \int_{\Gamma_s^2\cup\Gamma_s^3}V_\nu\left(\frac{\partial V_s}{\partial \n^+}+\frac{\partial V_s}{\partial \n^-}\right)\frac{ds}{2\pi}\geq0
\end{equation}
because $V_\nu$ is nonnegative  while $\partial V_s/\partial n^+$, $\partial V_s/\partial n^-$ are also nonnegative on
$\Gamma_s^2\cup\Gamma_s^3$ as $V_s\geq0$ vanishes there. Altogether, we obtain from \eqref{contr1}, \eqref{contr2}, and \eqref{contr3} that
\[
I_\nu[\Gamma_s]-I_\nu[\Gamma_\nu] \geq \di_D(h)  > 0
\]
by \eqref{positivity} and since $h=\tilde V -V_\nu$ is a non-constant harmonic function in $D$. This shows \eqref{contradiction} and finishes the proof of the lemma.
\end{proof}

\subsection{Proof of Proposition~\ref{prop:minset}}
\label{ss:prop}

It is well known that $H_{\omega,\Gamma}$ is holomorphic in the domain of harmonicity of $V_{\overline\C\setminus\Gamma}^\omega$, that is, in $\overline\C\setminus(\Gamma\cup\supp(\omega))$. It is also clear that 
$H_{\omega,\Gamma}^\pm$ exist smoothly on each $\gamma_j$ since $V_{\overline\C\setminus\Gamma}^\omega$ can be harmonically continued across each side of $\gamma_j$.

Denote by $\n_t^\pm$ the one-sided unit normals at $t\in\bigcup\gamma_j$ and by $\tau_t$ the unit tangent pointing in the positive direction. Let further $n^\pm(t)$ be the unimodular complex numbers corresponding to vectors $\n^\pm_t$. Then the complex number corresponding to $\tau_t$ is $\mp in^\pm(t)$ and it can be readily verified that
\[
\frac{\partial V_{\overline\C\setminus\Gamma}^\omega}{\partial\n^\pm_t} = 2\re\left(n^\pm(t)H_{\omega,\Gamma}^\pm(t)\right) \quad \mbox{and} \quad \frac{\partial \left(V_{\overline\C\setminus\Gamma}^\omega\right)^\pm}{\partial\tau_t} = \mp2\im\left(n^\pm(t)H_{\omega,\Gamma}^\pm(t)\right).
\]
As $\left(V_{\overline\C\setminus\Gamma}^\omega\right)^\pm\equiv0$ on $\Gamma$, the tangential derivatives above are identically zero, therefore $n^\pm H_{\omega,\Gamma}^\pm$ is real on $\Gamma$. Moreover since $n^+=-n^-$ and by the symmetry property \eqref{eq:GreenPotSym}, it holds that $H_{\omega,\Gamma}^+=-H_{\omega,\Gamma}^-$ on $\bigcup\gamma_j$. Hence, $H_{\omega,\Gamma}^2$ is holomorphic in $\overline\C\setminus(E_0\cup E_1\cup\supp(\omega))$. Since $E_0\cup E_1$ consists of isolated points around which $H_{\omega,\Gamma}^2$ is holomorphic each $e\in E_0\cup E_1$ is either a pole, a removable singularity, or an essential one. As $H_{\omega,\Gamma}$ is holomorphic on a two-sheeted Riemann surface above the point, it cannot have an essential singularity since its primitive has bounded real part $\pm V_{\overline\C\setminus\Gamma}^\omega$. Now, by repeating the arguments in \cite[Sec. 8.2]{Pommerenke2}, we deduce that $(z-e)^{j_e-2}H^2_{\omega,\Gamma}(z)$ is holomorphic and non-vanishing in some neighborhood of $e$ where $j_e$ is the number of arcs $\gamma_j$ having $e$ as an endpoint, that the tangents at $e$ to these arcs exist, and that they are equiangular if $j_e>1$.

\subsection{Proof of Theorem \ref{thm:mpa}}
\label{ss:52}

The following theorem \cite[Thm. 3]{GRakh87} and its proof are essential in establishing Theorem~\ref{thm:mpa}. Before stating this result, we remind the reader that a polynomial $v$ is said to by spherically normalized if it has the form
\begin{equation}
\label{eq:sphnp}
v(z) = \prod_{v(e)=0,~|e|\leq1}(z-e)\prod_{v(e)=0,~|e|>1}(1-z/e).
\end{equation}
We also recall from  \cite{GRakh87} the notions of a \emph{tame set} and a \emph{tame point} of a set. A point $z$ belonging to a compact set $\Gamma$ is called tame, if there is a disk centered at $z$ whose intersection with $\Gamma$ is an analytic arc. A compact set $\Gamma$ is called tame, if $\Gamma$ is non-polar and quasi-every point of $\Gamma$ is tame. 

A tame compact set $\Gamma$ is said to have the S-property in the field $\psi$, assumed to be harmonic in some neighborhood of $\Gamma$, if $\supp(\omega_{\Gamma,\psi})$ forms a tame set as well, every tame point of $\supp(\omega_{\Gamma,\psi})$ is also a tame point of $\Gamma$, and the equality in \eqref{eq:sproperty} holds at each tame point of $\supp(\omega_{\Gamma,\psi})$.

Whenever the tame compact set $\Gamma$ has connected complement in a simply connected region $G\supset\Gamma$ and  $g$ is holomorphic in $G\setminus\Gamma$, we write $\oint_\Gamma g(t)\,dt$ for the contour integral of $g$ over some (hence any) system of curves encompassing $\Gamma$ once in $G$ in the positive direction. Likewise, the Cauchy integral $\oint_\Gamma g(t)/(z-t)\,dt$ can be defined at any $z\in\overline{\C}\setminus\Gamma$ by choosing the previous system of curves in such a way that it separates $z$ from $\Gamma$.

If $g$ has limits from each side at tame points of $\Gamma$, and if these limits  are integrable with respect to linear measure on $\Gamma$, then the previous integrals may well be rewritten as integrals on $\Gamma$ with $g$ replaced by its jump across $\Gamma$. However, this is \emph{not} what is meant by the notation $\oint_\Gamma$.

\begin{gonchar}
\label{thm:GR}
Let $G\subset\D$ be a simply connected domain and $\Gamma\subset G$ be a tame compact set with connected complement. Let also $g$ be holomorphic in $G\setminus\Gamma$ and have continuous limits on $\Gamma$ from each side in the neighborhood of every tame point, whose jump across $\Gamma$ is non-vanishing q.e. Further, let $\{\Psi_n\}$ be a sequence of functions that satisfy:
\begin{enumerate}
\item $\Psi_n$ is holomorphic in $G$ and $-\frac{1}{2n}\log|\Psi_n| \to\psi$ locally uniformly there, where $\psi$ is harmonic in $G$;
\item $\Gamma$ possesses the S-property in the field $\psi$ \textnormal{(}see \eqref{eq:sproperty}\textnormal{)}.
\end{enumerate}
Then, if the polynomials $q_n$, $\deg(q_n)\leq n$, satisfy the orthogonality relations\footnote{Note that the orthogonality in \eqref{eq:orthorel} is non-Hermitian, that is, no conjugation is involved.}
\begin{equation}
\label{eq:orthorel}
\oint_\Gamma q_n(t)l_{n-1}(t)\Psi_n(t)g(t)dt = 0, \quad \mbox{for any} \quad l_{n-1}\in\poly_{n-1},
\end{equation}
then $\mu_n\cws\omega_{\Gamma,\psi}$, where $\mu_n$ is the normalized counting measure of zeros of $q_n$. Moreover, if the polynomials $q_n$ are spherically normalized, it holds that
\begin{equation}
\label{eq:funan1}
\left|A_n(z)\right|^{1/2n} \cic \exp\{-c(\psi;\Gamma)\} \quad \mbox{in} \quad \overline\C\setminus\Gamma,
\end{equation}
where $c(\psi;\Gamma)$ is the modified Robin constant \textnormal{(}Section~\ref{sss:wc}\textnormal{)}, and 
\begin{equation}
\label{eq:funan2}
A_n(z) := \oint_\Gamma q_n^2(t)\frac{(\Psi_ng)(t)dt}{z-t} = \frac{q_n(z)}{l_n(z)}\oint_\Gamma (l_nq_n)(t)\frac{(\Psi_ng)(t)dt}{z-t},
\end{equation}
where $l_n$ can be any\footnote{The fact that we can pick an arbitrary polynomial $l_n$ for this integral representation of $A_n$ is a simple consequence of orthogonality relations \eqref{eq:orthorel}.} nonzero polynomial of degree at most $n$.
\end{gonchar}
\begin{proof}[Proof of Theorem \ref{thm:mpa}]
Let $E_n$ be the sets constituting the interpolation scheme $\E$. Set $\Psi_n$ to be the reciprocal of the spherically normalized polynomial with zeros at the finite elements of $E_n$, i.e., $\Psi_n=1/\tilde v_n$, where $\tilde v_n$ is the spherical renormalization of $v_n$ (see Definition~\ref{df:pade} and \eqref{eq:sphnp}). Then the functions $\Psi_n$ are holomorphic and non-vanishing in $\C\setminus\supp(\E)$ (in particular, in $\D$), $\frac{1}{2n}\log|\Psi_n| \cic U^\nu$ in $\overline\C\setminus\supp(\nu^*)$ by Lemma~\ref{lem:cwssimcic}, and this convergence is locally uniform in $\D$ by definition of the asymptotic distribution and since $\log1/|z-t|$ is continuous on a neighborhood of $\supp(\E)$ for fixed $z\in\D$. As $U^{\nu}$ is harmonic in $\D$, requirement (1) of Theorem~\hyperref[thm:GR]{GR} is fulfilled with $G=\D$ and $\psi=-U^\nu$. Further, it follows from Theorem~\ref{thm:minset} that $\Gamma_\nu$ is a symmetric set. In particular it is a smooth cut, hence it is tame with tame points $\cup_j\gamma_j$. Moreover, since $\Gamma_\nu$ is regular, we have that $\supp(\omega_{\Gamma,\psi})=\Gamma_\nu$ by \eqref{eq:weqpot} and properties of balayage (Section \ref{dubalai}). Thus, by the remark after Definition~\ref{df:sym}, symmetry implies that $\Gamma_\nu$ possesses the S-property in the field $-U^\nu$ and therefore requirement (2) of Theorem~\hyperref[thm:GR]{GR} is also fulfilled. Let now $Q$, $\deg(Q)=:m$, be a fixed polynomial such that the only singularities of $Qf$ in $\D$ belong to $E_f$. Then $Qf$ is holomorphic and single-valued in $\D\setminus\Gamma_\nu$, it extends continuously from each side on $\cup\gamma_j$, and has a jump there which is continuous and non-vanishing except possibly at countably many points. All the  requirement of Theorem~\hyperref[thm:GR]{GR} are then fulfilled with $g=Qf$.

Let $L\subset\D$ be a smooth Jordan curve that separates $\Gamma_\nu$ and the poles of $f$ (if any) from $\E$. Denote by $q_n$ the  spherically normalized denominators of the multipoint Pad\'e approximants to $f$ associated with $\E$. It is a standard consequence of Definition \ref{df:pade} (see {\it e.g.} \cite[sec. 1.5.1]{GRakh87}) that
\begin{equation}
\label{eq:proof1}
\int_{L} z^jq_n(z)\Psi_n(z)f(z)dz = 0, \quad j\in\{0,\ldots,n-1\}.
\end{equation}

Clearly, relations \eqref{eq:proof1} imply that
\begin{equation}
 \label{eq:proof2}
 \oint_{\Gamma_\nu} (lq_n\Psi_nfQ)(t)dt = 0, \quad \deg(l) < n-m.
\end{equation}
 Equations \eqref{eq:proof2} differ from \eqref{eq:orthorel} only in the reduction of the degree of polynomials $l$ by a constant  $m$. However, to derive the first conclusion of Theorem~\hyperref[thm:GR]{GR}, namely that $\mu_n\cws\omega_{\Gamma,\psi}$, orthogonality relations \eqref{eq:orthorel} are used solely when applied to a specially constructed sequence $\{l_n\}$ such that $l_n=l_{n,1}l_{n,2}$, where $\deg(l_{n,1})\leq n\theta$, $\theta<1$, and $\deg(l_{n,2})=o(n)$ as $n\to\infty$ (see the proof of \cite[Thm. 3]{GRakh87} in between equations (27) and (28)).  Thus, the proof is still applicable in our situation, to the effect that the normalized counting measures of the zeros of $q_n$ converge weak$^*$ to $\widehat\nu^*=\omega_{\Gamma_\nu,-U^\nu}$, see \eqref{eq:whnustar}.

 For each $n\in\N$, let $q_{n,m}$, $\deg(q_{n,m})=n-m$, be a divisor of $q_n$. Observe that the polynomials $q_{n,m}$ have exactly the same asymptotic zero distribution in the weak$^*$ sense as the polynomials $q_n$. Put
 \begin{equation}
 \label{eq:proof3}
 A_{n,m}(z) := \oint (q_{n,m}q_n)(t)\frac{(\Psi_nfQ)(t)dt}{z-t}, \quad z\in D_\nu.
 \end{equation}
Due to orthogonality relations \eqref{eq:proof2}, $A_{n,m}$ can be equivalently rewritten as
 \begin{equation}
 \label{eq:proof4}
 A_{n,m}(z) := \frac{q_{n,m}(z)}{l_{n-m}(z)}\oint (l_{n-m}q_n)(t)\frac{(\Psi_n
fQ)(t)dt}{z-t}, \quad z\in D_\nu,
 \end{equation}
 where $l_{n-m}$ is an arbitrary polynomial of degree at most $n-m$. Formulae \eqref{eq:proof3} and \eqref{eq:proof4} differ from \eqref{eq:funan2} in the same manner as orthogonality relations \eqref{eq:proof2} differ from those in \eqref{eq:orthorel}.  Examination of the proof of \cite[Thm. 3]{GRakh87} (see the discussion there between equations (33) and (37)) shows that limit \eqref{eq:funan1} is proved using expression \eqref{eq:funan2} for $A_n$ with a choice of polynomials $l_n$ that satisfy some set of \emph{asymptotic} requirements  and can be chosen to have the degree $n-m$. Hence it still holds that
\begin{equation}
\label{eq:proof6}
|A_{n,m}(z)|^{1/2n}\cic\exp\left\{-c(-U^\nu;\Gamma_\nu)\right\} \quad \mbox{in} \quad D_\nu.
\end{equation}

Finally, using the Hermite interpolation formula like in \cite[Lem. 6.1.2]{StahlTotik}, the error of approximation has the following representation 

\begin{equation}
\label{eq:proof5}
(f-\Pi_n)(z) = \frac{A_{n,m}(z)}{(q_{n,m}q_nQ\Psi_n)(z)}, \quad z\in D_\nu.
\end{equation}
From Lemma~\ref{lem:cwssimcic} we know that $\log(1/|q_n|)/n\cic V_*^{\widehat{\nu^*}}=V^{\widehat{\nu^*}}$ in $D_\nu$, since ordinary and spherically normalized potentials coincide for measures supported in $\overline{\D}$. This fact together with \eqref{eq:proof6} and \eqref{eq:proof5} easily yield that
\[
|f-\Pi_n|^{1/2n} \cic \exp\left\{-c(-U^\nu;\Gamma_\nu)+V^{\widehat\nu^*}-U^\nu\right\} \quad \mbox{in} \quad D_\nu\setminus\supp(\nu^*).
\]
Therefore, \eqref{eq:mpa} follows from \eqref{eq:toRemind1} and the fact that $U^\nu=V_*^{\nu^*}$ by the remark at the beginning of Section~\ref{sss:wcf}.
\end{proof}

\subsection{Proof of Theorem~\ref{thm:dmcc}, Theorem~\ref{thm:convcap}, Corollary~\ref{cor:normconv},  Theorem~\ref{cor:L2T}, and Theorem~\ref{thm:L2T}}
\label{ss:53}

\begin{proof}[Proof of Theorem \ref{thm:dmcc}] Let $\Gamma\in\pk_f(G)$ be a smooth cut for $f$ that satisfies \eqref{eq:mincondcap} and $\Theta$ be a conformal map of $\D$ onto $G$. Set $K:=\Theta^{-1}(\Gamma)$. Then we get from the conformal invariance of the condenser capacity (see \eqref{eq:coninv}) and the maximum principle for harmonic functions that
\[
\cp(\Gamma,T) = \cp(K,\T) \quad \mbox{and} \quad V_{\overline\C\setminus K}^{\omega_{(K,\T)}} = V_{\overline\C\setminus\Gamma}^{\omega_{(\Gamma,T)}}\circ\Theta \quad \mbox{in} \quad \D.
\]
As $\Theta$ is conformal in $\D$, it can be readily verified that $V_{\overline\C\setminus K}^{\omega_{(K,\T)}}$ satisfies \eqref{eq:mincondcap} as well (naturally, on $K$). Univalence of $\Theta$ also implies that the continuation properties of $(f\circ\Theta)(\Theta^\prime)^{1/2}$ in $\D$ are exactly the same as those of $f$ in $G$. Moreover, this is also true for $f_\Theta$, the orthogonal projection of $(f\circ\Theta)(\Theta^\prime)^{1/2}$ from $L^2$ onto $\bar H_0^2$ (see Section~\ref{sec:ra}). Indeed, $f_\Theta$ is holomorphic in $\Om$ by its very definition and can be continued analytically across $\T$ by  $(f\circ\Theta)(\Theta^\prime)^{1/2}$ minus the orthogonal projection of the latter from $L^2$ onto $H^2$, which is holomorphic in $\D$ by definition. Thus, $f_\Theta\in\alg(\D)$ and $\Gamma\in\pk_f(G)$ if and only if $K\in\pk_{f_\Theta}$. Therefore, it is enough to consider only the case $G=\D$.

Let $\Gamma\in\pk_f$ be a smooth cut for $f$ that satisfies \eqref{eq:mincondcap} and $\K$ be the set of minimal condenser capacity ({\it cf.} Theorem \ref{thm:S}). We must prove that $\Gamma=\K$. Set, for brevity, $D_\Gamma:=\overline\C\setminus\Gamma$, $V_\Gamma:=V_{D_\Gamma}^{\omega_{(\T,\Gamma)}}$, $D_\K:=\overline\C\setminus\K$, $V_\K:=V_{D_\K}^{\omega_{(\T,\K)}}$, and $\Omega$ to be the unbounded component of $D_\K\cap D_\Gamma$.  Let also $f_{D_\Gamma}$ and $f_{D_\K}$ indicate the  meromorphic branches of $f$ in $D_\Gamma$ and $D_\K$, respectively. Arguing as we did to prove (\ref{Gammas2}), we see that no connected component of $\partial\Omega$ can lie entirely in $\Gamma\setminus \K$ (resp. $\K\setminus\Gamma$) otherwise the jump of $f_{D_\Gamma}$ (resp. $f_{D_\K}$) across some subarc of $\Gamma$ (resp. $\K$) would vanish. Hence by connectedness
\begin{equation}
\label{eq:gammanotk1}
\Gamma\cap \K\cap\partial\Omega\neq\varnothing.
\end{equation}
First, we deal with the special situation where $\omega_{(\T,\K)}=\omega_{(\T,\Gamma)}$. Then $V_\Gamma-V_\K$ is harmonic in $\Omega$ by the first claim in Section~\ref{ss:gp}.  As both potentials are constant in $\Om\subset\Omega$, we get that $V_\Gamma=V_\K+\const$ in $\Omega$. Since $\K$ and $\Gamma$ are regular sets, potentials $V_\Gamma$ and $V_\K$ extend continuously to $\partial\Omega$ and vanish at $\partial\Omega\cap\Gamma\cap\K$ which  is non-empty by \eqref{eq:gammanotk1}. Thus, equality of the equilibrium measures 
means that $V_\Gamma\equiv V_\K$ in $\overline{\Omega}$. However, because $V_\Gamma$ (resp. $V_\K$) vanishes precisely on $\Gamma$ (resp. $\K$), this is possible only if $\partial\Omega\subset\Gamma\cap\K$. Taking complements in $\overline{\C}$, we conclude that $D_\Gamma\cup D_\K$, which is connected and contains $\infty$, does not meet $\partial\Omega$. Therefore $D_\Gamma\cup D_\K\subset\Omega\subset D_\Gamma\cap D_\K$, hence $D_\Gamma= D_\K$ thus $\Gamma=\K$, as desired.

In the rest of the proof we assume for a contradiction that $\Gamma\neq\K$. Then $\omega_{(\T,\K)}\neq\omega_{(\T,\Gamma)}$ in view of what precedes, and therefore 
\begin{equation}
\label{eq:pr1-2}
\di_{D_\K}(V_\K) = I_{D_\K}\left[\omega_{(\T,\K)}\right] < I_{D_K}\left[\omega_{(\T,\Gamma)}\right] = \di_{D_\K}(V_{D_\K}^{\omega_{(\T,\Gamma)}})
\end{equation}
by \eqref{eq:di1} and since the Green equilibrium measure is the unique minimizer of the Green energy.

The argument now follows the lines of the proof of Lemma \ref{lem:1e}. Namely, we write
\[
\Gamma= E_0^\Gamma\cup E_1^\Gamma \cup \bigcup \gamma_j^\Gamma, \qquad 
\K = E_0^\K\cup E_1^\K \cup \bigcup \gamma_j^\K,
\]
and we define the sets $\Gamma^1$, $\Gamma^2$, $\Gamma^3$, $\Gamma^4$ like we did in that proof  for $\Gamma^1_s$, $\Gamma_s^2$, $\Gamma_s^3$, $\Gamma_s^4$, upon replacing $D_s$ by $D_\Gamma$, $D_\nu$ by $D_\K$, $E_j^s$ by $E_j^\Gamma$, $E_j^\nu$ by $E_j^\K$, $\gamma_j^s$ by $\gamma_j^\Gamma$ and
$\gamma_j^\nu$ by $\gamma_j^\K$. The same reasoning that led to us to \eqref{Gammas2} and \eqref{Gammas3}
yields
\begin{equation}
\label{GammaG}
\Gamma^2\cap\overline\Omega = \varnothing,\qquad
\left(\Gamma^3\setminus E_1^\Gamma\right)\cap\overline\Omega = \varnothing.
\end{equation}
Subsequently we set $D:=D_\K\setminus(\Gamma^2\cup\Gamma^3)$ and we prove in the same way that it is an open set satisfying 
\begin{equation}
\label{dirichlet0G}
\di_{D_K}(V_{D_\K}^{\omega_{(\T,\Gamma)}}) = \di_D(V_{D_\K}^{\omega_{(\T,\Gamma)}})
\quad \mbox{and} \quad \di_{D_\Gamma}(V_\Gamma) = \di_{D\setminus\Gamma^4}(V_\Gamma)
\end{equation}
(compare \eqref{Dirichlet1}). Defining $\widetilde V$ as in \eqref{defVt} with $V_s$ replaced by $V_\Gamma$, and using the symmetry of $\Gamma$ (that is,  \eqref{eq:mincondcap} with $\Gamma$ instead of $\K$, which allows us to continue $V_\Gamma$ harmonically by $-V_\Gamma$ across each arc $\gamma_j^\Gamma$) we find that $\widetilde V$ is harmonic in $D\setminus\T$, superharmonic in $D$, and that
\begin{equation}
\label{dirichlet1G}
\di_{D_\Gamma\setminus\Gamma^4}(V_\Gamma) = \di_D(\widetilde V) 
\end{equation}
(compare \eqref{Dirichlet2}). Next, we set $h:=\widetilde V-V_{D_\K}^{\omega_{(\T,\Gamma)}}$ which is harmonic in $D$ by the first claim in Section~\ref{ss:gp}, and since $h=V_\Gamma-V_{D_\K}^{\omega_{(\Gamma,\T)}}$ in $\Omega\supset\T$. Because $\widetilde V=-V_\Gamma$ in the neighborhood of $\Gamma_s^2\cup\Gamma_s^3$ 
by \eqref{GammaG}, the same computation as in \eqref{contr3} gives us
\[
\di_D(V_{D_\K}^{\omega_{(\T,\Gamma)}},h)\geq0,
\]
so  we get from \eqref{eq:di1}, \eqref{dirichlet0G}, \eqref{dirichlet1G}, \eqref{positivity} and \eqref{eq:pr1-2} that
\begin{eqnarray}
I_{D_\Gamma}[\omega_{(\T,\Gamma)}] &=& \di_{D_\Gamma}(V_\Gamma) = \di_D(\widetilde V) = \di_D(V_{D_\K}^{\omega_{(\T,\Gamma)}} + h)\nonumber\\
&=&\di_D(V_{D_\K}^{\omega_{(\T,\Gamma)}}) + 2\di_D(V_{D_\K}^{\omega_{(\T,\Gamma)}},h) + \di_D(h) \nonumber \\
\label{eq:pr1-3}
{} &\geq& \di_{D_\K}(V_{D_\K}^{\omega_{(\T,\Gamma)}}) + \di_D(h) > \di_{D_\K}(V_\K) = I_{D_\K}[\omega_{(\T,\K)}].
\end{eqnarray}
However, it holds that
\[
I_{D_\K}[\omega_{(\T,\K)}]=1/\cp(\K,\T) \quad \mbox{and} \quad I_{D_\Gamma}[\omega_{(\T,\Gamma)}]=1/\cp(\Gamma,\T)
\]
by \eqref{eq:exchange}. Thus, \eqref{eq:pr1-3} yields that $\cp(\Gamma,\T)<\cp(\K,\T)$, which is impossible by the very definition of $\K$. This contradiction finishes the proof.
\end{proof}

\begin{proof}[Proof of Theorem \ref{thm:convcap}]
Let $\{r_n\}$ be a sequence of irreducible critical points for $f$. Further, let $\nu_n$ be the normalized counting measures of the poles of $r_n$ and $\nu$ be a weak$^*$ limit point of $\{\nu_n\}$, i.e., $\nu_n\cws\nu$, $n\in\N_1\subset\N$. Recall that all the poles of $r_n$ are contained in $\D$ and therefore $\supp(\nu)\subseteq\overline\D$. 

By Theorem \ref{thm:minset}, there uniquely exists a minimal set $\Gamma_\nu$ for Problem $(f,\nu)$.  Let $Z_n$ be the set of poles of $r_n$, where each pole appears with twice its multiplicity. As mentioned in Section~\ref{sec:ra}, each $r_n$ interpolates $f$ at the points of $Z_n^*$, counting multiplicity. Hence, $\{r_n\}_{n\in\N_1}$ is the sequence of multipoint Pad\'e approximants associated with the triangular scheme $\E=\{Z_n^*\}_{n\in\N_1}$ that has asymptotic distribution $\nu^*$, where $\nu^*$ is the reflection of $\nu$ across $\T$. So, according to Theorem~\ref{thm:mpa} (applied for subsequences), it holds that $\nu = \widehat\nu^*$, $\supp(\nu)=\Gamma_\nu$, i.e., $\nu$ is the balayage of its own reflection across $\T$ relative to $D_\nu$.

Applying Lemma~\ref{lem:pt}, we deduce that $\nu$ is the Green equilibrium distribution on $\Gamma_\nu$ relative to $\D$, that is, $\nu=\omega_{(\Gamma_\nu,\T)}$, and $\widetilde\nu$, the balayage of $\nu$ onto $\T$, is the Green equilibrium distribution on $\T$ relative to $D_\nu$, that is, $\widetilde\nu=\omega_{(\T,\Gamma_\nu)}$. Moreover, Lemma~\ref{lem:pt} yields that $V_{D_\nu}^{\nu^*}=V_{D_\nu}^{\widetilde\nu}$ in $\D$ and therefore $V_{D_\nu}^{\widetilde\nu}$ enjoys symmetry property \eqref{eq:mincondcap} by Theorem~\ref{thm:minset}. Hence, we get from Theorem~\ref{thm:dmcc} that $\Gamma_\nu=\K$, the set of minimal condenser capacity for $f$, and that $\nu=\ged$. Since $\nu$ was an arbitrary limit point of $\{\nu_n\}$, we have that $\nu_n\cws\ged$ as $n\to\infty$. Finally, observe that \eqref{eq:Convergence1} is a direct consequence of Theorem~\ref{thm:mpa}.

To prove \eqref{eq:Convergence2}, we need to go back to representation \eqref{eq:proof5}, where $q_n\in\mpoly_n$ is the denominator of an irreducible critical point $r_n$ and $q_{n,m}$, $\deg(q_{n,m})=n-m$, is an arbitrary divisor of $q_n$, while $\Psi_n=1/\widetilde q_n^2$ with $\widetilde q_n(z)=z^n\overline{q_n(1/\bar z)}$.

Denote by $b_n$ the Blaschke product $q_n/\widetilde q_n$. It is easy to check that $b_n(z)\overline{b_n(1/\bar z)}\equiv1$ by algebraic properties of Blaschke products. Thus, \eqref{eq:proof5} yields that

\begin{equation}
\label{eq:4step}
(f-r_n)(z) = \overline{b^2_n(1/\bar z)}(l_{n,m}A_{n,m}/Q)(z), \quad z\in\overline\Om.
\end{equation}
where $l_{n,m}$ is the polynomial of degree $m$ such that $q_n=q_{n,m}l_{n,m}$. Choose $\epsilon>0$ so small that $\K\subset\D_{1-\epsilon}$ (see Theorem \ref{thm:minset}). As $l_{n,m}$ is an arbitrary divisor of $q_n$ of degree $m$, we can choose it to have zeros only in $\D_{1-\epsilon}$ for  all $n$ large enough (this is possible since in full proportion the zeros of $q_n$ approach $\K$). Then it holds that
\begin{equation}
\label{eq:1step}
\lim_{n\to\infty}|l_{n,m}/Q|^{1/2n} = 1
\end{equation}
uniformly on $\overline\Om$.
Further, by \eqref{eq:Convergence1} and the last claim of Lemma~\ref{lem:pt}, we have that
\begin{equation}
\label{eq:5step}
|f-r_n|^{1/2n} \cic \exp\left\{-\frac{1}{\cp(\K,\T)}\right\} \quad \mbox{on} \quad \T.
\end{equation}
As any Blaschke product is unimodular on the unit circle, we deduce from \eqref{eq:4step}--\eqref{eq:5step} with the help of \eqref{eq:proof6} ({\it i.e.,} $A_{n,m}$ goes to a constant) that
\[
|A_{n,m}|^{1/2n} \cic \exp\left\{-\frac{1}{\cp(\K,\T)}\right\} \quad \mbox{in} \quad \overline\C\setminus\K.
\]
Then we get from Lemma~\ref{lem:cic} that 
\begin{equation}
\label{eq:2step}
\limsup_{n\to\infty} |A_{nm}|^{1/2n} \leq \exp\left\{-\frac{1}{\cp(\K,\T)}\right\} 
\end{equation}
uniformly on closed subsets of $\overline\C\setminus\K$, in particular, uniformly on $\overline\Om$. Set $q_{n,\epsilon}$ for the monic polynomial whose zeros are those of $q_n$ lying in $\D_{1-\epsilon}$. Put $n_\epsilon:=\deg(q_{n,\epsilon})$, $\widetilde q_{n,\epsilon}(z)=z^{n_\epsilon}\overline{q_{n,\epsilon}(1/\bar z)}$,  and let $\nu_{n,\epsilon}$ be the normalized counting measure of the zeros of $q_{n,\epsilon}$. As $\nu_n\cws\ged$, it is easy to see that $n_{\epsilon}/n\to 1$ and that $\nu_{n,\epsilon}\cws\ged$ when $n\to+\infty$. Thus, by the principle of descent (Section~\ref{sss:wsccic}), it holds  that
\begin{equation}
\label{estqeps}
\limsup_{n\to\infty}|q_{n,\epsilon}|^{1/n} =
\limsup_{n\to\infty}|q_{n,\epsilon}|^{1/{n_\epsilon}} 
\leq \exp\left\{-V^{\ged}\right\},
\end{equation}
locally uniformly in $\C$. In another connection, since $\log|1-z\bar{u}|$ is 
continuous for $(z,u)\in \D_{1/(1-\epsilon)}\times\D_{1-\epsilon}$, it follows
easily from the weak$^*$ convergence of $\nu_{n,\epsilon}$ that
\begin{equation}
\label{estqteps}
\lim_{n\to\infty}|\widetilde q_{n,\epsilon}(z)|^{1/n} =
\lim_{n\to\infty}|\widetilde q_{n,\epsilon}(z)|^{1/{n_\epsilon}} 
=\exp\left\{\int\log|1-z\bar{u}|d\ged(u)\right\},
\end{equation}
uniformly in $\overline{\D}$. Put $b_{n,\epsilon}:=q_{n,\epsilon}/\widetilde{q}_{n.\epsilon}$. Since the Green function of $\D$ with pole at $u$ is given by $\log|(1-z\bar u)/(z-u)|$, we deduce from \eqref{estqeps}, \eqref{estqteps}, and a simple majorization that
\[
\limsup_{n\to\infty}|b_n|^{1/n} \leq \limsup_{n\to\infty}|b_{n,\epsilon}|^{1/n} \leq \exp\left\{-V_\D^{\ged}\right\}
\]
uniformly in $\overline\D$. Besides, the Green function of $\Om$ is still given by $\log|(1-z\bar u)/(z-u)|$, hence $V_\D^\omega(1/\bar z)=V_\Om^{\omega^*}(z)$, $z\in\Om$, where $\omega$ is any measure supported in $\D$. Thus, we derive that
\begin{equation}
\label{eq:3step}
\limsup_{n\to\infty}|b_n^2(1/\bar z)|^{1/2n} \leq \exp\left\{-V_\Om^{\ged^*}(z)\right\}
\end{equation}
holds uniformly on $\overline\Om$. Combining \eqref{eq:4step}--\eqref{eq:3step}, we deduce that
\[
\limsup_{n\to\infty}|f-r_n|^{1/2n} \leq \exp\left\{-\frac{1}{\cp(\K,\T)}-V_\Om^{\ged^*}\right\}
\]
uniformly on $\overline\Om$. This finishes the proof of the theorem since $V_{\overline\C\setminus\K}^{\ged^*}=\frac{1}{\cp(\K,\T)}+V_\Om^{\ged^*}$ in $\overline\Om$ by Lemma~\ref{lem:pt}, the maximum principle for harmonic functions applied in $\Om$, and the fact that the difference of two Green potentials of the same measure but on different domains is harmonic in a neighborhood of the support of that measure by the first claim in Section~\ref{ss:gp}.
\end{proof}

\begin{proof}[Proof of Corollary~\ref{cor:normconv}] 

It follows from \eqref{eq:Convergence2} and Lemma~\ref{lem:pt} that
\[
\limsup_{n\to\infty}\|f-r_n\|_\T^{1/2n} \leq \exp\left\{-\frac{1}{\cp(\K,\T)}\right\}.
\]
On the other hand, by \eqref{eq:Convergence1} and the very definition of convergence in capacity, we have for any $\epsilon>0$ small enough that
\[
|f-r_n|>\left(\exp\left\{-\frac{1}{\cp(\K,\T)}\right\} - \epsilon\right)^{2n} \quad \mbox{on} \quad \T\setminus S_{n,\epsilon},
\]
where $\cp(S_{n,\epsilon})\to0$ as $n\to\infty$. In particular, it means that $|S_{n,\epsilon}|\to0$ by \cite[Thm. 5.3.2(d)]{Ransford}, where $|S_{n,\epsilon}|$ is the arclength measure of $S_{n,\epsilon}$. Hence, we have that
\begin{eqnarray}
\liminf_{n\to\infty}\|f-r_n\|_2^{1/2n} &\geq& \lim_{n\to\infty}\left(\frac{|\T\setminus S_{n,\epsilon}|}{2\pi}\right)^{1/4n}\left(\exp\left\{-\frac{1}{\cp(\K,\T)}\right\}-\epsilon\right) \nonumber \\
{} &=& \exp\left\{-\frac{1}{\cp(\K,\T)}\right\}-\epsilon. \nonumber
\end{eqnarray}
As $\epsilon$ was arbitrary and since $\|f-r_n\|_2\leq2\pi\|f-r_n\|_\T$, this finishes the proof of the corollary.
\end{proof}

\begin{proof}[Proof of Theorem~\ref{cor:L2T}]
Let $\Theta$ be the conformal map of $\D$ onto $G$. Observe that $\Theta^\prime$ is a holomorphic function in $\D$ with integrable trace on $\T$ since $T$ is rectifiable \cite[Thm.~3.12]{Duren}, and that $\Theta$ extends in a continuous manner to $\T$ where it is absolutely continuous. Hence, $(f\circ\Theta)(\Theta^\prime)^{1/2}\in L^2$. Moreover, $g$ lies in $E_n^2(G)$ if and only if $(g\circ\Theta)(\Theta')^{1/2}$ lies in $H_n^2:=H^2\mpoly^{-1}_n$. Indeed, denote by $E^\infty(G)$ the space of bounded holomorphic functions in $G$ and set $E_n^\infty(G):=E^\infty(G)\mpoly^{-1}_n(G)$. It is clear that $g\in E_n^\infty(G)$ if and only if it is meromorphic in $G$ and bounded outside a compact subset thereof. This makes it obvious that $g\in E_n^\infty(G)$ if and only if $g\circ\Theta\in H_n^\infty:=H^\infty\mpoly^{-1}_n$, where $H^\infty$ is the space of bounded holomorphic functions in $\D$. It is also easy to see that $E_n^2(G)=E^2(G)E_n^\infty(G)$.  Since it is known that $g\in E^2(G)$ if and only if $(g\circ\Theta)(\Theta')^{1/2}\in H^2$ \cite[corollary to Thm.~10.1]{Duren}, the claim follows. Notice also that $g_n$ is a best approximant for $f$ from $E_n^2(G)$ if and only if $(g_n\circ\Theta)(\Theta')^{1/2}$ is a best approximant for $(f\circ\Theta)(\Theta')^{1/2}$ from $H_n^2$. This is immediate from the change of variable formula, namely,
\[
\|f-g\|_{2,T}^2=\int_{\T}|f\circ\Theta-g\circ\Theta|^2|\Theta^\prime|d\theta = \|(f\circ\Theta)(\Theta^\prime)^{1/2}-(g_n\circ\Theta)(\Theta^\prime)^{1/2}\|_2^2,
\]
where we used the fact that $|d\Theta(e^{i\theta})| =|\Theta'(e^{i\theta})| d\theta$ a.e. on $\T$ \cite[Thm.~3.11]{Duren}.

Now, let $g_n$ be a best meromorphic approximants for $f$ from $E^2_n(G)$. As $L^2=H^2\oplus\bar H^2_0$, it holds that $(g_n\circ\Theta)(\Theta^\prime)^{1/2}=g_n^++r_n$ and $(f\circ\Theta)(\Theta^\prime)^{1/2}=f^++f^-$, where $g_n^+,f^+\in H^2$ and $r_n,f^-\in\bar H^2_0$. Moreover, it can be easily checked that $r_n\in\rat_n$ and, as explained at the beginning of the proof of Theorem~\ref{thm:dmcc}, that $f^-\in\alg(\D)$. Since by Parseval's relation
\[
\|(f\circ\Theta)(\Theta^\prime)^{1/2}-(g_n\circ\Theta)(\Theta^\prime)^{1/2}\|_2^2 = \|f^+-g_n^+\|_2^2+\|f^--r_n\|_2^2,
\]
we immediately deduce that $g_n^+=f^+$ and that $r_n$ is an $\bar H^2_0$-best rational approximant for $f^-$. Moreover, by the conformal invariance of the condenser capacity (see \eqref{eq:coninv}), $\cp(\K,T)=\cp(\Theta^{-1}(\K),\T)$. It is also easy to verify that $K\in\pk_f(G)$ if and only if $\Theta^{-1}(K)\in\pk_{f^-}(\D)$. Hence, we deduce from Theorem~\ref{thm:convcap} and the remark thereafter  that
\[
|f^--r_n|^{1/2n} \cic \exp\left\{V_\D^{\omega_{(\Theta^{-1}(\K),\T)}}-\frac{1}{\cp(\Theta^{-1}(\K),\T)}\right\} \quad \mbox{in} \quad \D\setminus\Theta^{-1}(\K).
\]
The result then follows from the conformal invariance of the Green equilibrium measures, Green capacity, and Green potentials and the fact that, since $\Theta$ is locally Lipschitz-continuous in $\D$, it cannot locally increase the capacity by more than a multiplicative constant \cite[Thm. 5.3.1]{Ransford}.
\end{proof}

\begin{proof}[Proof of Theorem~\ref{thm:L2T}]
By Theorem~\hyperref[thm:S]{S} and decomposition \eqref{eq:toRemind}, the set  $\K$ of minimal condenser capacity for $f$ is a smooth cut, hence a tame compact set with tame points $\cup\gamma_j$, such that
\[
\frac{\partial}{\partial\n^+}V^{\widehat\omega_{(T,\K)}-\omega_{(T,\K)}} = \frac{\partial}{\partial\n^-} V^{\widehat\omega_{(T,\K)}-\omega_{(T,\K)}} \quad \mbox{on} \quad \bigcup \gamma_j,
\]
where $\widehat\omega_{(T,\K)}$ is the balayage of $\omega_{(T,\K)}$ onto $\K$. As $\widehat\omega_{(T,\K)}$ is the weighted equilibrium distribution on $\K$ in the field $V^{-\omega_{(T,\K)}}$ (see \eqref{eq:weqpot}), the set $\K$ possesses the S-property in the sense of \eqref{eq:sproperty}. If $f$ is holomorphic in $\overline{\C}\setminus\K$ and since it extends continuously  from both sides on each $\gamma_j$ with a jump that can vanish in at most countably many points, we get from \cite[Thm. 1$^\prime$]{GRakh87} that
\begin{equation}
\label{eq:upperrho}
\lim_{n\to\infty}\rho_{n,\infty}^{1/2n}(f,T) = \exp\left\{-\frac{1}{\cp(\K,T)}\right\}.
\end{equation}
However, Theorem~1$^\prime$ in \cite{GRakh87} is obtained as an application of Theorem~\hyperref[thm:GR]{GR}. Since the latter also holds for functions in $\alg(G)$, that is, those that are meromorphic in $\overline{\C}\setminus\K$, (see the explanation in the proof of Theorem~\ref{thm:mpa}), \eqref{eq:upperrho} is valid for these functions as well. As $\rho_{n,2}(f,T)\leq |T|\rho_{n,\infty}(f,T)$, where $|T|$ is the arclength of $T$, we get from \eqref{eq:upperrho} that
\[
\limsup_{n\to\infty}\rho_{n,2}^{1/2n}(f,T) \leq \exp\left\{-\frac{1}{\cp(\K,T)}\right\}.
\]

On the other hand, let $g_n$ be a best meromorphic approximants for $f$ from $E_n^2(G)$ as in Theorem~\ref{cor:L2T}. Using the same notation, it was shown that $((f-g_n)\circ\Theta)(\Theta^\prime)^{1/2}=(f^--r_n)$, where $r_n$ is a best $\bar H_0^2$-rational approximant for $f^-$ from $\rat_n$. Hence, we deduce from the chain of equalities
\[
\|f-g_n\|_{2,T} = \|(f\circ\Theta)(\Theta^\prime)^{1/2}-(g_n\circ\Theta)(\Theta^\prime)^{1/2}\|_2 = \|f^--r_n\|_2
\] 
and Corollary~\ref{cor:normconv} that
\[
\lim_{n\to\infty}\|f-g_n\|_{2,T}^{1/2n} = \exp\left\{-\frac{1}{\cp(\Theta^{-1}(\K),\T)}\right\} = \exp\left\{-\frac{1}{\cp(\K,T)}\right\}.
\]
As $\rho_{n,2}(f,T)\geq\|f-g_n\|_{2,T}$ by the very definition of $g_n$ and the inclusion $\rat_n(G)\subset E^2_n(G)$, the lower bound for the limit inferior of $\rho_{n,2}^{1/2n}(f,T)$ follows.
\end{proof}

\section{Some Potential Theory}
\label{sec:pt}

Below we give a brief account of logarithmic potential theory that was used extensively throughout the paper. We refer the reader  to the monographs \cite{Ransford,SaffTotik} for a thorough treatment.

\subsection{Capacities}

In this section we introduce, logarithmic, weighted, and condenser capacities.

\subsubsection{Logarithmic Capacity}
\label{sss:lc} 
The {\it logarithmic potential} of a finite positive measure $\omega$, compactly supported in $\C$, is defined by
\[
V^\omega(z):=-\int\log|z-u|d\omega(u), \quad z\in\C.
\]
The function $V^\omega$ is superharmonic with values in $(-\infty,+\infty]$ and is not identically $+\infty$. The {\it logarithmic energy} of $\omega$ is defined by
\[
I[\omega]:=\int V^\omega(z)d\omega(z)=-\iint\log|z-u|d\omega(u)d\omega(z).
\]
As $V^\omega$ is bounded below on $\supp(\omega)$, it follows that $I[\omega]\in(-\infty,+\infty]$.

Let $F\subset \C$ be compact and $\Lm(F)$ denote the set of all probability measures supported on $F$. If the logarithmic energy of every measure in $\Lm(F)$ is infinite, we say that $F$ is {\it polar}. Otherwise, there exists a unique $\omega_F\in\Lm(F)$ that minimizes the logarithmic energy over all measures in $\Lm(F)$. This measure is called the {\it equilibrium distribution} on $F$ and it is known that $\omega_F$ is supported on the outer boundary of $F$, i.e., the boundary of the unbounded component of the complement of $F$. Hence, if $K$ and $F$ are two compact sets with identical outer boundaries, then $\omega_K=\omega_F$.

The {\it logarithmic capacity}, or simply the capacity, of $F$ is defined as
\[
\cp(F)=\exp\{-I[\omega_F]\}.
\]
By definition, the capacity of an arbitrary subset of $\C$ is the {\it supremum} of the capacities of its compact subsets. We agree that the capacity of a polar set is zero. It follows readily from what precedes that the capacity of a compact set is equal to the capacity of its outer boundary.

We say that a property holds \emph{quasi everywhere (q.e.)} if it holds everywhere except on a set of zero capacity. We also say that a sequence of functions $\{h_n\}$ converges {\it in capacity} to a function $h$, $h_n\cic h$, on a compact set $K$ if for any $\epsilon>0$ it holds that
\[
\lim_{n\to\infty} \cp\left(\left\{z\in K:~ |h_n(z)-h(z)| \geq \epsilon\right\}\right) = 0.
\]
Moreover, we say that the sequence $\{h_n\}$ converges in capacity to $h$ in a domain $D$ if it converges in capacity on each compact subset of $D$. In the case of an unbounded domain, $h_n\cic h$ around infinity if $h_n(1/\cdot)\cic h(1/\cdot)$ around the origin.

When the support of $\omega$ is unbounded, it is easier to consider $V_*^\omega$, the spherical logarithmic potential of $\omega$, i.e.,
\begin{equation}
\label{eq:sphpot}
V_*^\omega(z) = \int k(z,u)d\omega(u), \quad k(z,u) = -\left\{
\begin{array}{ll}
\log|z-u|, & \mbox{if} ~ |u|\leq1, \smallskip \\
\log|1-z/u|, & \mbox{if} ~ |u|>1.
\end{array}
\right.
\end{equation}
The advantages of dealing with the spherical logarithmic potential shall become apparent later in this section.

\subsubsection{Weighted Capacity} 
\label{sss:wc}

Let $F$ be a non-polar compact set and $\psi$ be a lower semi-continuous function on $F$ such that $\psi<\infty$ on a non-polar subset of $F$. For any measure $\omega\in\Lm(F)$, we define the weighted energy\footnote{Logarithmic energy with an external field is called weighted as it turns out to be an important object in the study of weighted polynomial approximation \cite[Ch. VI]{SaffTotik}.} of $\omega$ by
\[
I_\psi[\omega] := I[\omega] + 2\int\psi d\omega.
\]
Then there exists a unique measure $\omega_{F,\psi}$, the \emph{weighted equilibrium distribution} on $F$, that minimizes $I_\psi[\omega]$ among all measures in $\Lm(F)$ \cite[Thm. I.1.3]{SaffTotik}. Clearly, $\omega_{F,\psi}=\omega_F$ when $\psi\equiv0$.

The measure $\omega_{F,\psi}$ admits the following characterization \cite[Thm. I.3.3]{SaffTotik}. Let $\omega$ be a positive Borel measure with compact support and finite energy such that $V^\omega+\psi$ is constant q.e. on $\supp(\omega)$ and at least as large as this constant q.e. on $F$. Then $\omega=\omega_{F,\psi}$. The value of $V^\omega+\psi$ q.e. on $\supp(\omega_{F,\psi})$ is called the \emph{modified Robin constant} and it can be expressed as
\begin{equation}
\label{eq:robinconst}
c(\psi;F) =  I_\psi[\omega_{F,\psi}]-\int\psi d\omega_{F,\psi} = I[\omega_{F,\psi}]+\int\psi d\omega_{F,\psi}.
\end{equation}
The \emph{weighted capacity} of $F$ is defined as $\displaystyle \cp_\psi(F) = \exp\left\{-I_\psi[\omega_{F,\psi}]\right\}$.

\subsubsection{Condenser Capacity}
\label{sss:cc}

Let now $D$ be a domain with non-polar boundary and $g_D(\cdot,u)$ be the Green function for $D$ with pole at $u\in D$. That is, the unique function such that
\begin{itemize}
\item [(i)]   $g_D(z,u)$ is a positive harmonic function in $D\setminus\{u\}$, which is bounded outside each neighborhood of $u$; \smallskip
\item [(ii)]  $\displaystyle g_D(z,u) + \left\{\begin{array}{ll}
-\log|z|,   & \mbox{if}\quad u=\infty, \\ 
\log|z-u|, & \mbox{if}\quad u\neq\infty, 
\end{array}\right.$ is bounded near $u$; \smallskip
\item [(iii)] $\displaystyle \lim_{z\to\xi, \; z\in D}g_D(z,u)=0$ for quasi every $\xi\in\partial D$.
\end{itemize}
For definiteness, we set $g_D(z,u)=0$ for any $z\in\overline\C\setminus\overline D$, $u\in D$. Thus, $g_D(z,u)$ is defined throughout the whole extended complex plane.

It is known that $g_D(z,u)=g_D(u,z)$, $z,u\in D$, and that the subset of $\partial D$ for which (iii) holds does not depend on $u$. Points of continuity of $g_D(\cdot,u)$ on $\partial D$ are called {\it regular}, other points on $\partial D$ are called irregular; the latter form $F_\sigma$ polar set (in particular, it is totally disconnected). When $F$ is compact and non-polar, we define regular points of $F$ as points of continuity of $g_D(\cdot,\infty)$, where $D$ is the unbounded component of the complement of $F$. In particular, all the inner points of $F$ are regular, i.e., the irregular points of $F$ are contained in the outer boundary of $F$, that is, $\partial D$. We call $F$ \emph{regular} if all the point of $F$ are regular.

It is useful to notice that for a compact non-polar set $F$ the uniqueness of the Green function implies that
\begin{equation}
\label{eq:EqPotGreenF}
g_{\overline\C\setminus F} (z,\infty) \equiv -\log\cp(F) - V^{\omega_F}(z), \quad z\in\overline\C\setminus F,
\end{equation}
by property (ii) in the definition of the Green function and the characterization of the equilibrium potential (see explanation before \eqref{eq:robinconst}).

In analogy to the logarithmic case, one can define the {\it Green potential} and the {\it Green energy} of a positive measure $\omega$ supported in a domain $D$ as
\[
V_D^\omega(z):=\int g_D(z,u)d\omega(u) \quad \mbox{and} \quad I_D[\omega] := \iint g_D(z,w)d\omega(z)d\omega(w).
\]
Exactly as in the logarithmic case, if $E$ is a non-polar compact subset of $D$, there exists a unique measure $\omega_{(E,\partial D)}\in\Lm(E)$ that minimizes the Green energy among all measures in $\Lm(E)$. This measure is called the {\it Green equilibrium distribution} on $E$ relative to $D$. The \emph{condenser capacity of $E$} relative to $D$ is defined as
\[
\cp(E,\partial D) := 1/I_D[\omega_{(E,\partial D)}].
\]
It is known that the Green potential of the Green equilibrium distribution satisfies
\begin{equation}
\label{eq:GreenEqual}
V_D^{\omega_{(E,\partial D)}}(z) = \frac{1}{\cp(E,\partial D)}, \quad \mbox{for q.e.} \quad z\in E.
\end{equation}
Moreover, the equality in (\ref{eq:GreenEqual}) holds at all the regular points of $E$. Furthermore, it is known that $\omega_{(E,\partial D)}$ is supported on the outer boundary of $E$. That is, 
\begin{equation}
\label{eq:outerboundary}
\omega_{(E,\partial D)}=\omega_{(\partial\Omega,\partial D)},
\end{equation}
where $\Omega$ is the unbounded component of the complement of $E$.

Let $F$ be a non-polar compact set, $D$ any component of the complement of $F$, and $E$ a non-polar subset of $D$. Then we define $\omega_{(E,F)}$ and $\cp(E,F)$ as $\omega_{(E,\partial D)}$ and $\cp(E,\partial D)$, respectively. It is known that
\begin{equation}
\label{eq:exchange}
\cp(E,F)=\cp(F,E),
\end{equation}
where $F$ and $E$ are two disjoint compact sets with connected complements. That is, the condenser capacity is symmetric with respect to its entries and only the outer boundary of a compact plays a role in calculating the condenser capacity.

As in the logarithmic case, the Green equilibrium measure can be characterized by the properties of its potential. Namely, if $\omega$ has finite Green energy, $\supp(\omega)\subseteq E$, $V_D^\omega$ is constant q.e. on $\supp(\omega)$ and is at least as large as this constant q.e. on $E$, then $\omega=\omega_{(E,\partial D)}$ \cite[Thm. II.5.12]{SaffTotik}. Using this characterization and the conformal invariance of the Green function, one can see that the condenser capacity is also conformally invariant. In other words, it holds that
\begin{equation}
\label{eq:coninv}
\cp(E,\partial D) = \cp(\phi(E),\partial\phi(D)),
\end{equation}
where $\phi$ is a conformal map of $D$ onto its image.

\subsection{Balayage}
\label{ss:balayage}

In this section we introduce the notion of balayage of a measure and describe some of its properties.

\subsubsection{Harmonic Measure}
Let $D$ be a domain  with compact boundary $\partial D$ of positive capacity and $\{\omega_z\}_{z\in D}$, be the harmonic measure for $D$. That is, $\{\omega_z\}_{z\in D}$ is the collection of probability Borel measures on $\partial D$ such that for any bounded Borel function $f$ on $\partial D$ the function
\[
P_Df(z) := \int fd\omega_z, \quad z\in D,
\]
is harmonic \cite[Thm. 4.3.3]{Ransford} and $\lim_{z\to x}P_Df(z) = f(x)$ for any regular point $x\in\partial D$ at which $f$ is continuous \cite[Thm. 4.1.5]{Ransford}. 

The generalized minimum principle \cite[Thm. I.2.4]{SaffTotik} says that if $u$ is superharmonic, bounded below, and $\liminf_{z\to x,z\in D}u(z)\geq m$ for q.e. $x\in\partial D$, then $u>m$ in $D$ unless $u$ is a constant. This immediately, implies that
\begin{equation}
\label{eq:equalharmonic}
P_Dh=h
\end{equation}
for any $h$ which is bounded and harmonic in $D$ and extends continuously to q.e. point of $\partial D$.

For $z\in\C$ and $z\neq w\in D\setminus\{\infty\}$, set
\begin{equation}
\label{eq:funhd}
h_D(z,w) := \left\{
\begin{array}{l}
\log|z-w|+g_D(z,w), \quad  \mbox{if} ~ D ~ \mbox{is bounded}, \smallskip \\
\log|z-w|+g_D(z,w)-g_D(z,\infty)-g_D(w,\infty), \quad \mbox{otherwise}.
\end{array}
\right.
\end{equation}
Observe that by the properties of Green function $h_D(z,\cdot)$ is harmonic at $z$. Moreover, it can be computed using \eqref{eq:EqPotGreenF} that $\lim_{|zw|\to\infty}h_D(z,w)=\log\cp(\partial D)$ when $D$ is unbounded. Therefore, $h_D(z,w)$ is defined for all $w\in D$ and $z\in\C\cup D$.   Moreover, for each $w\in D$, the function $h_D(\cdot,w)$ is bounded and harmonic in $D$ and extends continuously to every regular point of $\partial D$. It is also easy to see that $h_D(z,w)=h_D(w,z)$ for $z,w\in D$. Hence, we deduce from \eqref{eq:equalharmonic} that
\begin{equation}
\label{eq:harmoniclog}
h_D(z,w) = \left\{
\begin{array}{l}
P_D(\log|z-\cdot|)(w), \quad  \mbox{if} ~ D ~ \mbox{is bounded}, \smallskip \\
P_D(\log|z-\cdot|-g(z,\infty))(w), \quad \mbox{otherwise},
\end{array}
\right.
\end{equation}
$z\in(\C\cup D)\setminus\partial D$, for $w\in D$ and all regular $w\in\partial D$.

\subsubsection{Balayage}
\label{dubalai}
Let $\nu$ be a finite Borel measure supported in $D$. The \emph{balayage} of $\nu$, denoted by $\widehat\nu$, is a Borel measure on $\partial D$ defined by
\begin{equation}
\label{eq:balayage}
\widehat\nu(B) := \int\omega_t(B)d\nu(t)
\end{equation}
for any Borel set $B\subset\partial D$. Since $\omega_z(\partial D)=1$, the total mass of $\widehat\nu$ is equal to the total mass of $\nu$. Moreover, it follows immediately from \eqref{eq:balayage} that $\widehat\delta_z=\omega_z$, $z\in D$. In particular, if $D$ is unbounded, $\widehat\delta_\infty=\omega_\infty=\omega_{\partial D}$ (for the last equality see \cite[Thm. 4.3.14]{Ransford}). In other words, $\widehat\delta_\infty$ is the logarithmic equilibrium distribution on $\partial D$.

It is a straightforward consequence of \eqref{eq:balayage} that
\begin{equation}
\label{eq:sweep}
\int fd\widehat\nu = \int P_Dfd\nu
\end{equation}
for any bounded Borel function on $\partial D$. Thus, we can conclude from \eqref{eq:equalharmonic} and \eqref{eq:sweep} that
\begin{equation}
\label{balayageh}
\int hd\widehat\nu=\int hd\nu
\end{equation}
for any function $h$ which is bounded and harmonic in $D$ and extends continuously to q.e. point of $\partial D$. 

Assume now that $x\in\partial D$ is a regular point and $W$ an open neighborhood of $x$ in $\partial D$. Let further $f\geq0$ be a continuous function on $\partial D$ which is supported in $W$ and such that $f(x)>0$. Since $P_Df(z)\to f(x)$ when $D\ni z\to x$, we see from (\ref{eq:sweep}) that $\hat{\nu}(W)>0$. In particular, $\partial D\setminus\supp(\hat{\nu})$ is polar.

Let $D^\prime$ be a domain with non-polar compact boundary such that $\overline D\subset D^\prime$ and let $\{\omega_z^\prime\}_{z\in D^\prime}$ be the harmonic measure for $D^\prime$. For any Borel set $B\subset\partial D^\prime$ it holds that $\omega_z^\prime(B)$ is a harmonic function in $D$ with continuous boundary values on $\partial D$. Thus,
\[
\int\omega_z^\prime(B)d\widehat\nu(z) = \int\omega_z^\prime(B)d\nu(z)
\]
by \eqref{balayageh}. This immediately implies that
\begin{equation}
\label{eq:balinsteps}
\widetilde\nu=\widetilde{\widehat\nu},
\end{equation}
where $\widetilde\nu$ is the balayage of $\nu$ onto $\partial D^\prime$. In other words, balayage can be done step by step.

\subsubsection{Balayage and Potentials}
\label{BP}
It readily follows from \eqref{eq:funhd}, \eqref{eq:harmoniclog}, and \eqref{balayageh} that
\begin{equation}
\label{eq:inthd}
\int h_D(z,w)d\nu(w) = \left\{
\begin{array}{l}
-V^{\widehat\nu}(z), \quad  \mbox{if} ~ D ~ \mbox{is bounded}, \smallskip \\
-V^{\widehat\nu}(z)-g_D(z,\infty), \quad \mbox{otherwise},
\end{array}
\right. \quad z\in (\C\cup D)\setminus\partial D.
\end{equation}
Clearly, the left-hand side of \eqref{eq:inthd} extends continuously to q.e. $z\in\partial D$. Thus, the same is true for the right-hand side. In particular, this means that $V^{\widehat\nu}$ is bounded on $\partial D$ and continuous q.e. on $\partial D$. Hence, $\widehat\nu$ has finite energy.

In the case when $\nu$ is compactly supported in $D$, formula \eqref{eq:inthd} has even more useful consequences. Namely, it holds that
\begin{equation}
\label{eq:toRemind}
V_D^\nu(z) = V^{\nu-\widehat\nu}(z) + c(\nu;D), \quad z\in\overline\C,
\end{equation}
where $c(\nu;D)=\int g_D(z,\infty)d\nu(z)$ if $D$ is unbounded and $c(\nu;D)=0$ otherwise, and where we used a continuity argument to extend \eqref{eq:toRemind} to every $z\in\overline\C$. This, in turn, yields that
\begin{equation}
\label{eq:equalBal}
V^{\widehat\nu}(z) = V^\nu(z)+c(\nu;D) \quad \mbox{for q.e.} \quad z\in\C\setminus D,
\end{equation}
where  equality holds for all $z\in\C\setminus\overline D$ and also at all regular points of $\partial D$. Moreover, employing the characterization of weighted equilibrium measures, we obtain from \eqref{eq:equalBal} that
\begin{equation}
\label{eq:weqpot}
\widehat\nu = \omega_{\partial D,-V^\nu} \quad \mbox{and} \quad c(-V^\nu;\partial D) = c(\nu;D).
\end{equation}

If a measure $\nu$ is not compactly supported, the logarithmic potential of $\nu$ may not be defined. However, representations similar to \eqref{eq:toRemind}--\eqref{eq:weqpot} can be obtained using the spherical logarithmic potentials. Indeed, it follows from \eqref{eq:inthd} that
\begin{eqnarray}
(V_*^\nu-V^{\widehat\nu} - V_D^\nu)(z) &=& \int \left[k(z,u)+\log|z-u|-g_D(u,\infty)\right]d\nu(u) \nonumber \\
{} &=&  \int_{|u|>1}\left[\log|u|-g_D(u,\infty)\right]d\nu(u)-\int_{|u|\leq1}g_D(u,\infty)d\nu(u). \nonumber
\end{eqnarray}
As the right-hand side of the chain of the equalities above is a finite constant and $V_D^\nu$ vanishes quasi everywhere on $\partial D$, we deduce as in \eqref{eq:toRemind}--\eqref{eq:weqpot} that this constant is $-c(-V_*^\nu;\partial D)$ and that
\begin{equation}
\label{eq:weqspot}
\widehat\nu = \omega_{\partial D,-V_*^\nu}.
\end{equation}
Moreover, it holds that
\begin{equation}
\label{eq:toRemind1}
V_D^\nu(z) = V_*^\nu(z) - V^{\widehat\nu}(z) + c(-V_*^\nu;\partial D), \quad z\in\overline\C.
\end{equation}

Let now $D$ be a bounded domain and $K$ be a compact non-polar subset of $D$. If $E\subseteq K$ is also non-polar and compact, then
\begin{equation}
\label{eq:energies}
|I_D[\omega_E]-I[\omega_E]| \leq \max_{z\in K,u\in\partial D}|\log|z-u|| 
\end{equation}
by integrating both sides of \eqref{eq:toRemind} against $\omega_E$ with $\nu=\omega_E$. This, in particular, yields that
\begin{equation}
\label{eq:capgreencap}
\left|\frac{1}{\cp(E,\partial D)}+\log\cp(E)\right| \leq \max_{z\in K,u\in\partial D}|\log|z-u||.
\end{equation}

\subsubsection{Weighted Capacity in the Field $-U^\nu$}
\label{sss:wcf}
Let $\nu$ be a probability Borel measure supported in $\overline\D$, $K\subset\D_r$, $r<1$, be a compact non-polar set, and $D$ be the unbounded component of the complement of $K$. Further, let $U^\nu(z)=-\int\log|1-z\bar u|d\nu(u)$ as defined in \eqref{eq:unu}. It is immediate to see that $U^\nu=V_*^{\nu^*}$, where, as usual, $\nu^*$ is the reflection of $\nu$ across $\T$. In particular, it follows from \eqref{eq:weqspot}, \eqref{eq:toRemind1}, and the characterization of the weighted equilibrium distribution that
\begin{equation}
\label{eq:whnustar}
\widehat\nu^* = \omega_{K,-U^\nu},
\end{equation}
where $\widehat\nu^*$ is the balayage of $\nu^*$ onto $\partial D$ relative to $D$. Thus, $\omega_{K,-U^\nu}$ is supported on the outer boundary of $K$ and remains the same for all sets whose outer boundaries coincide up to a polar set. In another connection, it holds that
\[
U^\nu(z) = -\int\log|1-z/u|d\nu^*(u) = -\int\log|1-z/u|d\widetilde\nu^*(u) = V^{\widetilde\nu^*}(z) - V^{\widetilde\nu^*}(0)
\]
for any $z\in\D_r$ by \eqref{balayageh} and harmonicity of $\log|1-z/u|$ as a function of $u\in\D_r^*$, where $\widetilde\nu^*$ is the balayage of $\nu^*$ onto $\T_r$. It is also true that $\widehat\nu^* = \widehat{\widetilde{\nu^*}}$ by \eqref{eq:balinsteps}. Thus,
\[
I_\nu[K] = I[\widehat\nu^*] - 2\int V^{\widetilde\nu^*}d\widehat\nu^* + 2V^{\widetilde\nu^*}(0),
\]
where $I_\nu[K]$ was defined\footnote{In \eqref{eq:wcap} we slightly changed the notation comparing to Section~\ref{sss:wc}. Clearly, $I_\nu[\cdot]$ and $\cp_\nu(\cdot)$ should be $I_{-U^\nu}[\cdot]$ and $\cp_{-U^\nu}(\cdot)$. Even though this change is slightly ambiguous, it greatly alleviates the notation throughout the paper.} in \eqref{eq:wcap}. Using the harmonicity of  $V^{\widehat\nu^*}+g_D(\cdot,\infty)$ in $D$ and continuity at regular points of $\partial D$, \eqref{balayageh}, the Fubini-Tonelli theorem, and \eqref{eq:toRemind}, we obtain that
\begin{eqnarray}
I_\nu[K] &=&  \int\left(V^{\widehat\nu^*}(z) + g_D(z,\infty)\right)d\widehat\nu^*(z)  - 2\int V^{\widehat\nu^*}d\widetilde\nu^* + 2V^{\widetilde\nu^*}(0) \nonumber \\
{} &=& \int\left(V_D^{\widetilde\nu^*}(z) - V^{\widetilde\nu^*}(z) - c(\widetilde\nu^*;D) + g_D(z,\infty)\right)d\widetilde\nu^*(z) + 2V^{\widetilde\nu^*}(0) \nonumber \\
\label{eq:weightgreen}
{} &=& I_D[\widetilde\nu^*] - I[\widetilde\nu^*] + 2V^{\widetilde\nu^*}(0).
\end{eqnarray}
Equation \eqref{eq:weightgreen}, in particular, means that the problem of maximizing  $I_\nu[\cdot]$ among the sets in $\D_r$ is equivalent to the problem of maximizing the Green energy of $\widetilde\nu^*$ among the domains with boundary in $\D_r$.

\subsubsection{Weak$^*$ Convergence and Convergence in Capacity}
\label{sss:wsccic}

By a theorem of F. Riesz, the space of complex continuous functions on $\overline{\C}$, endowed with the {\it sup} norm, has dual the space of complex measures on $\overline{\C}$ normed with the mass of the total variation
(the so-called strong topology for measures). We say that a sequence of Borel measures $\{\omega_n\}$ on $\overline{\C}$ converges \emph{weak$^*$} to a Borel measure $\omega$ if $\int fd\omega_n\to\int fd\omega$ for any complex continuous function $f$ on $\overline{\C}$. By the Banach-Alaoglu theorem, any bounded sequence of measures has a subsequence that converges in the weak$^*$ sense. Conversely, by the Banach-Steinhaus theorem, a weak$^*$ converging sequence is bounded.
 
We shall denote weak$^*$ convergence by the symbol $\cws$. Weak$^*$ convergence of measures implies some convergence properties of logarithmic and spherical logarithmic potentials, which we mention below.

The following statement is known as the \emph{Principle of Descent} \cite[Thm. I.6.8]{SaffTotik}. Let $\{\omega_n\}$ be a sequence of probability measures all having support in a fixed compact set. Suppose that $\omega_n\cws\omega$ and $z_n\to z$, $z_n,z\in\C$. Then
\[
V^\omega(z) \leq \liminf_{n\to\infty}V^{\omega_n}(z_n) \quad \mbox{and} \quad I[\omega] \leq \liminf_{n\to\infty}I[\omega_n].
\]
Weak$^*$ convergence of measures entails some convergence in capacity of their spherical potentials. This is stated rather informally in \cite[Sec. 3 and 4]{GRakh87}, but the result is slightly  subtle because, as examples  show, convergence in capacity generally occurs outside the support of the limiting measure only. A precise statement is as 
follows.

\begin{lem}
\label{lem:cwssimcic}
Let $\{\omega_n\}$ be a sequence of positive Borel measures such that $\omega_n\cws\omega$. Then $V_*^{\omega_n}\cic V_*^\omega$ in $\overline\C\setminus\supp(\omega)$. In particular, if $\omega$ is the zero measure, then the spherically normalized potentials $V_*^{\omega_n}$ converge to zero in capacity in the whole extended complex plane.
\end{lem}
\begin{proof}
Suppose first that $\omega_n$ converges weak$^*$ to the zero measure. Then the convergence is actually strong.
Assume moreover that the measures $\omega_n$ are supported on a fixed compact set $K\subset\C$. Let $G$ be a simply connected domain that contains $K$, $L$ be a Jordan curve that contains the closure of $G$ in its interior, and $D$ be a bounded simply connected domain that contains $L$. Fix $\epsilon>0$ and define $E_n:=\{z\in D:~V_D^{\omega_n}(z)>\epsilon\}$. By superharmonicity of $V_D^{\omega_n}$ the set $E_n$ is open,  and we 
can assume $E_n\subset G$ by taking $n$ large enough. If $E_n$ is empty then $\cp(E_n)=0$, otherwise
let $E\subset E_n$ be a nonpolar compact set. Then the Green equilibrium potential $V_D^{\omega_{(E,\partial D)}}$ is bounded above by $1/\cp(E,\partial D)$ \cite[Thm. 5.11]{SaffTotik} which is finite. Hence $h:=V_D^{\omega_n}-\varepsilon\cp(E,\partial D) V_D^{\omega_{(E,\partial D)}}$ is superharmonic and bounded below in in $D\setminus E$, with $\liminf h(z)\geq0$ as $z$ tends to $\partial E\cup\partial D$. By the minimum principle, we thus have 
\[
V_D^{\omega_n} \geq \epsilon\cp(E,\partial D) V_D^{\omega_{(E,\partial D)}} \quad \mbox{in} \quad D\setminus E.
\]
Set
\[
m := \min_{u\in\overline G}\min_{z\in L}g_D(z,u)>0.
\]
Clearly, $V_D^{\omega_{(E,\partial D)}}(z)>m$, $z\in L$, thus
\[
V_D^{\omega_n} \geq \epsilon m\cp(E,\partial D) \quad \mbox{on} \quad L.
\]
Hence, in view of \eqref{eq:capgreencap} applied with $K=\overline{G}$, we get
\[
-\log\cp(E_n)=-\sup_{E\subset E_n}\log\cp(E) \geq \frac{\epsilon m}{\sup_L V_D^{\omega_n}}-C
\]
where $C$ is independent of $n$. Using the uniform convergence to 0 of $V_D^{\omega_n}$ on $L$, we get that $\cp(E_n)\to0$ as $n\to\infty$, that is, $V_D^{\omega_n}\cic0$ in $D$.  Let, as usual, $\widehat\omega_n$ be the balayage of $\omega_n$ onto $\partial D$. Since $|\widehat\omega_n|=|\omega_n|\to0$ as $n\to\infty$, we have that $V^{\widehat\omega_n}\to0$ locally uniformly in $D$. Combining this fact with \eqref{eq:toRemind}, we get that $V^{\omega_n}\cic0$ in $D$. Let $u$ be an arbitrary point in $G$. Then $\{V^{\omega_n}+|\omega_n|\log|\cdot-u|\}$ is a sequence of harmonic functions in $\overline\C\setminus G$. It is easy to see that this sequence converges uniformly to 0 there. As $|\omega_n|\log|\cdot-u|\cic0$ in $\overline\C$, we deduce that $V^{\omega_n}\cic0$ in the whole extended complex plane and so does $V_*^{\omega_n}=V^{\omega_n}+\int\log^+|u|d\omega_n$ 
({\it cf.} \eqref{eq:sphpot}) since $\supp(\omega_n)\subset K$.

Next, let $\{\omega_n\}$ be an arbitrary sequence of positive measures that converges to the zero measure. As the restriction ${\omega_n}_{|{\overline{\D}}}$ converges to zero, we may assume by the first part of the proof that $\supp(\omega_n)\subset \overline{\Om}$. It can be easily seen from the definition of the spherical potential \eqref{eq:sphpot} that
\begin{equation}
\label{refsper}
V_*^{\omega_n}(1/z) = V_*^{\tilde\omega_n}(z)+|\omega_n|\log|z|, \quad z\in\C\setminus\{0\},
\end{equation}
where $\tilde\omega_n$ is the reciprocal measure of $\omega_n$, i.e., $\tilde\omega_n(B)=\omega_n(\{z:1/z\in B\})$ for any Borel set $B$. 
Clearly $\tilde\omega_n\to0$ and $\supp(\tilde\omega_n)\subset\overline{\D}$,
thus from the first part of the proof we get $V_*^{\tilde\omega_n}\cic0$.
Since $|\omega_n|\to0$, we also see by inspection that 
$|\omega_n|\log|z|\cic0$.
Therefore, by \eqref{refsper}, we obtain that $V_*^{\omega_n}(1/z)\cic0$
which is equivalent to $V_*^{\omega_n}\cic0$.

Let now $\{\omega_n\}$ be a sequence of positive measures converging weak$^*$ to some Borel measure $\omega\neq0$. If $\supp(\omega)=\overline{\C}$, there is nothing to prove. Otherwise, to each $\varepsilon>0$, we set 
$F_\varepsilon:=\{z\in\overline{\C}:~d_c(z,\supp(\omega))\geq\varepsilon\}$ where $d_c$ is the chordal distance on the Riemann sphere. Pick a continuous function $f$, with $0\leq f\leq1$, which is identically $1$ on $F_\varepsilon$ and supported in $F_{\varepsilon/2}$.By the positivity of $\omega_n$ and its weak$^*$ convergence to $\omega$,
we get
\[
0\leq \lim_{n\to+\infty}\omega_n(F_\varepsilon) \leq \lim_{n\to+\infty}\int f\,d\omega_n=\int f \,d\omega =0.
\]
From this, it follows easily that if $\varepsilon_n\to0$ slowly enough, then the restriction $\omega_n^1:={\omega_n}_{|F_{\varepsilon_n}}$ converges strongly to the zero measure. Therefore $V_*^{\omega^1_n}\cic0$ in $\overline\C$ by the previous part of  the proof. Now, put $\omega^2_n:=\omega_n-\omega_n^1={\omega_n}_{|\overline{\C}\setminus F_{\varepsilon_n}}$. For fixed $z\in \C\setminus\supp(\omega)$, the function $k(z,u)$ from \eqref{eq:sphpot} is continuous on a neighborhood of $\overline{\C}\setminus F_{\varepsilon_n}$ for all $n$ large enough. Redefining $k(z,u)$ near $z$ to make it continuous does not change its integral against $\omega$ 
nor $\omega_n^2$, therefore $V_*^\omega(z)-V_*^{\omega_n^2}(z)\to0$ as $n\to+\infty$ since $\omega^2_n\cws\omega$. Moreover, it is straightforward to check from the boundedness of $|\omega^2_n|$ that the convergence is locally uniform with respect to $z\in\C$. Finally, if $\supp(\omega)$ is bounded, we observe that when $z\to\infty$
\[
V_*^\omega(z)-V_*^{\omega_n^2}(z)\sim\log|z|\left(\omega_n^2(\overline{\C}) - \omega(\overline{\C})\right)+\int\log^+|u|\,d\omega - \int \log^+|u| \,d\omega_n^2
\]
which goes to zero in capacity since $\omega_n^2(\overline{\C})\to \omega(\overline{\C})$ and $\log^+$ is continuous in a neighborhood of both $\supp(\omega)$ and $\supp(\omega^2_n)$ for $n$ large enough. This finishes the proof of the lemma.
\end{proof}

The following lemma is needed for the proof of Theorem~\ref{thm:convcap}.

\begin{lem}
\label{lem:cic}
Let $D$ be a domain in $\overline\C$ and $\{A_n\}$ be a sequence of holomorphic functions in $D$ such that $|A_n|^{1/n}\cic c$ in $D$ as $n\to\infty$ for some constant $c$. Then $\limsup_{n\to\infty}|A_n|^{1/n} \leq c$ uniformly on closed subsets of $D$.
\end{lem}
\begin{proof}
By the maximum principle, it is enough to consider only compact subsets of $D$ and therefore it is sufficient to consider closed disks. Let $z\in D$ and $x>0$ be such that the closure of $D_{3x}:=\{w:|w-z|<3x\}$ is contained in $D$. We shall show that $\limsup_{n\to\infty}\|A_n\|_{\overline D_x}^{1/n}\leq c$.

Fix $\epsilon>0$. As $|A_n|^{1/n}\cic c$ on $\overline D_{3x}\setminus D_{2x}$, there exists $y_n\in(2x,3x)$ such that $|A_n|^{1/n}\leq c+\epsilon$ on $L_n:=\{w:|w-z|=y_n\}$ for all $n$ large enough. Indeed, define $S_n:=\{w\in \overline D_{3x}\setminus D_{2x}:||A_n(w)|^{1/n}-c|>\epsilon\}$. By the definition of convergence in capacity, we have that $\cp(S_n)\to0$ as $n\to\infty$. Further, define $S_n^\prime:=\{|w-z|:w\in S_n\}\subset[2x,3x]$. Since the mapping $w\mapsto|w-z|$ is contractive, $\cp(S_n^\prime)\leq\cp(S_n)$ by \cite[Thm. 5.3.1]{Ransford} and therefore $\cp(S_n^\prime)\to0$ as $n\to\infty$. The latter {\it a fortiori} implies that $|S_n^\prime|\to0$ as $n\to\infty$ by \cite[Thm. 5.3.2(c)]{Ransford}, where $|S_n^\prime|$ is the Lebesgue measure of $S_n^\prime$. Thus, $y_n$ with the claimed properties always exists for all $n$ large enough. Using the Cauchy integral formula, we get that
\[
\|A_n\|_{\overline D_x}^{1/n} \leq \left(\frac{\max_{L_n}|A_n|}{2\pi x}\right)^{1/n} \leq \frac{c+\epsilon}{\sqrt[n]{2\pi x}}.
\]
As $x$ is fixed and $\epsilon$ is arbitrary, the claim of the lemma follows.
\end{proof}

\subsection{Green Potentials}
\label{ss:gp}
In this section, we prove some facts about Green potentials that we used throughout the paper. We start from the following useful fact.

Let $D_1$ and $D_2$ be two domains with non-polar boundary and $\omega$ be a Borel measure supported in $D_1\cap D_2$. \emph{Then $V_{D_1}^\omega-V_{D_2}^\omega$ is harmonic in $D_1\cap D_2$.}

Clearly, this claim should be shown only on the support of $\omega$. Using the conformal invariance of Green potentials, it is only necessary to consider measures with compact support. Denote by $\widehat\omega$ and $\widetilde\omega$ the balayages of $\omega$ onto $\partial D_1$ and $\partial D_2$, respectively. Since $V_{D_1}^\omega=V^{\omega-\widehat\omega}+c(\omega;D_1)$ and $V_{D_2}^\omega=V^{\omega-\widetilde\omega}+c(\omega;D_2)$ by \eqref{eq:toRemind} and $V^{\widehat\omega}$ and $V^{\widetilde\omega}$ are harmonic on $\supp(\omega)$, it follows that $V_{D_1}^\omega-V_{D_2}^\omega=V^{\widetilde\omega-\widehat\omega}+c(\omega;D_1)-c(\omega;D_2)$ is also harmonic there.

\subsubsection{Normal derivatives}
Throughout this section, $\partial/\partial\n_i$ (resp. $\partial/\partial\n_o$) will stand for the partial derivative with respect to the inner (resp. outer) normal on the corresponding curve. 

\begin{lem}
\label{lem:nder0}
Let $L$ be a $C^1$-smooth Jordan curve in a domain $D$ and $V$ be a continuous function in $D$. If $V$ is harmonic in $D\setminus L$, extends continuously to the zero function on $\partial D$ and to $C^1$-smooth functions on each side of $L$, then $V=-V_D^\sigma$, where $\sigma$ is a signed Borel measure on $L$ given by
\[
d\sigma = \frac{1}{2\pi}\left(\frac{\partial V}{\partial\n_i}+\frac{\partial V}{\partial\n_o}\right)ds
\]
and $ds$ is the arclength differential on $L$.
\end{lem}
\begin{proof}
As discussed just before this section, the distributional Laplacian of $-V_D^\sigma$ in $D$ is equal to $2\pi\sigma$. Thus, according to Weyl's Lemma and the fact that $V=V_D^\sigma\equiv0$ on $\partial D$, we only need to show that $\Delta V=2\pi\sigma$. By the very definition of the distributional Laplacian, it holds that
\begin{equation}
\label{eq:weakL1}
\iint_D \phi\Delta V dm_2 = \iint_D V\Delta\phi dm_2 = \iint_{O}V\Delta\phi dm_2 + \iint_{D\setminus\overline O} V\Delta\phi dm_2,
\end{equation}
for any infinitely smooth function $\phi$ compactly supported in $D$, where $O$ is the interior domain of $L$ and $dm_2$ is the area measure. According to Green's formula (see \eqref{eq:greenformula} further below) it holds that
\begin{equation}
\label{eq:weakL2}
\iint_{O}V\Delta\phi dm_2 = \iint_{O}\Delta V\phi dm_2 + \int_L\left(\phi\frac{\partial V}{\partial\n_i}-V\frac{\partial\phi}{\partial\n_i}\right)ds = \int_L\left(\phi\frac{\partial V}{\partial\n_i}-V\frac{\partial\phi}{\partial\n_i}\right)ds
\end{equation}
as $V$ is harmonic in $O$. Analogously, we get that
\begin{equation}
\label{eq:weakL3}
\iint_{D\setminus\overline O} V\Delta\phi dm_2 = \int_L\left(\phi\frac{\partial V}{\partial\n_o}-V\frac{\partial\phi}{\partial\n_o}\right)ds,
\end{equation}
where we also used the fact that $\phi\equiv0$ in some neighborhood of $\partial D$. Combining \eqref{eq:weakL2} and \eqref{eq:weakL3} with \eqref{eq:weakL1} and observing that $\partial\phi/\partial\n_i=-\partial\phi/\partial\n_o$ yield
\[
\iint_D \phi\Delta V dm_2 = \int_L \left(\frac{\partial V}{\partial\n_i}+\frac{\partial V}{\partial\n_o}\right) ds = 2\pi\int_L\phi d\sigma.
\]
That is, $\Delta V=2\pi\sigma$, which finishes the proof of the lemma.
\end{proof}

\begin{lem}
\label{lem:nder1}
Let $F$ be a regular compact set and $G$ a simply connected neighborhood of $F$. Let also $V$ be a continuous function in $G$ that is harmonic in $G\setminus F$ and is identically zero on $F$. If $L$ is an analytic Jordan curve in $G$ such that $V\equiv\delta>0$ on $L$, then
\[
\frac{1}{2\pi}\int_L \frac{\partial V}{\partial\n_i} ds = -\delta\cp(F\cap\Omega,L),
\]
where $\Omega$ is the inner domain of $L$.
\end{lem}
\begin{proof}
It follows immediately from the maximum principle for harmonic functions, applied in $\Omega\setminus F$, that $V = \delta\cp(F\cap\Omega,L) V_D^\omega$ in $\overline\Omega$, where $D:=\overline\C\setminus(F\cap\Omega)$ and $\omega:=\omega_{(L,F\cap\Omega)}$. Thus, it is sufficient to show that
\begin{equation}
\label{eq:nder4}
\frac{1}{2\pi}\int_L \frac{\partial V_D^\omega}{\partial\n_i} ds = -1.
\end{equation}
Observe that $V_D^\omega$ can be reflected harmonically across $L$ by the assumption on $V$ and therefore normal inner derivative of $V_D^\omega$ does exist at each point of $L$. According to \eqref{eq:toRemind}, it holds that
\begin{equation}
\label{eq:nder0}
\frac{1}{2\pi}\int_L \frac{\partial V_D^\omega}{\partial\n_i} ds = \frac{1}{2\pi}\int_L \frac{\partial V^{\omega-\widehat\omega}}{\partial\n_i} ds,
\end{equation}
where $\widehat\omega$ is the balayage of $\omega$ onto $F\cap\Omega$. By Gauss' theorem \cite[Thm. II.1.1]{SaffTotik}, it is true that
\begin{equation}
\label{eq:nder1}
\frac{1}{2\pi}\int_L \frac{\partial V^{\widehat\omega}}{\partial\n_i} ds = \widehat\omega(\Omega) = \widehat\omega(F\cap\Omega) = 1.
\end{equation}
Since $V_D^\omega\equiv1/\cp(L,F\cap\Omega)$ outside of $\Omega$ and $V^{\widehat\omega}$ is harmonic across $L$, we get from \eqref{eq:nder1} and the analog of \eqref{eq:nder0} with $\partial\n_i$ replaced by $\partial\n_0$, that
\begin{equation}
\label{eq:nder2}
\frac{1}{2\pi}\int_L \frac{\partial V^\omega}{\partial\n_o} ds = \frac{1}{2\pi}\int_L \frac{\partial V^{\widehat\omega}}{\partial\n_o} ds = -\frac{1}{2\pi}\int_L \frac{\partial V^{\widehat\omega}}{\partial\n_i} ds = -1.
\end{equation}
As $\partial V^\omega/\partial\n_i$ and $\partial V^\omega/\partial\n_o$ are smooth on $L$ by \eqref{eq:toRemind}, in particular, Lipschitz smooth, we obtain from \cite[Thm. II.1.5]{SaffTotik} that
\[
d\omega = -\frac{1}{2\pi}\left(\frac{\partial V^\omega}{\partial\n_i}+\frac{\partial V^\omega}{\partial\n_o}\right)ds
\]
and therefore
\begin{equation}
\label{eq:nder3}
\frac{1}{2\pi}\int_L \frac{\partial V^\omega}{\partial\n_i} ds = -\omega(L)-\frac{1}{2\pi}\int_L \frac{\partial V^\omega}{\partial\n_o} ds = 0
\end{equation}
by \eqref{eq:nder2}. Finally, by plugging \eqref{eq:nder1} and \eqref{eq:nder3} into \eqref{eq:nder0}, we see the validity of \eqref{eq:nder4}. Hence, the lemma follows.
\end{proof}

\subsubsection{Reflected sets}
\label{sss:rs}
In the course of the proof of Theorem \ref{thm:convcap}, we used the conclusions of Lemma \ref{lem:pt} below. It has to do with the specific geometry of the disk, and we could not find an appropriate reference for it in the literature.

\begin{lem}
\label{lem:pt}
Let $E\subset\D$ be a compact set of positive capacity with connected complement $D$, and $E^*$ stand for the reflection of $E$ across $\T$. Further, let $\omega\in\Lm(E)$ be such that $\omega=\widehat\omega^*$, where $\omega^*$ is the reflection of $\omega$ across $\T$ and $\widehat\omega^*$ is the balayage of $\omega^*$ onto $E$. Then  $\omega=\omega_{(E,\T)}$ and $\widetilde\omega=\omega_{(\T,E)}$, where $\widetilde\omega$ is the balayage of $\omega$ onto $\T$ relative to $\D$. Moreover, it holds that $V_D^{\omega^*}=V_D^{\widetilde\omega}=1/\cp(E,\T)-V_\D^\omega$ in $\overline\D$.
\end{lem}
\begin{proof} Denote by $\widetilde\omega$ and $\widetilde\omega^*$ the balayage of $\omega$ onto $\T$ relative to $\D$ and the balayage of $\omega^*$ onto $\T$ relative to $\Om$. It holds that $\widetilde\omega=\widetilde\omega^*$. Indeed, since $g_{\Om}(z,\infty)=\log|z|$, we get from \eqref{eq:inthd} for $z\in\T$ that
\begin{eqnarray}
V^{\widetilde\omega^*}(z) &=& \int\left[\log|t|-\log|z-t|\right]d\omega^*(t) = \int\left[-\log|u|-\log|z-1/\bar u|\right]d\omega(u) \nonumber \\
{} &=& -\int\log|1-z\bar u|d\omega(u) = V^\omega(z) = V^{\widetilde\omega}(z), \nonumber
\end{eqnarray}
where we used the fact that $z=1/\bar z$ for $z\in\T$ and \eqref{eq:equalBal} applied to $\omega$. Since both measures, $\widetilde\omega$ and $\widetilde\omega^*$, have finite energy, the uniqueness theorem \cite[Thm. II.4.6]{SaffTotik} yields that $\widetilde\omega=\widetilde\omega^*$.

By \eqref{eq:toRemind}, we have that $V_D^{\widetilde\omega} =  V^{\widetilde\omega-\widehat{\widetilde\omega}} + c(\widetilde\omega;D)$. Since $\widehat{\widetilde\omega} = \widehat{\widetilde\omega^*} = \widehat\omega^* = \omega$ and by the equality $\widetilde\omega=\widetilde\omega^*$, \eqref{eq:balinsteps}, and the conditions of the lemma, it holds that
\begin{equation}
\label{eq:eqpots}
V_D^{\widetilde\omega}(z) = V^{\widetilde\omega-\omega}(z)+c(\widetilde\omega;D) = c(\widetilde\omega;D) - V_\D^\omega(z), \quad z\in\overline\C,
\end{equation}
where we used \eqref{eq:toRemind} once more. Hence, $V_\D^\omega=c(\widetilde\omega;D)$ q.e. on $E$ and the unique characterization of the Green equilibrium distribution implies that $\omega=\omega_{(E,\T)}$ and $c(\widetilde\omega;D)=1/\cp(E,\T)$. Moreover, it also holds that $V_D^{\widetilde\omega}=c(\widetilde\omega;D)=1/\cp(E,\T)$ in $\overline\Om$ and therefore $\widetilde\omega=\omega_{(\T,E)}$, again by the characterization of the Green equilibrium distribution.

The first part of the last statement of the lemma is independent of the geometry of the reflected sets and follows easily from \eqref{balayageh} and the fact that for any $z\in\D$ the function $g_D(z,u)$ is a harmonic function of $u\in\Om$ continuous on $\T$. The second part was shown in \eqref{eq:eqpots}.
\end{proof}

\subsection{Dirichlet Integrals}
\label{DI}
Let $D$ be a domain with compact boundary comprised of finitely many analytic arcs that possess tangents at the endpoints. In this section we only consider functions continuous on $\overline D$ whose weak ({\it i.e.,} distributional) Laplacian in $D$ is a signed measure supported in $D$ with total variation of finite Green energy, and whose gradient, which is smooth off the support of the Laplacian, extends continuously to $\partial D$ except perhaps at the corners where its norm grows at most like the reciprocal of the square root of the distance to the corner. These can be written as a sum of a Green potential of a signed measure as above and a harmonic function whose boundary behavior has the smoothness just prescribed above. By Proposition \ref{prop:minset}, the results apply for instance to $V_{\overline\C\setminus\Gamma}^\omega$ on $\overline{\C}\setminus\Gamma$ as soon as $\omega$ has finite energy.

Let $u$ and $v$ be two such functions. We define the Dirichlet integral of $u$ and $v$ in $D$ by
\begin{equation}
\label{eq:greenformula}
\di_D(u,v) = -\frac{1}{2\pi}\iint_Du\Delta vdm_2 - \frac{1}{2\pi}\int_{\partial D}u\frac{\partial v}{\partial\n}ds,
\end{equation}
where $\Delta v$ is the weak Laplacian of $v$ and $\partial/\partial\n$ is the partial derivative with respect to the inner normal on $\partial D$. The Dirichlet integral is well-defined since the measure $|\Delta v|$ has finite Green energy and is supported in $D$ while $|u\partial v/\partial\n|$ is integrable on $\partial D$. Moreover, it holds that
\begin{equation}
\label{symmetry}
\di_D(u,v) = \di_D(v,u).
\end{equation}
Indeed, this follows from Fubini's theorem if $u$ and $v$ are both Green potentials and from Green's formula when they are both harmonic. Thus, we only need to check \eqref{symmetry} when $v$ is harmonic and $u$ is a Green potential. Clearly, then it should hold $\di_D(u,v)=0$. Let $a$ be a point in the support of $\Delta u$ and $\varepsilon>0$ be a regular value of $g_D(.,a)$ which is so small that the open set $A :=\{z\in D:~ g_D(z,a)<\varepsilon\}$ does not intersect the support of $\Delta u$. By our choice of $\varepsilon$, the boundary of  $A$ consists of $\partial D$ and a finite union of closed smooth Jordan curves. Write $v=v_1+v_2$  for some $C^\infty$-smooth functions $v_1, v_2$ such that the support of $v_2$ is included in $A$ (hence $v_2$ is identically zero in a neighborhood of $\overline{D\setminus A}$ where the closure is taken with respect to $D$) while the support of $v_1$ is compact in $D$. Such a decomposition is easily constructed using a smooth partition of unity subordinated to the open covering of $D$ consisting of $A$ and $\{z\in D:~ g_D(z,a)>\varepsilon/2\}$. By the definition of the weak Laplacian we have that
\begin{eqnarray}
\di_D(v,u) &=& -\frac{1}{2\pi}\iint_Dv_1\Delta udm_2 - \frac{1}{2\pi}\int_{\partial D}v_2\frac{\partial u}{\partial\n}ds\nonumber \\
{} &=& -\frac{1}{2\pi}\iint_Du\Delta v_1dm_2 - \frac{1}{2\pi}\int_{\partial D}v_2\frac{\partial u}{\partial\n}ds \nonumber \\
{} &=& - \frac{1}{2\pi}\iint_Du\Delta v_1dm_2 -\frac{1}{2\pi}\iint_Du\Delta v_2dm_2  = 0,  \nonumber
\end{eqnarray}
where we used Green's formula
\begin{equation}
\label{greensformula}
\iint_{A}(v_2\Delta u-u\Delta v_2)dm_2 = \int_{\partial A}\left(u\frac{\partial v_2}{\partial\n}-v_2\frac{\partial u}{\partial \n}\right)ds.
\end{equation}

Note that if $\gamma\subset D$ is an analytic arc which is closed in $D$ and $u,v$ are harmonic across $\gamma$, then 
\begin{equation}
\label{ajoutbord}
\di_{D}(u,v)=\di_{D\setminus\gamma}(u,v)
\end{equation}
because the rightmost integral in (\ref{eq:greenformula}) vanishes on $\gamma$ as the normal derivatives of $v$ from each side of $\gamma$ have opposite signs.

Observe also that if $\nu$ is a positive Borel measure supported in $D$ with finite Green's energy then $\Delta V_D^\nu=-2\pi \nu$ by Weyl's lemma (see Section~\ref{ss:gp}) and so by \eqref{eq:greenformula}
\begin{equation}
\label{eq:di1}
\di_D(V_D^\nu) := \di_D(V_D^\nu,V_D^\nu) = I_D[\nu].
\end{equation}
Finally, if $v$ is harmonic in $D$, it follows from  the divergence theorem that
\begin{equation}
\label{positivity}
\di_D(v) = \iint_{D}\|\nabla v\|^2dm_2,
\end{equation}
which is the usual definition for Dirichlet integrals. In particular, if $D^\prime\subset D$ is a subdomain with the same smoothness as $D$, and if we assume that $\supp\,\Delta v\subset D^\prime$, we get from \eqref{ajoutbord} and
\eqref{positivity} that
\begin{equation}
\label{decompositivity}
\di_D(v) = \di_{D^\prime}(v)+\di_{D\setminus \overline{D^\prime}}(v) = \di_{D^\prime}(v)+\iint_{D\setminus\overline{ D^\prime}}\|\nabla v\|^2dm_2.
\end{equation}

\section{Numerical Experiments}
\label{sec:numer}

In order to numerically construct rational approximants, we first compute the truncated Fourier series of the approximated function (resulting rational functions are polynomials in $1/z$ that converge to the initial function in the Wiener norm) and then use {\it Endymion} software (it uses the same algorithm as the previous version \emph{Hyperion} \cite{rGr}) to compute critical points of given degree $n$. The numerical procedure in Endymion is a descent algorithm followed by a quasi-Newton iteration that uses a compactification of the set $\rat_n$ whose boundary consists of $n$ copies of $\rat_{n-1}$ and $n(n-1)/2$ copies of $\rat_{n-2}$ \cite{BCarO91}. This allows to generate several initial conditions leading to a critical point. If the sampling of the boundary gets sufficiently refined, the best approximant will be attained. In practice, however, one cannot be absolutely sure the sampling was fine enough. This why we speak below of rational approximants and do not claim they are best rational approximants. They are, however, irreducible critical points, up to numerical precision.

 In the numerical experiments below we approximate functions given by
\[
f_1(z) = \frac{1}{\sqrt[4]{(z-z_1)(z-z_2)(z-z_3)(z-z_4)}} + \frac{1}{z-z_1},
\]
where $z_1=0.6+0.3i$, $z_2=-0.8+0.1i$, $z_3=-0.4+0.8i$, $z_4=0.6-0.6i$, and $z_5=-0.6-0.6i$; and
\[
f_2(z) = \frac{1}{\sqrt[3]{(z-z_1)(z-z_2)(z-z_3)}} + \frac{1}{\sqrt{(z-z_4)(z-z_5)}},
\]
where $z_1=0.6+0.5i$, $z_2=-0.1+0.2i$, $z_3=-0.2+0.7i$, $z_4=-0.4-0.4i$, and $z_5=0.1-0.6i$. We take the branch of each function such that $\lim_{z\to\infty}zf_j(z)=2$, $j=1,2$, and use first 100 Fourier coefficients for each function.

\begin{figure}[h!]
\centering
\includegraphics[scale=.275]{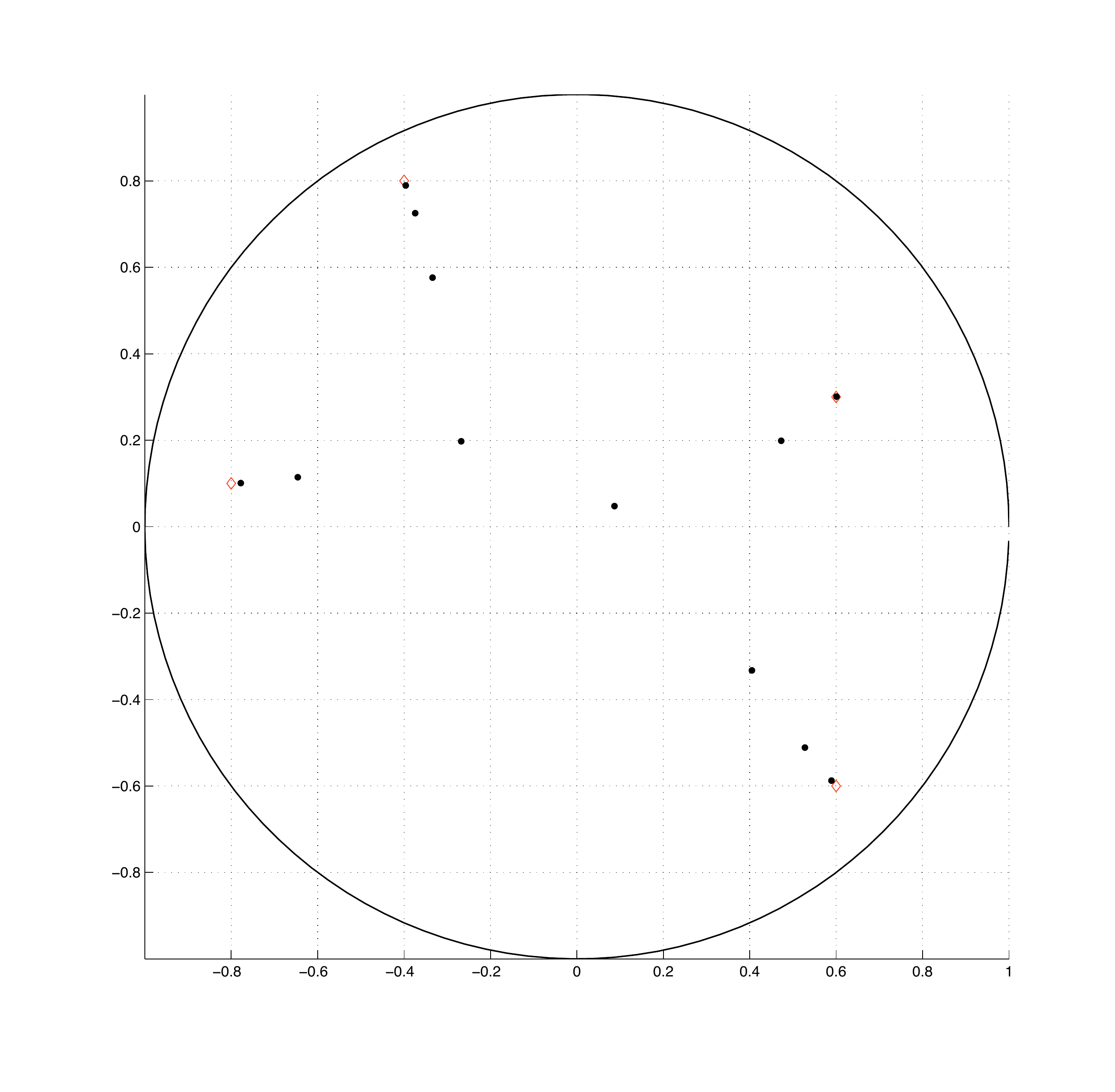}
\includegraphics[scale=.275]{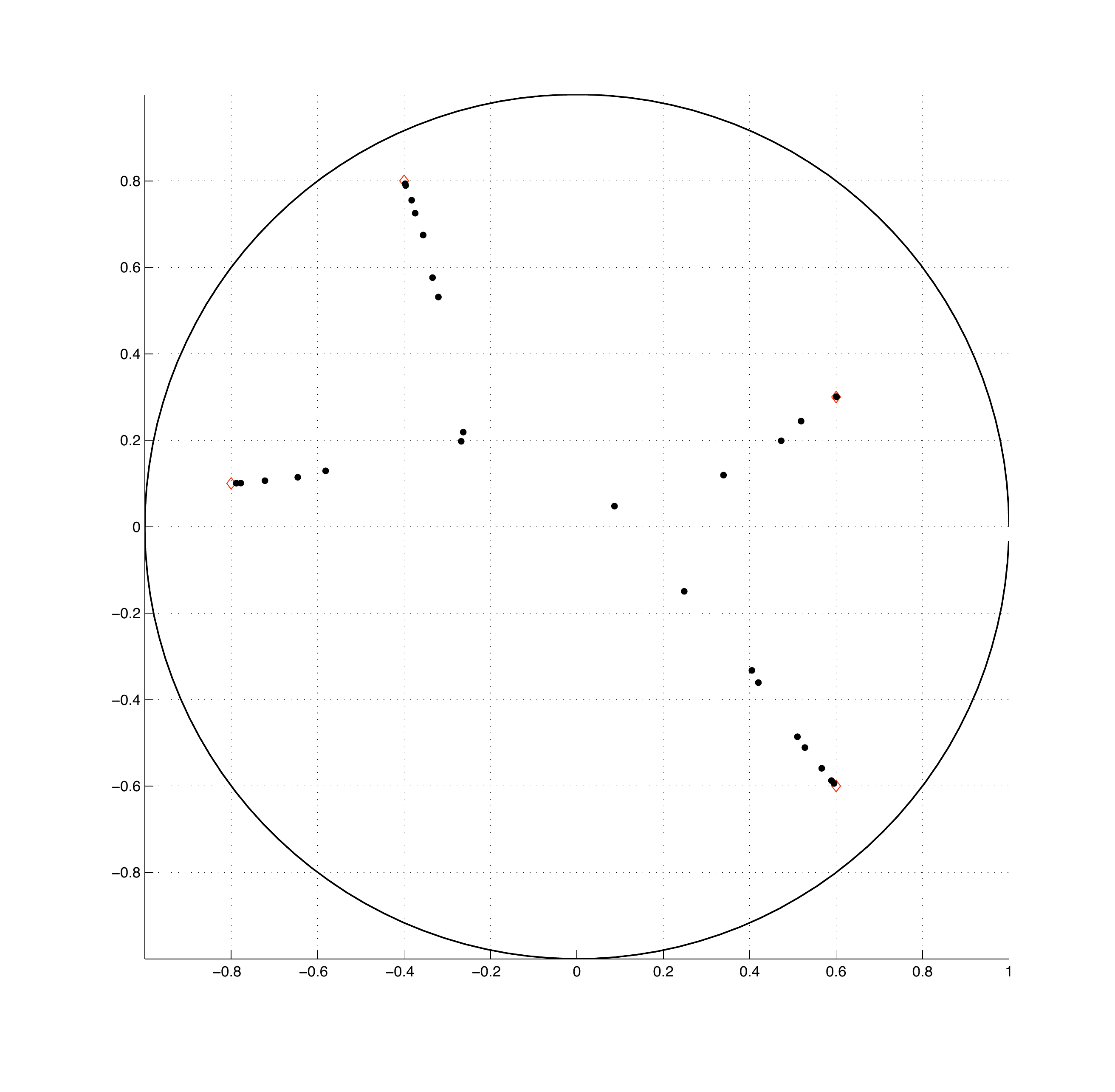}
\caption{\small The poles of rational approximant to $f_1$ of degree 12 (left) and superimposed poles of rational approximants to $f_1$ of degrees 12 and 16 (right).}
\end{figure}

\begin{figure}[h!]
\centering
\includegraphics[scale=.275]{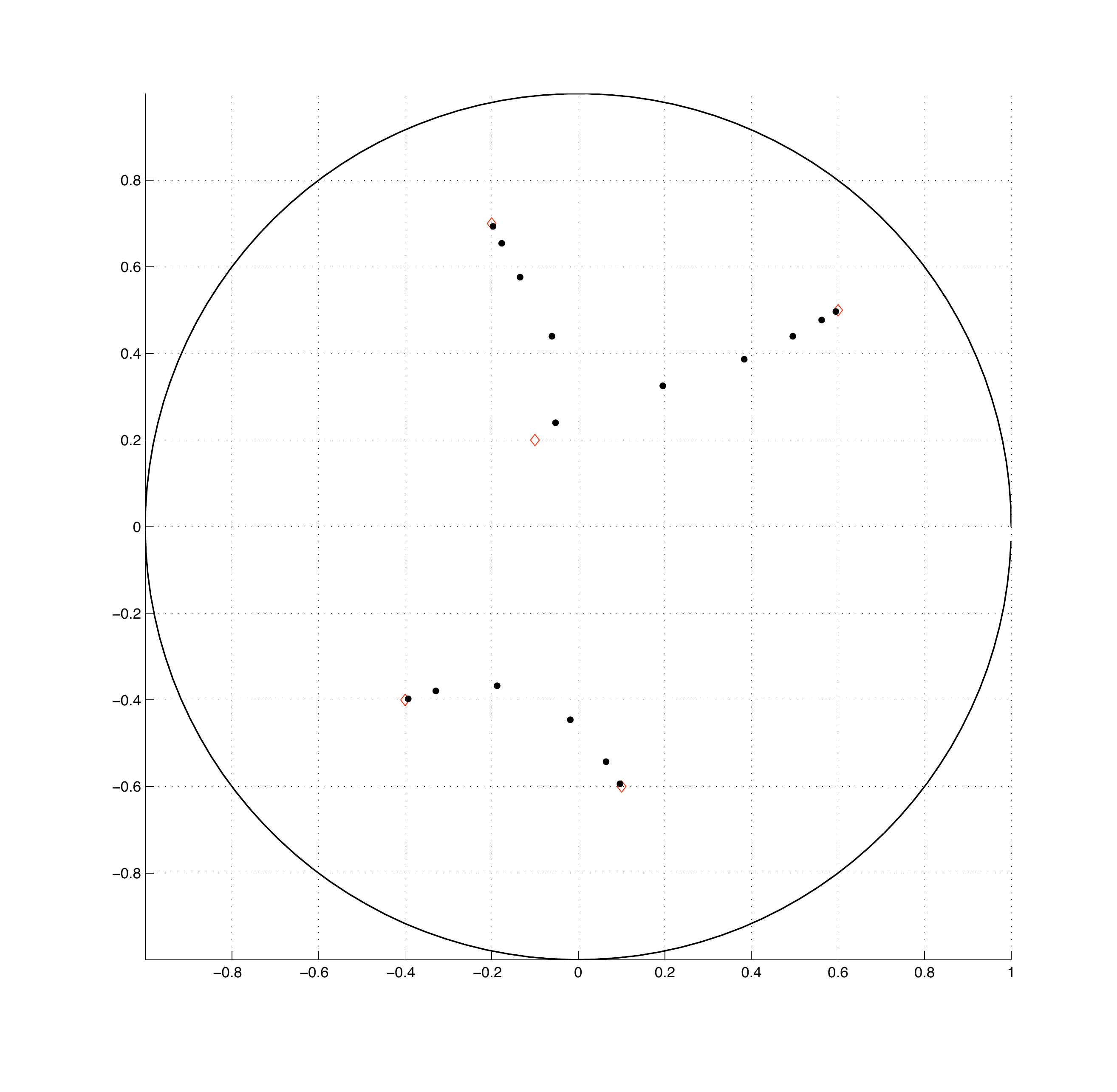}
\includegraphics[scale=.275]{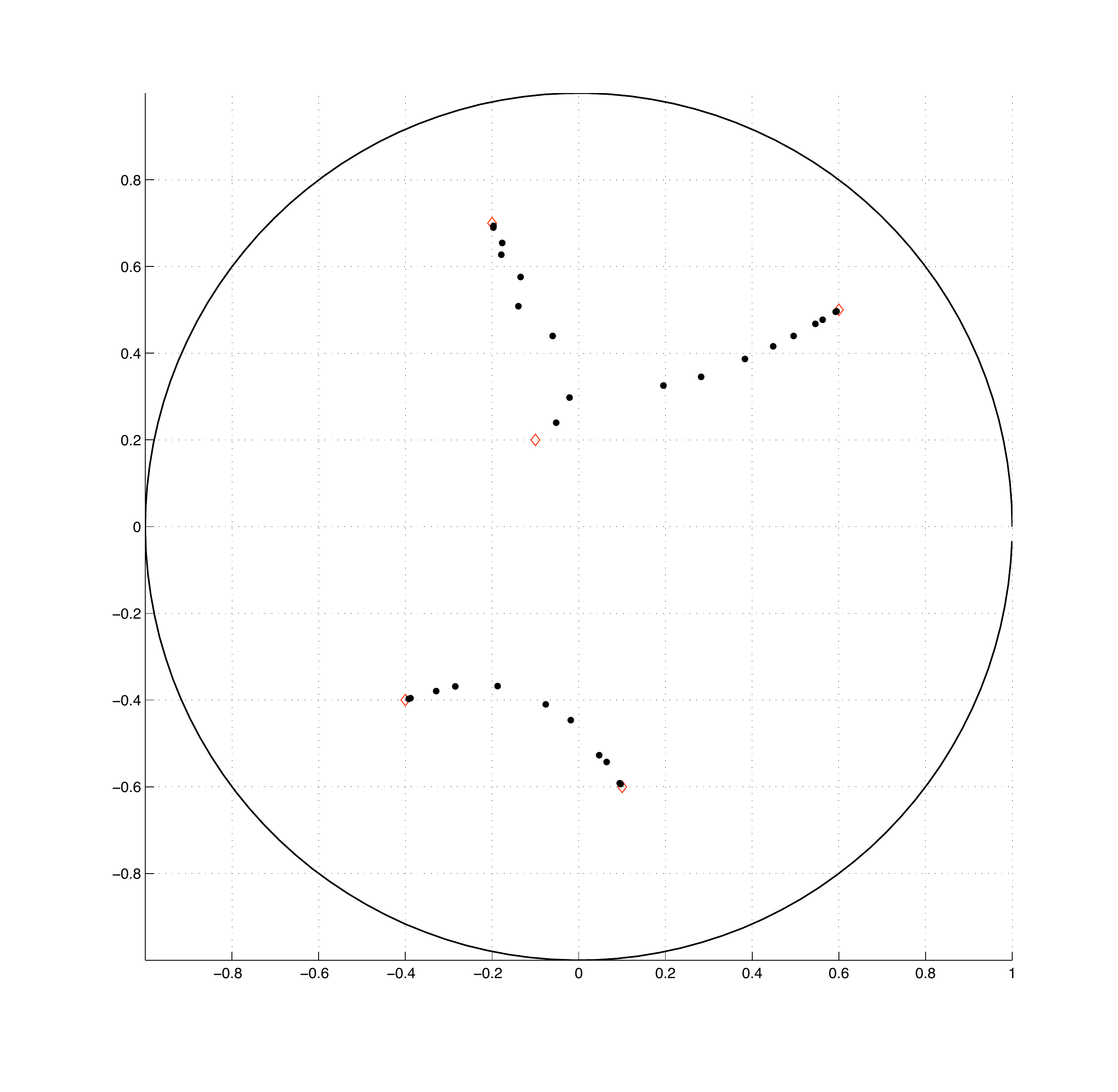}
\caption{\small The poles of rational approximant to $f_2$ of degree 16 and superimposed poles of rational approximants to $f_2$ of degrees 13 and 16 (right).}
\end{figure}

On the figures diamonds depict the branch points of $f_j$, $j=1,2$, and disks denote the poles of the corresponding approximants.

\bibliographystyle{elsarticle-num}

\end{document}